\documentclass{amsart}
\pdfoutput=1
\usepackage{amsmath,amssymb}
\usepackage{amsthm}
\usepackage{tikz-cd}

\usepackage{eucal}
\usepackage[hyphens]{url}
\usepackage{hyperref}
\usepackage{mathtools}
\usepackage[capitalise]{cleveref}
\usepackage{microtype}
\usepackage[T1]{fontenc}

\newcommand{\join}{\star}

\newcommand{\Decbot}[1]{\operatorname{Dec}_\bot{}\kern-2pt{#1}}
\newcommand{\Dectop}[1]{\operatorname{Dec}_\top{}\kern-2pt{#1}}

\newcommand{\simplexcategory}{\boldsymbol{\Delta}}
\newcommand{\Deltainert}{\simplexcategory_{\operatorname{inert}}}

\newcommand{\Decomp}{\kat{Decomp}}
\newcommand{\toppreserving}{last-point-preserving}

\newcommand{\CC}{\mathcal{C}}
\newcommand{\DD}{\mathcal{D}}
\newcommand{\EE}{\mathcal{E}}
\providecommand{\kat}[1]{\text{\textbf{\textsl{#1\/}}}}

\newcommand{\PrSh}{\kat{PrSh}}
\newcommand{\op}{^{\text{{\rm{op}}}}}
\newcommand{\N}{\mathbb{N}}
\newcommand{\Map}{\operatorname{Map}}

\newcommand{\Fun}{\operatorname{Fun}}
\newcommand{\id}{\operatorname{id}}
\newcommand{\Id}{\operatorname{Id}}
\newcommand{\sd}{\operatorname{Sd}}
\newcommand{\smallsd}{\operatorname{sd}}
\newcommand{\el}{\operatorname{el}}
\newcommand{\Nel}{\operatorname{Nel}}

\newcommand{\isleftadjointto}{\dashv}

\newcommand{\name}[1]{\ulcorner #1\urcorner}
\newcommand{\isopil}{\stackrel{\raisebox{0.1ex}[0ex][0ex]{\(\sim\)}}%
      {\raisebox{-0.15ex}[0.28ex]{\(\rightarrow\)}}}

\DeclareRobustCommand\comma{%
  \mathchoice%
    {\kern1.2pt\raise0pt\hbox{$\displaystyle\downarrow$}\kern1.2pt}% \displaystyle
    {\kern1pt\raise0pt\hbox{$\textstyle\downarrow$}\kern1pt}% \textstyle
    {\kern0.4pt\raise0pt\hbox{$\scriptstyle\downarrow$}\kern0.4pt}% \scriptstyle
    {\kern0.2pt\raise0pt\hbox{$\scriptscriptstyle\downarrow$}\kern0.2pt}% \scriptscriptstyle
}%

\newcommand{\last}{\xi}
\newcommand{\newlast}{\nu}
\newcommand{\newlen}{\pi}

\newcommand{\actto}{\rightarrow\Mapsfromchar}

\DeclareRobustCommand\upperstar{%
  \mathchoice%
    {\kern0pt\raise0.55ex\hbox{$\displaystyle *$}\kern0.8pt}% \displaystyle
    {\kern0pt\raise0.58ex\hbox{$\textstyle *$}\kern0.8pt}% \textstyle
    {\kern0pt\raise0.45ex\hbox{$\scriptstyle *$}\kern0.4pt}% \scriptstyle
    {\kern0pt\raise0.4ex\hbox{$\scriptscriptstyle *$}\kern0.2pt}% \scriptscriptstyle
}%
\DeclareRobustCommand\lowerstar{%
  \mathchoice%
    {\kern0pt\raise-0.65ex\hbox{$\displaystyle *$}\kern0.8pt}% \displaystyle
    {\kern0pt\raise-0.68ex\hbox{$\textstyle *$}\kern0.8pt}% \textstyle
    {\kern0pt\raise-0.55ex\hbox{$\scriptstyle *$}\kern0.4pt}% \scriptstyle
    {\kern0pt\raise-0.5ex\hbox{$\scriptscriptstyle *$}\kern0.2pt}% \scriptscriptstyle
}%

\newcommand{\lowershriek}{_!}

\DeclareMathOperator*{\colim}{colim}

\usepackage[utf8]{inputenc}
\usepackage[only, Mapsfromchar]{stmaryrd}
\SetSymbolFont{stmry}{bold}{U}{stmry}{m}{n}

\usepackage{tikz,tikz-cd}
\usetikzlibrary[arrows,decorations.pathmorphing,backgrounds,positioning,fit,petri]

\tikzset{
  act /.tip = >|
}

\usetikzlibrary{shapes, decorations.markings}

\newcommand{\drpullback}{\arrow[phantom]{dr}[very near start,description]{\lrcorner}}
\newcommand{\dlpullback}{\arrow[phantom]{dl}[very near start,description]{\llcorner}}

\newtheorem{lemma}{Lemma}[section]
\newtheorem{prop}[lemma]{Proposition}
\newtheorem{proposition}[lemma]{Proposition} % appendix
\newtheorem{theorem}[lemma]{Theorem}
\newtheorem{cor}[lemma]{Corollary}
\newtheorem{corollary}[lemma]{Corollary} % appendix
\newtheorem{specialtheorem}{Theorem}
 
\theoremstyle{definition}
\newtheorem{definition}[lemma]{Definition} % appendix
\newtheorem{remark}[lemma]{Remark} % appendix

\newtheoremstyle{noname}
  {8pt plus 2pt minus 4pt} % Space above
  {8pt plus 2pt minus 4pt} % Space below
  {}                       % Body font
  {}                       % Indent amount
  {\bfseries}              % Theorem head font
  {.}                      % Punctuation after theorem head
  {5pt plus 1pt minus 1pt} % Space after theorem head
  {\thmnumber{#2}\thmnote{.\ #3}} % Theorem head spec

\theoremstyle{noname}
\newtheorem{blanko}[lemma]{}

\newcommand{\RFIB}{\kat{RFib}}
\newcommand{\RFib}{\kat{Rfib}}

\newcommand{\simprfib}{\kat{RFib}}
\newcommand{\simplfib}{\kat{LFib}}
\newcommand{\nerve}{\mathrm{N}}

\newcommand{\spaces}{\mathcal{S}}
\newcommand{\simpspaces}{\kat{s}\spaces}
\newcommand{\catinf}{{\kat{Cat}_\infty}}

\newcommand{\culf}{\kat{Culf}}

\newcommand{\simplrighteous}{\kat{Righteous}}
\newcommand{\simplculfy}{\kat{Culfy}}

\newcommand{\map}{\Map}
\newcommand{\ho}{\operatorname{ho}}

\makeatletter
\let\@wraptoccontribs\wraptoccontribs
\makeatother

\begin{document}

\title{Culf maps and edgewise subdivision}

\author{Philip Hackney}
\address{Department of Mathematics, University of Louisiana at Lafayette, USA}
\email{philip@phck.net}

\author{Joachim Kock}
\address{Department of Mathematical Sciences, University of Copenhagen, Denmark and Departament de Matem\`atiques, Universitat Aut\`onoma de Barcelona, Spain and Centre de Recerca Matem\`atica (Barcelona), Spain}
\email{joachim.kock@uab.cat}

\contrib[with an appendix coauthored with]{Jan Steinebrunner}
\address{Department of Mathematical Sciences, University of Copenhagen, Denmark}
\curraddr{Gonville \& Caius College, University of Cambridge, Cambridge, UK}
\email{js2675@cam.ac.uk}

\thanks{This work was supported by a grant from the Simons Foundation (\#850849, PH).
PH was partially supported Louisiana Board of Regents through the Board of Regents Support fund LEQSF(2024-27)-RD-A-31.
This material is partially based upon work supported by the National Science Foundation under Grant No. DMS-1928930 while PH and JS participated in a program supported by the Mathematical Sciences Research Institute; 
the program was held in the summer of 2022 in partnership with the Universidad Nacional Autónoma de México.
JK has been funded by grant No.~10.46540/3103-00099B from the
Independent Research Fund Denmark, and also received support
from grants PID2020-116481GB-I00 (AEI/FEDER, UE) of Spain and
2017-SGR-1725 of Catalonia. Support from the Severo Ochoa and
Mar\'ia de Maeztu Program for Centers and Units of Excellence
(CEX2020-001084-M) as well as the Danish National Research
Foundation through the Copenhagen Centre for Geometry and
Topology (DNRF151) is also acknowledged.
}

\subjclass[2020]{%Primary: 
18N45, % Categories of fibrations, relations to K-theory, relations to type theory
18N50, % Simplicial sets, simplicial objects
18N60} % (∞,1)-categories (quasi-categories, Segal spaces, etc.); ∞-topoi, stable ∞-categories

\begin{abstract}
We show that, for any simplicial space $X$, the $\infty$-category of culf maps over $X$
  is equivalent to the $\infty$-category of right fibrations over $\sd(X)$,
  the edgewise subdivision of $X$. (When $X$ is a Rezk complete Segal or 2-Segal space, $\sd(X)$ is the twisted arrow category of $X$.)
We give two proofs of independent
  interest; one exploiting comprehensive factorization and the natural
  transformation from the edgewise subdivision to the nerve of the category
  of elements, and another exploiting a new factorization system of
  ambifinal and culf maps, together with the right adjoint to edgewise
  subdivision.
Using this main theorem, we show that the $\infty$-category of decomposition
  spaces and culf maps is locally an $\infty$-topos.
\end{abstract}

\maketitle
\vspace{\fill}
\setcounter{tocdepth}{1} 
\tableofcontents

\vspace{\fill}

%%%%%%%%%%%%%%%%%%%%%%%%%%%%%%%%%%%%%%%%%%%%%%%%%%
\section{Introduction}
%%%%%%%%%%%%%%%%%%%%%%%%%%%%%%%%%%%%%%%%%%%%%%%%%%

%%%%%%%%%%%%%%%%%%%%%%%%%%%%%%%%%%%%%%%%%%%%%%%%%%
\subsection*{Background}
%%%%%%%%%%%%%%%%%%%%%%%%%%%%%%%%%%%%%%%%%%%%%%%%%%

\begin{blanko}[Simplicial spaces]
Simplicial spaces (simplicial $\infty$-groupoids) are central
objects in homotopy theory, where they serve among other things
to express up-to-coherent-homotopy algebraic
structures~\cite{Boardman-Vogt:LNM347}, \cite{Lurie:HA}. At the
foundational level this accounts for important
models for $\infty$-categories, notably Rezk complete Segal
spaces~\cite{Rezk:MHTHT}, and it is also an important tool in
applications of homotopy theory to algebraic
geometry~\cite{Toen-Vezzosi:0207},
K-theory~\cite{Quillen:higher-K}, and representation
theory~\cite{Dyckerhoff-Kapranov:1212.3563}, to mention a few.

In this paper we are concerned with certain properties of
simplicial maps between simplicial spaces, which are of
general interest in homotopy theory. Our motivation, however,
comes from combinatorics and process algebra,
in both cases via the notion of decomposition spaces ($2$-Segal
spaces), and we now proceed to approach the results of the
paper from that angle.
\end{blanko}

\begin{blanko}[Decomposition spaces ($2$-Segal spaces)]
  Decomposition spaces~\cite{GKT1,GKT2,GKT3} (the same thing as $2$-Segal
  spaces~\cite{Dyckerhoff-Kapranov:1212.3563}; see \cite{FGKPW}) are
  simplicial $\infty$-groupoids (simplicial spaces) subject to an exactness
  condition weaker than the Segal condition. Technically the condition says 
  that certain simplicial identities are pullback squares; equivalently,  
  a simplicial space is a decomposition space
  when every slice and every coslice is a Segal space.

  Where the Segal condition expresses composition, the
  weaker condition expresses decomposition. 
  The motivation of G\'alvez--Kock--Tonks~\cite{GKT1,GKT2,GKT3} for
  introducing and studying decomposition spaces was that they have incidence
  coalgebras and M\"obius inversion. The motivation of Dyckerhoff and
  Kapranov~\cite{Dyckerhoff-Kapranov:1212.3563} came rather from
  homological algebra and representation theory. In both lines of 
  development, an
  important example of a decomposition space is Waldhausen's
  S-construction~\cite{Waldhausen:AKTS}, of an abelian category
  $\mathcal{A}$, say. Recall that $S(\mathcal{A})$ is a simplicial groupoid
  which is contractible in degree $0$, has the objects of $\mathcal{A}$ in
  degree $1$, and short exact sequences in degree $2$, etc. Wide-ranging
  generalizations of Waldhausen's construction resulted from the
  decomposition-space viewpoint 
  \cite{Bergner-Osorno-Ozornova-Rovelli-Scheimbauer:1609.02853,
  Bergner-Osorno-Ozornova-Rovelli-Scheimbauer:1809.10924,
  Bergner-Osorno-Ozornova-Rovelli-Scheimbauer:1901.03606}, culminating 
  with the discovery that every decomposition space arises from a certain 
  generalized
  Waldhausen construction, which takes as input certain double Segal 
  spaces.
\end{blanko}

\begin{blanko}[Edgewise subdivision]
  The edgewise subdivision of a simplicial space $X$, first introduced by
  Segal~\cite{Segal:1973}, is a new simplicial space $\sd(X)$ (of the same
  homotopy type) with $(\sd X)_n = X_{2n+1}$. Formally (cf.~\ref{sd}
  below), $\sd \coloneqq Q\upperstar$, for $Q\colon \simplexcategory\to
  \simplexcategory$ given by $[n] \mapsto [n]\op\join[n]=[2n{+}1]$. When
  $X$ is the nerve of a category, $\sd(X)$ is the nerve of the twisted
  arrow category. 
  A significant example of
  edgewise subdivision is the fact (due to
  Waldhausen~\cite[\S1.9]{Waldhausen:AKTS}) that the edgewise subdivision of the
  Waldhausen S-construction is the Quillen
  Q-construction~\cite{Quillen:higher-K}, in this way relating the two main
  approaches to K-theory of categories.
  
  Decomposition spaces can be characterized in terms of edgewise 
  subdivision, by a theorem of Bergner, Osorno,  Ozornova,
  Rovelli, and Scheimbauer~\cite{BOORS:Edgewise}:
  {\em $X$ is 
  decomposition if and only if $\sd(X)$ is Segal}.   
  In this paper we explore similar viewpoints, not just on simplicial
  spaces but also on simplicial maps. 
\end{blanko}

\begin{blanko}[Culf maps]
  The most important class of simplicial maps for decomposition spaces ---
  those that induce coalgebra homomorphisms --- are the {\em culf} maps
  (standing for ``conservative'' and ``unique-lifting-of-factorization''), introduced in this setting by G\'alvez--Kock--Tonks~\cite{GKT1}. The
  culf condition is weaker than being a right (or left) fibration. For
  $\infty$-categories, the culf maps are the same thing as the conservative
  exponentiable fibrations studied by Ayala and
  Francis~\cite{AyalaFrancis:Fibrations}. For $1$-categories, culf functors
  are also called discrete Conduch\'e fibrations
  \cite{Johnstone:Conduche'}.

  A technically convenient formulation of the culf condition states that
  certain squares are pullbacks (cf.~\ref{culf} below). While that
  condition will feature in all our proofs, it is useful to know 
  (cf.~\ref{culf<=>rfib})
  that a
  simplicial map $p$ is culf if and only if $\sd(p)$ is a right fibration. (For
  $1$-categories, where edgewise subdivision is just the twisted arrow
  category, this result goes back to Lamarche and
  Bunge--Niefield~\cite{Bunge-Niefield}.) % Proposition 4.7 of Bunge-Niefield
  
  Further interpretations can be given in analogy with right (or left)
  fibrations. Recall that a functor $p \colon \mathcal{E} \to \mathcal{B}$
  is a right fibration when for
  every object $x \in \mathcal{E}$, the induced functor on slices $p_x
  \colon \mathcal{E}_{/x} \to \mathcal{B}_{/px}$ is an equivalence. 
  Similarly, $p \colon \mathcal{E} \to
  \mathcal{B}$ is a left fibration if every induced map on coslices is an
  equivalence.
  The culf condition is weaker: $p$ is culf when for every $x\in
  \mathcal{E}$ the induced map on coslices is a right fibration, or 
  equivalently, the induced map on slices is a left fibration. 
\end{blanko}

\begin{blanko}[Interval preservation, and culf maps in combinatorics]
  The data over which to slice and then coslice, or coslice and then slice,
  is just a $1$-simplex $f \colon x \to y$. The slice of the coslice (or
  the coslice of the slice) is then precisely Lawvere's notion of {\em
  interval of $f$}, denoted $I(f)$. Intuitively, the interval of an arrow
  $f$ is the category of its factorizations. Yet another characterization
  of culf maps is that they are the maps that induce equivalences on all
  intervals (cf.~\ref{culf=int-pres}). This is the original viewpoint on culf
  maps of Lawvere~\cite{Lawvere:statecats}.

  The notion of interval of a $1$-simplex is central to the combinatorial
  theory of decomposition spaces~\cite{GKT3}, \cite{GKT:ex},
  \cite{Forero:2103.11508}, since it generalizes the notion of intervals in
  a poset, which form the basis for the incidence coalgebra of the poset.
  Just as the comultiplication map in classical incidence coalgebras splits
  poset intervals, the general notion of incidence coalgebra of
  decomposition spaces is about splitting decomposition-space intervals, or
  equivalently, summing over factorizations. The interpretation of the culf
  condition from the viewpoint of combinatorics is thus to preserve
  interval structure, or to preserve decomposition structure, loosely
  speaking.
\end{blanko}

\begin{blanko}[Culf maps in dynamical systems and process algebra]
  Lawvere's original motivation, both for the notion of interval and the
  notion of culf map, came from dynamical systems and the general theory of
  processes~\cite{Lawvere:statecats} (part of his long-time effort to
  understand continuum mechanics categorically). In this theory, the
  general role of culf maps is to express abstract notions of duration and
  synchronization, but depending on the situation they are also given
  interpretation in terms of ``response'' and ``control.'' The interval of
  an arrow, thought of as a process, is then the space of trajectories, or
  executions, of the process. It is important that the culf
  condition is weaker than left fibrations (discrete opfibrations) or right fibrations
  (discrete fibrations): where left or right fibrations express determinism, 
  namely unique evolution forward or backward from a given state (object) 
  (see \cite{Winskel-Nielsen:1995} for a
  development of this viewpoint in computer science), the
  culf condition only expresses synchronization of a given process, or
  control of it, by a scheduling.
  
  Brown and Yetter~\cite{Brown-Yetter:10.1216/RMJ-2017-47-3-711} 
  interpreted the culf condition as preservation of more abstract
  notions of dynamics in the theory of $C\upperstar$-algebras.
  Melli\`es~\cite{DBLP:journals/pacmpl/Mellies19} and
  Eberhart--Hirschowitz--Laouar~\cite{DBLP:conf/rta/EberhartHL19} exploit 
  similar viewpoints in game semantics.
\end{blanko}

\begin{blanko}[Lamarche conjecture]
  Working on abstract notions of processes in computer science, at a time
  when presheaf semantics was gaining importance to model concurrency (see
  for example Cattani--Winskel~\cite{DBLP:conf/csl/CattaniW96}), Lamarche
  (1996) made the conjecture that for any category $C$, the category
  $\kat{Cat}^{\text{culf}} / C$ of culf maps over $C$ is a topos.
  It was soon discovered, though, that the conjecture is false in general, by
   counterexamples due to Johnstone~\cite{Johnstone:Conduche'},
  Bunge--Niefield~\cite{Bunge-Niefield}, and
  Bunge--Fiore~\cite{Bunge-Fiore}. (For the interesting history of this
  conjecture, see \cite{Kock-Spivak:1807.06000}.)

  The categories $C$ for which $\kat{Cat}^{\text{culf}} / C$ {\em is} a topos are
  very special, expressing a certain local linear time
  evolution~\cite{Bunge-Fiore} (see Fiore~\cite{DBLP:conf/ifipTCS/Fiore00}
  for further analysis). This includes 
  the nonnegative reals, the monoid $\N$, and more
  generally free categories on a graph --- these were the examples of
  importance to Lawvere~\cite{Lawvere:statecats} for dynamical systems.
  From the viewpoint of computer science the condition expresses a
  strict interleaving property (covering models such as labeled transition
  systems and synchronization trees \cite{Winskel-Nielsen:1995}), but comes
  short in capturing more general notions of concurrency.
\end{blanko}

\begin{blanko}[Kock--Spivak theorem]
  Decomposition spaces were first considered in connection with process
  algebra when Kock and Spivak~\cite{Kock-Spivak:1807.06000} discovered
  that Lamarche's conjecture is actually true in general, if just
  categories are replaced by decomposition spaces: they showed that
  for any {\em discrete}
  decomposition space $D$ (i.e.~a simplicial set rather than a simplicial
  space), there is a natural equivalence of categories
  $$
  \Decomp_{/D} \simeq \PrSh(\sd D) .
  $$
  
  This result shows that not only are culf maps natural to consider in
  connection with decomposition spaces, but that also decomposition spaces
  are a natural setting for culf maps: even if the base $D$ is actually a
  category, the nicely behaved class of culf maps into it is from
  decomposition spaces rather than from categories. From the viewpoint of
  processes, the lack of composability is something that occurs naturally
  in applications: Schultz and Spivak~\cite{Schultz-Spivak:1710.10258}
  observe that even if time intervals compose, processes over them do not
  necessarily compose, since constraints (called ``contracts'') may not extend
  over time.
\end{blanko}

%%%%%%%%%%%%%%%%%%%%%%%%%%%%%%%%%%%%%%%%%%%%%%%%%%
\subsection*{Contributions of this paper}
%%%%%%%%%%%%%%%%%%%%%%%%%%%%%%%%%%%%%%%%%%%%%%%%%%

One version of our main theorem is the following $\infty$-version of
the Kock--Spivak result:

\setcounter{specialtheorem}{3}
\begin{specialtheorem}[\cref{thm_d_actual}]\label{thm_d}
  The $\infty$-category of decomposition spaces and culf maps is locally an
  $\infty$-topos. More precisely, for $X$ a decomposition space, we have an
  equivalence 
$$
\Decomp_{/X} \simeq \simprfib(\sd X) \simeq 
\simprfib(\widehat{\sd X})\simeq \PrSh(\widehat{\sd X}) .
$$
\end{specialtheorem}

Here $\widehat{(-)}$ denotes the Rezk completion of a Segal space.
We also will explain in \cref{prop twist rezk} that if $X$ itself is Rezk complete as a 
decomposition space, then $\sd(X)$ is Rezk complete as a Segal space,
and we can write $\Decomp_{/X} \simeq \PrSh(\sd X)$ directly. 
For example, all M\"obius decomposition spaces are Rezk complete by \cite[Corollary 8.7]{GKT2}.

The substantial part of the result is the first equivalence in the display,
which we establish as a special case of the following general theorem:

\setcounter{specialtheorem}{2}
\begin{specialtheorem}[\cref{untwisting theorem} \& \cref{thm_c_actual_2}]\label{thm_c}
For any simplicial space $X$, there is a natural 
equivalence
$$
\culf(X) \isopil \simprfib(\sd X).
$$
\end{specialtheorem}

\Cref{thm_d} follows from this since anything culf over a decomposition space is 
again a decomposition space, so for $X$ a decomposition space, we have
$\culf(X) \allowbreak \simeq \Decomp_{/X}$.

We give two proofs of \cref{thm_c}. The first uses the ideas of the
proof of the Kock--Spivak theorem in the discrete case, but develops these 
ideas into more formal and conceptual arguments (as often required when
upgrading a $1$-categorical argument to $\infty$-categories).
In particular we (prove 
and) exploit the {\em comprehensive factorization system} (final, right-fibration) 
in the $\infty$-category of
simplicial spaces, extending the one for $\infty$-categories.

We show that Waldhausen's 
% Waldhausen, p.355 
last-vertex map $\Nel (X) \to X$ from the nerve
of the $\infty$-category of elements back to a simplicial space $X$ is 
final (\cref{L-final}).
This was
shown by Lurie and Cisinski for simplicial {\em sets} by combinatorial
constructions. Here we give a conceptual high-level proof.

We then exploit the natural transformation $\lambda\colon \Nel \Rightarrow 
\sd$ first studied by Thomason~\cite{Thomason:notebook85}, and show that
it is cartesian on culf maps (\cref{lemma:lambdaculf}). 

With these preparations, we can exhibit an inverse to the displayed 
equivalence: it is given essentially by pullback along $\lambda$ (modulo some 
identifications involving $\Nel$).

The second proof is completely new, and involves the right adjoint to
edgewise subdivision. It also involves a new factorization system of
{\em ambifinal maps} and culf maps. This factorization system restricts to
the stretched-culf factorization system on the $\infty$-category of 
intervals of \cite{GKT3}, which in turn restricts to the 
active-inert factorization system on $\simplexcategory$. Indeed, the class 
of ambifinal maps is the saturation of the class of active maps between 
representables. 

The second proof of \cref{thm_c} follows from several small lemmas of independent 
interest:

First we study the $Q\lowershriek \isleftadjointto Q\upperstar$ adjunction,
and show that its unit 
is final on representables (\cref{cor counit final}) while its counit is 
ambifinal on representables (\cref{prop counit ambifinal}).

Moving on to the $Q\upperstar \isleftadjointto Q\lowerstar$ adjunction, we
show that just as $Q\upperstar$ takes culf maps to right fibrations
(\cref{culf<=>rfib}), its right adjoint $Q\lowerstar$ takes right
fibrations to culf maps (\cref{prop rke rfib to culf}).
The key properties are now
that the unit for the $Q\upperstar \isleftadjointto Q\lowerstar$ adjunction is 
cartesian on culf maps  (\cref{lem unit cartesian on culf}) and that
the counit
is cartesian on right fibrations
 (\cref{lem counit cart on rfib}).

After these preparations, the inverse to the equivalence displayed in
\cref{thm_c} is shown to be given by first applying $Q\lowerstar$ to get a culf
map, and then pullback along the unit $\eta'$ of the $Q\upperstar
\isleftadjointto Q\lowerstar$ adjunction.

Lemma~\ref{culf<=>rfib} together with the theorem of
Bergner et al.~\cite{BOORS:Edgewise} shows that edgewise subdivision is a
key aspect of decomposition spaces and culf maps. The lemmas just quoted
show that conversely, the classical notion of edgewise subdivision
inevitably leads to culf maps and ambifinal maps, which are much more
recent notions.

  \bigskip

  It should be noted that there is {\em another} convention for edgewise 
  subdivision and twisted arrow category, which 
  relates to the functor $Q'\colon \simplexcategory \to 
  \simplexcategory$ given by $[n] \mapsto [n] \star [n]\op$ (instead of $[n] 
  \mapsto [n]\op\star [n]$). That convention is 
  also widely used in the literature; see for example \cite{Lurie:HA}.
By taking opposites, we arrive at the following alternative version of
  \cref{thm_c}: \emph{ For any simplicial space $X$, there is a natural 
  equivalence $\culf(X) \isopil \simplfib(\sd' X)$.}

%%%%%%%%%%%%%%%%%%%%%%%%%%%%%%%%%%%%%%%%%%%%%%%%%%
\subsection*{Motivation and related work}
%%%%%%%%%%%%%%%%%%%%%%%%%%%%%%%%%%%%%%%%%%%%%%%%%%

\begin{blanko}[$\infty$-aspects in process algebra?]
  Viewing our \cref{thm_d} as an $\infty$-version of the Kock--Spivak theorem,
  it is natural to ask if it has any implications in process algebra. At
  the moment we don't know of any, but rather than writing it off, we
  prefer to think that the theorem is a little bit ahead of its time,
  as category theory applied to computer science is still in the
  process of upgrading to $\infty$-categories. In the light of 
  homotopy type theory~\cite{HoTT-book}, where set-based semantics is
  routinely being replaced by 
   semantics in $\infty$-groupoids, 
   this upgrade seems inevitable.

Assuming this, \cref{thm_d} does have potential for applications. In
process algebra there is usually a base to slice over, playing the role of
time, or template for evolution, and the importance of being a topos --- or
even an $\infty$-topos --- is the use of internal logic for them, as
demonstrated by Schultz--Spivak~\cite{Schultz-Spivak:1710.10258} and
Schultz--Spivak--Vasilakopoulou~\cite{Schultz-Spivak-Vasilakopoulou}. Since
every $\infty$-topos interprets homotopy type theory with the univalence
axiom (by a recent breakthrough result of
Shulman~\cite{Shulman:1904.07004}), this logic now becomes available as an
internal language to reason in any slice. The notion of temporal type
theory, introduced by Schultz and Spivak~\cite{Schultz-Spivak:1710.10258}
for the purpose of dynamical systems, is still formulated in ordinary sheaf
semantics as in $1$-toposes, but as the theory develops and constructive
concerns impose themselves, it is to be expected that identity types and
higher structures will creep in, thus necessitating $\infty$-sheaf
semantics in the setting of $\infty$-toposes.

Even without reference to homotopy type theory, simplicial methods can be
useful in process algebra and concurrency to overcome non-strict situations 
(as already occurs in combinatorics).
Recently it was shown~\cite{Kock:2005.05108} that processes of a 
Petri net rather easily assemble into a simplicial groupoid which is Segal, 
whereas it is very subtle to actually assemble them into an ordinary 
category.

\bigskip

While we do hope our theorem can find use in these contexts, our own 
motivations for it were very different:
\end{blanko}

\begin{blanko}[Free decomposition spaces]
  Our motivations for \cref{thm_d} originate in combinatorics. In fact, the
  proof of \cref{thm_d}, grew out of work on a more specific problem, whose
  solution is presented in the companion paper \cite{HK-free}, and which is
  now an application of the theorem.

  For $j\colon \Deltainert \to \simplexcategory$ the inclusion of the
  subcategory of inert maps in $\simplexcategory$, we show in
  \cite{HK-free} that {\em the simplicial space given by left Kan extension
  along $j$ is always a decomposition space}, and that {\em the left Kan
  extension of any map is always culf}. More precisely we establish

\end{blanko}

\setcounter{specialtheorem}{5}
\begin{specialtheorem}[\cite{HK-free}]\label{thm_f}
Left Kan extension along $j$ induces a canonical
  equivalence of $\infty$-categories
  \begin{equation*}
  \PrSh(\Deltainert) \simeq \Decomp_{/B\N}.
  \end{equation*}
\end{specialtheorem}

  Here $B\N$ is the nerve of the monoid of natural numbers, appearing
  here because $B\N \simeq j\lowershriek (1)$. This theorem can be derived as a
  corollary of \cref{thm_d} of the present paper, via the neat identification
  $$
  \Deltainert \simeq \sd(B\N) .
  $$
  Some more work is involved
  (in particular to identify the general equivalence
  with left Kan extension), and there is some machinery to set up. The proof
  of \cref{thm_d} and \cref{thm_c} grew out of an attempt at optimizing the original
  proof of \cref{thm_f}.

  Decomposition spaces arising from left Kan extension along $j$ are called
  {\em free}. We show in \cite{HK-free} that virtually all
  comultiplications of deconcatenation type in combinatorial coalgebras arise
  as incidence coalgebras of free decomposition spaces. In particular the
  Hopf algebra of quasisymmetric functions arises in this way, and the
  universal map it receives (as terminal object in the category of combinatorial
  coalgebras equipped with a zeta function~\cite{Aguiar-Bergeron-Sottile})
  may be given an interpretation in terms of free decomposition
  spaces.

  For the theory of free decomposition spaces,
  \Cref{thm_d} may be regarded as somewhat of an overkill,  
  but it has a second motivation coming from
  combinatorial Hopf algebras:

\begin{blanko}[Implications in conjunction with the G\'alvez--Kock--Tonks conjecture]
  \label{GKT-conj}
  \Cref{thm_d} acquires further interest in connection with the so-called
  G\'alvez--Kock--Tonks conjecture, from \cite{GKT3}. Lawvere's interval
  construction and the universal Hopf algebra of
  intervals~\cite{Lawvere-Menni} was shown to be the incidence bialgebra of
  a decomposition space $U$ of all intervals~\cite{GKT3}. It was
  conjectured that $U$ enjoys the following universal property: for any
  decomposition space $X$, we have $\Map(X,U) \simeq 1$. The mapping space
  is the space of all culf maps. This would explain in which sense
  Lawvere's Hopf algebra is universal. This is almost like saying that $U$
  is a terminal object in $\Decomp$, but size issues prevent this
  interpretation. However, the whole construction and the conjecture can be
  restricted to the case of {\em M\"obius decomposition spaces}
  \cite{GKT2}, certain decomposition spaces satisfying a finiteness
  condition ensuring that the general M\"obius inversion principle admits a
  homotopy cardinality. Most decomposition spaces from combinatorics are
  M\"obius.

  The decomposition space of M\"obius intervals $U^{\text{M\"obius}}$ is
  small, so as to constitute a genuine terminal object in
  $\Decomp^{\text{M\"obius}}$, according to the conjecture. This is where
  \cref{thm_d} comes in: if a decomposition space $X$ is itself M\"obius, then
  everything culf over $X$ is M\"obius again, so that
  $$
  \Decomp^{\text{M\"obius}} _{/X} \simeq 
  \Decomp_{/X} .
  $$
  By \cref{thm_d}, the latter slice is an $\infty$-topos, and if $X$ is taken to 
  be $U^{\text{M\"obius}}$, and we assume the conjecture is true, then
  $$
  \Decomp^{\text{M\"obius}} \simeq 
  \Decomp _{/U^{\text{M\"obius}}} \simeq 
  \PrSh(\sd(U^{\text{M\"obius}})),
  $$
  so that $\Decomp^{\text{M\"obius}}$ itself will be an $\infty$-topos!

  The current status of the conjecture is the following (see
  Forero~\cite{Forero:2103.11508} for a detailed exposition of the 
  conjecture's history and
  motivation). The work of Lawvere (suitably upgraded to
  the present context) shows that $\Map(X,U)$ is inhabited: it contains the
  interval construction $f\mapsto I(f)$ from \cite{Lawvere:statecats}.
  G\'alvez--Kock--Tonks~\cite{GKT3} proved that it is also connected: every
  culf map $X\to U$ is homotopy equivalent to $I$. The finer property of
  being contractible is the full homotopy uniqueness statement, that not
  only is every map equivalent to $I$: it is so uniquely (in a coherent
  homotopy sense). Forero~\cite{Forero:2103.11508} has proved the
  conjecture in the discrete case (where $X$ is a simplicial set). In this
  case there is a shift in categorical dimension: the universal $U$ for
  discrete decomposition spaces is not itself discrete but rather a
  simplicial groupoid. This shift in categorical dimension is
  unavoidable in the truncated situation, but goes away in
  the untruncated situation.

  The prospective of a universal decomposition space (which cannot
  exist in truncated settings) was one of the motivations for
  G\'alvez, Kock, and Tonks to develop the theory of decomposition spaces in
  the $\infty$-setting (see the introduction of \cite{GKT1}), although most
  examples in combinatorics are $0$- or $1$-truncated~\cite{GKT:ex}.
\end{blanko}

\subsection*{Acknowledgments}
We thank Nima Rasekh for help with the comprehensive factorization system 
for simplicial spaces. 
We are also very grateful to the anonymous referee, whose insightful report led to considerable improvements to the paper.

%%%%%%%%%%%%%%%%%%%%%%%%%%%%%%%%%%%%%%%%%%%%%%%%%%
\section{Comprehensive factorization}
%%%%%%%%%%%%%%%%%%%%%%%%%%%%%%%%%%%%%%%%%%%%%%%%%%

\begin{blanko}[Conventions and setting]
  In this paper we work with $\infty$-categories in a model-independent 
fashion. We assume a (large) $\infty$-category $\catinf$
of all small
$\infty$-categories, with a full sub-$\infty$-category $\spaces$ of
$\infty$-groupoids, which we call {\em spaces}. We are in particular 
concerned with simplicial spaces within the given model of 
$\infty$-categories: by definition $\simpspaces$ is the functor
$\infty$-category $\Fun(\simplexcategory\op,\spaces)$. 
There is a fully faithful nerve functor
\begin{eqnarray*}
  \nerve \colon \catinf & \longrightarrow & \simpspaces  \\
  \CC & \longmapsto & \Map( - , \CC)
\end{eqnarray*}
whose essential 
image is the subcategory of Rezk complete Segal spaces \cite{Joyal-Tierney}, which can therefore
be considered an internal model of $\infty$-categories within the given 
model.

As a specific choice, one can take $\infty$-category to mean 
quasi-category in the sense of Joyal~\cite{quadern45} (simply called 
$\infty$-categories by Lurie~\cite{HTT}). 
For our emphasis on synthetic reasoning, we recommend 
Riehl--Verity~\cite{RiehlVerity:EICT} as a 
background reference for $\infty$-categories, and also 
Ayala--Francis~\cite{AyalaFrancis:Fibrations}
for more specific results on fibrations, and
Anel--Biedermann--Finster--Joyal~\cite{ABFJ:LELI} for factorization 
systems.

All concepts in this paper are the relevant equivalence-invariant ones, which are the only versions which make sense in this context.
For instance, 
``unique'' lifts means that the space of lifts is contractible, and so on.
\end{blanko}

\begin{blanko}[Factorization systems]\label{factorization systems}
We recall some basics on factorization systems from \cite[\S3.1]{ABFJ:LELI}.
Suppose $\CC$ is an $\infty$-category.
If $i, p$ are two maps in $\CC$, we write $i \perp p$ to mean that $i$ is left orthogonal to $p$ (or $p$ is right orthogonal to $i$), that is, each commutative square 
\[
\begin{tikzcd}
\cdot \dar[swap]{i} \rar & \cdot \dar{p} \\
\cdot \ar[ur,dotted] \rar & \cdot
\end{tikzcd}
\]
has a contractible space of lifts.
If $\mathcal{A}, \mathcal{B}$ are classes of maps in $\mathcal{C}$, we write $\mathcal{A} \perp \mathcal{B}$ whenever $i \perp p$ for all $i\in \mathcal{A}$ and $p \in \mathcal{B}$, and we write $\mathcal{A}^\perp = \{p \mid \mathcal{A} \perp p \}$ and $\vphantom{\mathcal{B}}^\perp\mathcal{B} = \{i\mid i \perp \mathcal{B}\}$ for the classes of maps which are left orthogonal to $\mathcal{A}$ / right orthogonal to $\mathcal{B}$.
A \emph{factorization system} is a pair of classes of maps $(\mathcal{L}, \mathcal{R})$ of $\CC$ such that both classes span replete subcategories of the arrow $\infty$-category of $\CC$, $\mathcal{L}\perp \mathcal{R}$, and every map $f$ in $\CC$ factors as $f = pi$ with $p\in \mathcal{R}$ and $i\in \mathcal{L}$.

A number of important properties of the classes in a factorization system are given in \cite[Proposition 3.1.11]{ABFJ:LELI}, but one especially important one for us is that $\mathcal{R}$ is closed under left-cancellation: if $pq$ and $p$ are both in $\mathcal{R}$, then so is $q$.

Notice also that if $F \colon \CC \rightleftarrows \DD : G$ is an adjunction, $i$ is a map in $\CC$ and $p$ is a map in $\DD$, then $i \perp G(p)$ if and only if $F(i) \perp p$. (See \cite[Lemma 3.1.4]{ABFJ:LELI}.)
\end{blanko}

\begin{blanko}[Right fibrations of simplicial spaces]
  A simplicial map $f \colon Y \to X$ is called a {\em right fibration} when
  it is right orthogonal to all terminal-object-preserving maps
  $\ell \colon\Delta^m \to \Delta^n$; or equivalently, considered as a
  natural transformation, $f$ is cartesian on all
  \toppreserving\ maps $\ell \colon [m] \to [n]$. The
  diagram on the left expresses the right orthogonality; the diagram
  on the right expresses the equivalent cartesian condition:
  \[
  \begin{tikzcd}
  \Delta^m \ar[r] \ar[d, "\ell"'] & Y \ar[d, "f"]  \\
  \Delta^n \ar[r] \ar[ru, dotted, "\exists!"] & X
  \end{tikzcd}
  \qquad\qquad
  \begin{tikzcd}
  Y_n \drpullback \ar[d, "f_n"'] \ar[r, "\ell\upperstar"] & Y_m 
  \ar[d, "f_m"]  \\
  X_n \ar[r, "\ell\upperstar"'] & X_m
  \end{tikzcd}
  \]
\end{blanko}

The following three lemmas are exercises using pullbacks.
\begin{lemma}\label{lemma right fib char}
  A simplicial map is a right fibration if and only if it is cartesian on 
  each last-point inclusion $[0] \to [n]$.
\end{lemma}

\begin{lemma}\label{lem rfib creates segal}
  For a right fibration $Y \to X$, if $X$ is a Segal space (resp.~a 
  Rezk complete Segal space) then also $Y$ is a Segal space (resp.~a 
  Rezk complete Segal space).
\end{lemma}

\begin{lemma}\label{lem rfib between segal spaces}
  A simplicial map between Segal spaces $Y \to X$ is a right fibration 
  if and only if it is cartesian on the coface map $d^0 \colon [0] 
  \to [1]$; that is, the square
  \[
  \begin{tikzcd}
  Y_0 \ar[d] & Y_1 \ar[l, "d_0"'] \ar[d] \dlpullback  \\
  X_0 & X_1 \ar[l, "d_0"]
  \end{tikzcd}
  \]
  is a pullback.
\end{lemma}

Thus our definition of right fibration recovers the usual one for Segal
spaces from \cite{Boavida:SOGC}. There is an evident dual notion of
\emph{left fibration} of simplicial spaces using initial-object-preserving
maps between representables; restricted to Segal spaces one recovers the
notion of left fibration from \cite[2.1.1]{KazhdanVarshavski}.

\begin{blanko}[D\'ecalage]\label{dec}
  Recall that the {\em upper d\'ecalage} $\Dectop{(X)}$ of a simplicial space $X$ is
  obtained by deleting $X_0$ as well as the top face and degeneracy maps,
  and shifting all spaces one degree down: $(\Dectop{X})_k = X_{k+1}$. 
  More formally, as we shall 
  exploit, let $\simplexcategory^t$ denote
  the category of ordinals with a top element, and top-preserving monotone maps, with
  forgetful functor $u\colon \simplexcategory^t \to \simplexcategory$ and
  left adjoint $i\colon \simplexcategory \to \simplexcategory^t$. (One can
  think of a $(\simplexcategory^t)\op$-diagram as a simplicial object with
  missing top face maps.) The upper d\'ecalage comonad on $\simpspaces$ can 
  now be described as $\Dectop{} = i\upperstar \circ u\upperstar$.
  Similarly, there is a {\em lower d\'ecalage} $\Decbot{(X)}$ which deletes the bottom face and degeneracy maps.
\end{blanko}

\begin{blanko}[Slices]\label{blanko slices}
  The notion of slice makes sense for general simplicial spaces $X$ (not 
  just for Segal spaces): for $x\in 
  X_0$, the slice $X_{/x}$ is defined as the pullback
  \[
  \begin{tikzcd}
  X_{/x} \drpullback \ar[d] \ar[r] & \Dectop{(X)} \ar[d, "d_0"]  \\
  1 \ar[r, "\name{x}"'] & X_0 .
  \end{tikzcd}
  \]
  Here the simplicial spaces in the bottom row are constant, and $d_0 
  \colon \Dectop{(X)} \to X_0$ denotes the canonical augmentation sending an 
  $(n{+}1)$-simplex in $X$ to its last vertex. Note also that $\Dectop(X)$ (and 
  hence $X_{/x}$) comes 
  with a canonical splitting, given by the original top degeneracy maps,
  making it into a $(\simplexcategory^t)\op$-diagram.
  More precisely, 
  the pullback square above is $i\upperstar$ applied to the following pullback square of 
  $\simplexcategory^t$-presheaves
  \[
  \begin{tikzcd}
  (u\upperstar X)_{/x} \drpullback \ar[d] \ar[r] & u\upperstar X \ar[d]  \\
  1 \ar[r, "\name{x}"'] & X_0 ,
  \end{tikzcd}
  \]
  where the bottom row consists of constant
  $\simplexcategory^t$-presheaves and the right map is the augmentation.
  Here, $u\upperstar X$ simply deletes the top face maps of $X$. 
  (We can more generally take slices of an arbitrary $\simplexcategory^t$-presheaf, and each such presheaf is the sum over all of its slices.)
\end{blanko}

\begin{blanko}[Warning]
  The canonical projection $X_{/x} \to X$ is not in general 
  a right fibration. (It is a right fibration when $X$ is Segal, of 
  course, and it is culf (cf.~\S\ref{sec:culf}) when $X$ is a decomposition 
  space~\cite[Proposition~4.9]{GKT1}.)
\end{blanko}

\begin{lemma}\label{rfib-slice}
  A simplicial map $p\colon Y \to X$ is a right fibration if and only if 
  for every $y\in Y_0$, the induced simplicial map $Y_{/y} \to X_{/py}$ is 
  a (levelwise) equivalence.
\end{lemma}
\begin{proof}
  To say that $p$ is a right fibration means that for all $n\geq 0$ the 
  square
  \[
  \begin{tikzcd}
  Y_0 \ar[d]  & Y_{n+1} \ar[d] \ar[l, "\operatorname{last}"'] \\
  X_0 & X_{n+1} \ar[l, "\operatorname{last}"]
  \end{tikzcd}
  \]
  is a pullback (\cref{lemma right fib char}). This
  in turn is equivalent to saying that the induced map on fibers is 
  an equivalence for every $y\in Y_0$. But this map is precisely 
  $(Y_{/y})_n \to (X_{/py})_n$.
\end{proof}

\begin{blanko}[Remark]\label{rmk simp-t-op-slices}
  As a variation of the lemma, we have also that $p\colon Y \to X$ 
  is a right fibration if and only if for each $y\in Y_0$ the induced map 
  $(u\upperstar Y)_{/y} \to (u\upperstar X)_{/py}$ is a levelwise
  equivalence of $\simplexcategory^t$-presheaves. 
  We shall use this in the proof of \cref{lem: terminal objects}.
\end{blanko}

\begin{blanko}[Terminal vertex]
  A vertex $a\in A_0$ of a simplicial space $A$ is called
  a {\em terminal vertex} when the canonical projection $A_{/a}\to A$ is a levelwise equivalence.
\end{blanko}

\begin{blanko}[Final maps]
  A simplicial map is called {\em final} if it is left orthogonal to
  every right fibration. Note that every terminal-object-preserving map
  between representables $\ell \colon \Delta^m \to \Delta^n$ is final.       
\end{blanko}

\begin{lemma}\label{lem: terminal objects}
  If a simplicial map $f\colon B \to A$
  between simplicial spaces with a terminal vertex preserves 
  those terminal vertices, then it is final.
\end{lemma}

\begin{proof}
  Let $b\in B_0$ be a terminal vertex of $B$, then $a = f(b) \in A_0$ is terminal in $A$.
We have the commutative square
\[ \begin{tikzcd}
B_{/b} \rar{\simeq} \dar & B \dar{f} \\
A_{/a} \rar{\simeq} & A  .
\end{tikzcd} \]  
The left-hand map is in the image of the left adjoint $i\upperstar$ of the d\'ecalage adjunction (see~\ref{dec})
\[
  \begin{tikzcd}
  i\upperstar : \Fun((\simplexcategory^t)\op, \spaces) \ar[r, shift left]
  &  \Fun(\simplexcategory\op, \spaces) \ar[l, shift left] : u\upperstar ,
  \end{tikzcd}
\]
following \ref{blanko slices}.
It thus suffices to check there is a contractible space of lifts for each square of simplicial spaces on the left below, where $p\colon Y \to X$ is a right fibration.
    \[
  \begin{tikzcd}
  i\upperstar((u\upperstar B)_{/b}) \dar \rar& Y \ar[d, "p"]  \\
  i\upperstar((u\upperstar A)_{/a}) \ar[r] \ar[ru, dotted]& X
  \end{tikzcd}
  \qquad\qquad
  \begin{tikzcd}
  (u\upperstar B)_{/b} \dar \rar & u\upperstar(Y) \ar[d, "u\upperstar(p)"]  \\
  (u\upperstar A)_{/a} \ar[r] \ar[ru, dotted] & u\upperstar(X)
  \end{tikzcd}
  \] 
By adjunction (see \ref{factorization systems}), this is equivalent to there being a contractible space of lifts for the diagram of $\simplexcategory^t$-presheaves on the 
  right.
In the next paragraph we explain why this holds.

Since $\simplexcategory^t$ has an initial object, by a standard lemma (see,
for instance, \cite[1.15]{GKT3}) the $\infty$-category
$\PrSh(\simplexcategory^t) = \Fun((\simplexcategory^t)\op, \spaces)$ has a
factorization system where the left class consists of those natural
transformations which are equivalences at the initial object of
$\simplexcategory^t$, and where the right class consists of the cartesian
natural transformations. Now $(u\upperstar B)_{/b} \to (u\upperstar
A)_{/a}$ is in the left class, since both of these presheaves have a
contractible space as augmentation object (that is, the value at the initial object of $\simplexcategory^t$).
On the other hand, $u\upperstar(p)$ is in the
right class: since $p$ is a right fibration, its restriction $u\upperstar(p)$ to the last-point-preserving maps is cartesian. 
Thus there
is a contractible space of lifts in the square above-right, and therefore, by the adjunction property \ref{factorization systems}, also in the square above-left, as required.
\end{proof}

\begin{theorem}\label{comprehensive}
  The $\infty$-category of simplicial spaces admits a factorization system, called the {\em comprehensive 
  factorization system}, whose left class consists of the final maps, and whose right class consists of the right fibrations.
\end{theorem}

\begin{proof}
  Let $\Sigma$ be the set of last-vertex-preserving maps between representables (alternatively: the set of bottom coface maps, or the set of last-vertex inclusions $\Delta^0 \to \Delta^n$)
  and let $\overline{\Sigma}$ be the saturated class generated by $\Sigma$.
  Since the $\infty$-category of simplicial spaces is presentable, 
  it follows from 
  \cite[Proposition 3.1.18]{ABFJ:LELI} that $(\overline{\Sigma}, \Sigma^{\perp})$ is a factorization system.
  By definition, $\Sigma^{\perp}$ is the class of right fibrations.
  Since this is a factorization system, $\vphantom{R}^\perp (\Sigma^{\perp}) = \overline{\Sigma}$ by \cite[Proposition 5.2.8.11]{HTT} or \cite[Lemma 3.1.9]{ABFJ:LELI}, hence $\overline{\Sigma}$ is the class of final maps.
\end{proof}

\begin{blanko}[Remarks]
  The comprehensive factorization system for $1$-categories is
  classical, due to Street and Walters~\cite{StreetWalters}. For
  $\infty$-categories, the result appears in
  Joyal~\cite[169--173]{quadern45} and a proof can be found in
  Ayala--Francis~\cite[\S6.3]{AyalaFrancis:Fibrations}. 
  The comprehensive factorization system for simplicial
  spaces is also implicit in
  Rasekh~\cite[5.30--5.34]{Rasekh:1711.03160}, in a model-categorical
  setting, where the final maps are defined as certain contravariant
  equivalences.
\end{blanko}

\begin{blanko}[Comprehensive factorization for $\infty$-categories]
The next proposition will justify the usage of the name ``final'' for 
  these simplicial maps.
Recall that a \emph{final functor} $f \colon \CC \to \mathcal{B}$ between $\infty$-categories is a functor so that for any other functor $\mathcal{B} \to \mathcal{Z}$, the map $\colim (\CC \to \mathcal{B} \to \mathcal{Z}) \to \colim(\mathcal{B} \to \mathcal{Z})$ exists and is an equivalence if either colimit exists \cite[Definition 6.1.1]{AyalaFrancis:Fibrations}.
This is the left class in a comprehensive factorization system on $\catinf$, 
whose right class consists of the \emph{right fibrations}, those functors 
$\pi \colon \EE \to \CC$ whose space of lifts is contractible for any square
\begin{equation}\label{pi}
  \begin{tikzcd}
{\Delta^0} \rar \ar[d, "1"'] & \EE \dar{\pi} \\
{\Delta^1} \rar \ar[ur,dashed] & \CC  .
\end{tikzcd} 
\end{equation}
\end{blanko}

\begin{prop}\label{prop CFS restriction}
The comprehensive factorization system on simplicial spaces restricts to the comprehensive factorization system on $\infty$-categories.
\end{prop}

\begin{proof}
  Temporarily write $(\mathcal{L},\mathcal{R})$ for the comprehensive factorization system on
  simplicial spaces, where $\mathcal{L}$ is the class of final maps and $\mathcal{R}$ is the
  class of right fibrations. Let $\nerve \colon \catinf \to \simpspaces$
  be the inclusion of $\infty$-categories into simplicial spaces, where we
  may think of the former as the Rezk complete Segal spaces. By \cref{lem
  rfib between segal spaces} we have that $\mathcal{R} \cap \catinf$ is the usual
  class of right fibrations between $\infty$-categories, as in \eqref{pi}. Denoting the usual
  final-rightfibration factorization system on $\catinf$ by $(\mathcal{L}',\mathcal{R}')$, we have
\begin{equation}\label{eq L cap to L prime}
  \mathcal{L}\cap \catinf \subseteq \vphantom{\mathcal{R}}^\perp (\mathcal{R} \cap \catinf) = \vphantom{\mathcal{R}}^\perp \mathcal{R}' = \mathcal{L}'
\end{equation}
since $\nerve$ is fully faithful.
We wish to show that this inclusion is an equivalence.

  Suppose $f\colon \CC \to \mathcal{B}$ is a final functor between
  $\infty$-categories, that is, a morphism in $\mathcal{L}'$. Form the factorization
  of simplicial maps below left
\[ \begin{tikzcd}[column sep=small]
& X \ar[dr,"r"] \\
\nerve \CC \ar[ur,"\ell"] \ar[rr,"\nerve f"'] & & \nerve \mathcal{B}
\end{tikzcd} \qquad 
\begin{tikzcd}[column sep=small]
& \EE \ar[dr,"\tilde r"] \\
\CC \ar[ur,"\tilde \ell"] \ar[rr,"f"'] & & \mathcal{B}
\end{tikzcd}
 \]
  with $\ell \in \mathcal{L}$ and $r\in \mathcal{R}$. By \cref{lem rfib creates segal}, $X$ is
  a Rezk complete Segal space, $X\simeq \nerve \EE$ for some $\EE \in
  \catinf$, and we may regard the triangle above left as the nerve of the
  triangle above right. Since $\tilde \ell$ is a final functor by \eqref{eq
  L cap to L prime}, it follows that $\tilde r$ is in $\mathcal{L}' \cap \mathcal{R}'$, hence
  is an equivalence. Thus $r$ is also an equivalence, and we conclude that
  $\nerve f \colon \nerve \CC \to \nerve \mathcal{B}$ is in $\mathcal{L}$. Thus $\mathcal{L}' =
  \mathcal{L}\cap \catinf$.
\end{proof}

\begin{blanko}[Categories of right fibrations]\label{bl cat right fib}
  For a simplicial space $X$, denote by $\RFIB(X)$ the full subcategory
of $\simpspaces_{/X}$ spanned by the right fibrations. By the 
left cancellation 
property satisfied by right classes, the morphisms in $\RFIB(X)$
are again right fibrations, so $\RFIB(X)$ can also be described as
the slice over $X$ of the $\infty$-category whose objects are simplicial spaces and whose morphisms are 
right fibrations.
In light of \cref{lem rfib creates segal} and \cref{prop CFS restriction}, if $\CC$ is an $\infty$-category and $\RFib(\CC) \subset {\catinf}_{/\CC}$ is the usual $\infty$-category of right fibrations over $\CC$, then the nerve functor induces an equivalence $\RFib(\CC) \simeq \simprfib(\nerve \CC)$.

Closely related to the existence of the comprehensive factorization system on $\simpspaces$ is
the fact that the inclusion functor $\RFIB(X) \to \simpspaces_{/X}$
has a left adjoint (reflection): it sends an arbitrary simplicial map
$Y \to X$ to the right-fibration part of its comprehensive
factorization. The left part is the unit of the adjunction. (This
reflection is the main ingredient in the construction of the
factorization system (cf.~proof of \cite[5.5.5.7]{HTT}; see in 
particular \cite[5.5.4.15]{HTT}).)

\end{blanko}

\begin{blanko}[Base change]
  For a simplicial map $F\colon X'\to X$, there is a canonical {\em base-change} functor $F\upperstar 
\colon \simprfib(X) \to \simprfib(X')$ given by pullback along $F$, sending $p\colon Y \to X$ to 
$p'$ as in the diagram
\[
\begin{tikzcd}
Y' \drpullback \ar[d, "p'"'] \ar[r] & Y \ar[d, "p"]  \\
X' \ar[r, "F"']& X  .
\end{tikzcd}
\]
\end{blanko}

The following is a special case of \cite[Proposition 3.1.22]{ABFJ:LELI}.

\begin{blanko}[Cobase change]
  The base-change functor $F\upperstar$ has a left adjoint {\em 
  cobase-change} functor 
$F\lowershriek \colon \simprfib(X') \to \simprfib(X)$ given by first postcomposing
with $F$, then taking factorization into final map followed by right 
fibration, and finally returning the right fibration, as $q' \mapsto q$
in the diagram
\[
\begin{tikzcd}
Y' \ar[d, "q'"'] \ar[r, dotted, "\text{final}"] & Y 
\ar[d, dotted, "q \text{ r.fib.}"]  \\
X' \ar[r, "F"']& X   .
\end{tikzcd}
\]
(For $\infty$-categories, see also \cite[Remark~6.1.10]{AyalaFrancis:Fibrations}
and \cite[6.1.14]{Cisinski:HCHA}.)

The unit is given by the universal property 
of the pullback. The counit $F\lowershriek F\upperstar (p) \to p$ is 
given by the universal property of the final-rightfibration factorization
system, as exemplified in Lemma~\ref{lemma:epsilonL is equiv}.
\end{blanko}

%%%%%%%%%%%%%%%%%%%%%%%%%%%%%%%%%%%%%%%%%%%%%%%%%%
\section{Culf maps and ambifinal maps}
%%%%%%%%%%%%%%%%%%%%%%%%%%%%%%%%%%%%%%%%%%%%%%%%%%
\label{sec:culf}

\begin{blanko}[The active-inert factorization system]
  The category $\simplexcategory$ has an active-inert factorization system:
  the {\em active maps}, written $g\colon[k]\actto [n]$, are those that  
  preserve end-points,
  $g(0)=0$ and $g(k)=n$; the {\em inert maps}, written $f\colon 
  [m]\rightarrowtail [n]$, are
  those that are distance preserving,
  $f(i{+}1)=f(i)+1$ for $0\leq i\leq m-1$.  The active maps are generated by
  the codegeneracy maps and the inner coface maps; the inert maps are
  generated by the outer coface maps $d^\bot$ and $d^\top$.
  (This orthogonal factorization system is an instance of the 
  important general notion of generic-free factorization system of
  Weber~\cite{Weber:TAC18} who referred to the two classes as generic
  and free. The active-inert terminology is due to
  Lurie~\cite{Lurie:HA}.)
\end{blanko}

\begin{blanko}[Culf maps]\label{culf}
  Recall (from \cite[\S4]{GKT1}) that a simplicial map $p\colon Y \to X$ is
  {\em culf} when it is right orthogonal to every active map $\Delta^k
  \actto \Delta^n$, or equivalently, when it is cartesian on active maps.
  The picture on the left expresses the right orthogonality; the picture on
  the right expresses the equivalent cartesian condition:
  \[
  \begin{tikzcd}
  \Delta^k \ar[r] \ar[d, -act] & Y \ar[d, "p"]  \\
  \Delta^n \ar[r] \ar[ru, dotted, "\exists!"] & X
  \end{tikzcd}
  \qquad\qquad
  \begin{tikzcd}
  Y_n \drpullback \ar[d, "p_n"'] \ar[r, -act, "g\upperstar"] & Y_k 
  \ar[d, "p_m"]  \\
  X_n \ar[r, -act, "g\upperstar"'] & X_k .
  \end{tikzcd}
  \]

  Note that every left or right fibration is culf.
\end{blanko}

\begin{blanko}[Remark]
  Culf stands for ``conservative'' and ``unique lifting of factorizations'' 
  (cf.~Lawvere~\cite{Lawvere:statecats}), as the notion recovers these
  conditions in the case of strict nerves of ordinary categories.
  In this case the notion of culf functor is also
  the same as discrete Conduch\'e fibration~\cite{Johnstone:Conduche'}. In the case of 
  $\infty$-categories, the culf maps are the same 
  thing as the conservative exponentiable fibrations studied by Ayala and 
  Francis~\cite{AyalaFrancis:Fibrations} (see \cite[Lemma 2.2.22]{Barkan-Steinebrunner:2211.02576} for a proof), which in the quasi-category model would be called conservative flat fibrations \cite[B.3.8]{Lurie:HA}.
\end{blanko}

\begin{lemma}\label{lem:simp-culf}
  {\rm \cite[Lemma 4.1]{GKT1}}
  A simplicial map is culf if and only if it is cartesian on 
  each active map of the form $[1]\actto [n]$ for $n\geq 0$.
\end{lemma}

A simplicial map is culf if and only if it is cartesian on degeneracy maps and inner face maps.
The next lemma gives that only the inner face maps are necessary.

\begin{lemma}\label{ulf=culf}
  To check that a general simplicial map $F\colon Y \to X$ is culf, it is enough to
  check that it is cartesian on active maps of the form $[1]\actto [n]$ for $n\geq 1$; it is then automatically 
  cartesian also on $[1]\actto [0]$. (In other words, ulf implies culf.)
\end{lemma}

\begin{proof}
  By \cref{lem:simp-culf} it only remains to check that $F$ is cartesian 
  on the codegeneracy map $[0] \leftarrow [1]$.
  This is the leftmost (=rightmost) square in
  the commutative $(\Delta^1\times\Delta^1\times \Delta^2)$-diagram
  \[
  \begin{tikzcd}[sep={30pt,between origins}]
      {Y_0} \ar[rrr, "s_0"] \ar[dd] \ar[dr, "s_0"] & & & 
      {Y_1} \ar[rrr, "d_0"] \ar[dd] \ar[dr, "s_1"] & & & 
      {Y_0} \ar[dd] \ar[dr, "s_0"] & \\ &
      {Y_1} \ar[rrr, pos=0.4, "s_0", crossing over]   & & & 
      {Y_2} \ar[rrr, pos=0.4, "d_0", crossing over]   & & & 
      {Y_1} \ar[dd] \\ 
      {X_0} \ar[rrr, pos=0.64, "s_0"] \ar[dr, "s_0"] & & & 
      {X_1} \ar[rrr, pos=0.64, "d_0"] \ar[dr, "s_1"] & & & 
      {X_0} \ar[dr, "s_0"] & \\ &
      {X_1} \ar[rrr, "s_0"] \ar[from=uu, crossing over] & & &  
      {X_2} \ar[rrr, "d_0"] \ar[from=uu, crossing over] & & &  
      {X_1} ,
  \end{tikzcd}
  \]
  which
  exhibits the square as a retract of the middle square. The middle square 
  in turn is already 
  known to be a pullback, since $s_1$ is a section to the active face map 
  $d_1$. Since pullbacks are stable 
  under retracts, it follows that also the leftmost square is a pullback.
\end{proof}

\begin{blanko}[Remark]
  Lawvere and Menni~\cite[Lemma~4.4]{Lawvere-Menni} proved this for the
  case of $1$-categories in which all identities are indecomposable. Their
  proof works more generally for all decomposition spaces that are
  split~\cite[\S5]{GKT2}. G\'alvez, Kock, and Tonks~\cite[Proposition
  4.2]{GKT1} proved the result for general decomposition spaces, but their
  prism-lemma proof does not work for general simplicial spaces. The above
  retract argument giving the general case is directly inspired by the
  proof by Feller et al.~\cite{FGKPW} that all $2$-Segal spaces are unital.
  The result in full generality was found independently by Barkan and Steinebrunner (personal communication).
\end{blanko}

\begin{blanko}[Intervals]
  An {\em interval} is a simplicial space with an initial and a terminal
  object. To every $1$-simplex $f\colon a\to b$ in a simplicial space $X$
  there is associated an interval $I(f) \coloneqq (X_{/b})_{f/}$ (with initial
  object $s_0(f)$ and terminal object $s_1(f)$).
  This simplicial space can also be described as $\Decbot{\Dectop{(X)}}
  \times_{X_1} \{f\}$.
  Usually (see \cite{GKT3}),
  this notion is considered only when $X$ is a decomposition space, in
  which case $I(f)$ is always a Segal space. Here we consider the more
  general case only to be able to state the following result (which goes
  back to Lawvere~\cite{Lawvere:statecats} in the case where $X$ is the
  strict nerve of a $1$-category).
\end{blanko}

\begin{blanko}[Warning]
  For general simplicial spaces $X$, the canonical projection $I(f) \to X$
  (given on objects by taking a $2$-simplex with long edge $f$ to its 
  middle vertex) is not always culf. (It is culf when $X$ is a 
  decomposition space~\cite[\S3]{GKT3}.) 
\end{blanko}

\begin{lemma}\label{culf=int-pres}
  A simplicial map $p\colon Y \to X$ is culf if and only if for every 
  $f\in Y_1$, the corresponding map of intervals $I(f) \to I(pf)$ is a
  (levelwise) equivalence.
\end{lemma}

\begin{proof}
  By \cref{ulf=culf}, to say that $p$ is culf means that for all $n\geq 0$ the square
  \[
  \begin{tikzcd}
  Y_1 \ar[d]  & Y_{n+2} \ar[d] \ar[l, "\operatorname{long}"'] \\
  X_1 & X_{n+2} \ar[l, "\operatorname{long}"]
  \end{tikzcd}
  \]
  is a pullback. (The horizontal maps return the long edge of a simplex.)
  This in turn means that for each $f\colon a \to b$ in $Y_1$ the induced
  map on fibers $(Y_{n+2})_f \to (X_{n+2})_{pf}$ is an equivalence. But
  this map is precisely $((Y_{/b})_{f/})_n \to ((X_{/pb})_{pf/})_n$, which
  is the $n$-component of the map on intervals $I(f) \to I(pf)$.
\end{proof}

\begin{blanko}[Ambifinal maps]
  A simplicial map is called {\em ambifinal} if it is left orthogonal to
  every culf map.
\end{blanko}

In close analogy with Lemma~\ref{lem: terminal objects} we have the following result, which we treat only briefly as it is not necessary in what follows.
\begin{lemma}\label{lem: initial and terminal objects}
  If a simplicial map between simplicial spaces with both an initial and a
  terminal vertex preserves those initial and terminal vertices, then it is
  ambifinal.
\end{lemma}

\begin{proof}[Proof sketch]
The proof is analogous to that of \cref{lem: terminal objects}, but using instead the adjunction 
 $$
  \begin{tikzcd}
  i\upperstar : \Fun((\simplexcategory^{t,b})\op, \spaces) \ar[r, shift left]
  &  \Fun(\simplexcategory\op, \spaces) \ar[l, shift left] : u\upperstar
  \end{tikzcd}
  $$
  from \cite[\S\S2--3]{GKT3},
whose induced comonad $i\upperstar u\upperstar$ is double d\'ecalage (both upper and lower).
Here, $\simplexcategory^{t,b}$ is the category of ordinals with distinct top and bottom elements, and monotone maps which preserve these.
\end{proof}

\begin{theorem}
  The classes of ambifinal maps and culf maps form a factorization system 
  on $\simpspaces$.
\end{theorem}

\begin{proof}
  The proof is completely analogous to the proof of \cref{comprehensive},
  but with generating set $\Sigma$ of the left class now being the set of
  active maps $\Delta^k \actto \Delta^n$ between representables; then the
  saturated class is the class of ambifinal maps.
\end{proof}

\begin{blanko}[The generating set $\Sigma$]\label{remark gen set sigma}
In the proof of the preceding theorem, one can choose many possible alternate generating sets $\Sigma$ for the left class.
For instance, from \cref{lem:simp-culf} we know that we could take only the active maps $\Delta^1 \actto \Delta^n$.
Instead, one could take active maps
$\Delta^1 \actto \Delta^{2n+1}$ landing in odd-dimensional simplices. 
This is because the active map $\Delta^1 \actto\Delta^{2n}$ into a even-dimensional simplex, for $n>0$,
is a retract of $\Delta^1 \actto \Delta^{2n+1}$.
\Cref{ulf=culf} takes care of the remaining map $\Delta^n \actto \Delta^0$.
In fact, one could take the active maps $\Delta^1 \actto \Delta^{n_k}$ for any infinite collection of non-negative integers $n_k$.
\end{blanko}

%%%%%%%%%%%%%%%%%%%%%%%%%%%%%%%%%%%%%%%%%%%%%%%%%%
\section{The last-vertex map}
%%%%%%%%%%%%%%%%%%%%%%%%%%%%%%%%%%%%%%%%%%%%%%%%%%
\label{sec:lastvertex}

\begin{blanko}[Categories of elements]
Let $X\colon \CC\op\to\spaces$ be a presheaf. The {\em category of elements} of $X$
is by definition $\el(X) \coloneqq \CC \comma X$, the domain of the right
fibration corresponding to $X$ under the basic
straightening-unstraightening equivalence of $\infty$-categories $\RFib(\CC) \simeq \PrSh(\CC)$
(due to Lurie~\cite{HTT};
see \cite[Thm~3.4.6]{AyalaFrancis:Fibrations} for a model-independent
statement).  
Notice that $\el(X)$ is an $\infty$-category, but we retain the traditional and shorter terminology ``category of elements'' throughout.
\end{blanko}

For a simplicial space $X\colon \simplexcategory\op\to\spaces$, we shall be concerned with
the nerve of its category of elements, written $\Nel(X)$. It is thus
a simplicial space again. The $k$-simplices of $\Nel(X)$ are configurations
$$
\Delta^{n_0} \to \Delta^{n_1} \to \cdots \to \Delta^{n_k} \longrightarrow X .
$$

\begin{blanko}[The last-vertex map $\last\colon \Nel(X) \to 
  X$]\label{lastvertex}
  For any simplicial space $X$, the {\em last-vertex map} $$\last_X \colon \Nel(X) 
  \to X$$ is given on objects by sending an $n$-simplex $\sigma\colon \Delta^n 
  \to X$ to its last vertex
  $$
  \Delta^0 \stackrel{\text{last}}\to \Delta^n \stackrel{\sigma}\to X .
  $$
  The action of $\last$ on higher simplices is given, with 
  reference to the general combinatorial  {\em lower-segments 
  construction} below, by sending
  $$
  f \; \colon \ \Delta^{n_0} \to \cdots\to \Delta^{n_k} \to X 
  \qquad\qquad \in (\Nel X)_k
  $$
  to
  $$
  \Delta^{k} \stackrel{\beta_f}{\to} \Delta^{n_k} \to X
    \qquad\qquad \in X_k ,
  $$
  the $k$-simplex given by all the successive last vertices.
  
  Our next main task is to formally define $\last_X$ and to show that they assemble into a natural transformation $\last \colon \Nel \Rightarrow \id$.
  For this we will need several auxiliary constructions and lemmas, which will also be useful when describing an important natural transformation $\lambda$ in  \cref{sec:Sd}.
\end{blanko}

\begin{blanko}[Remark]
  The last-vertex map $\last \colon \Nel(A) \to A$
  for simplicial {\em sets} already has some history. It was used by Waldhausen~\cite[p.\ 359]{Waldhausen:AKTS} and 
  more recently by Lurie~\cite[4.2.3.14]{HTT}.
  The case of simplicial spaces is considerably subtler. 
A version of $\last$ for Segal spaces was studied by Mazel-Gee~\cite[\S5.1]{MG:GC}, and we will use this below.
\end{blanko}

\begin{blanko}[Lower-segments construction]\label{lower-segment}
  (See~\cite[Lemma 3.2]{Kock-Spivak:1807.06000}). For any $k$, let $f \in
  (\nerve\simplexcategory)_k$ denote a sequence of maps
  $[n_0]\stackrel{f_1}\to[n_1]\stackrel{f_2}\to\cdots\stackrel{f_k}\to[n_k]$
  in $\simplexcategory$. Then there is a unique commutative diagram of the form
\[\begin{tikzcd}[column sep=large]
  {[0]}\ar[r, "d^\top"]
  \ar[d]&
%   \ar[d, "\beta(n_0)"']&
  {[1]}\ar[r, "d^\top"]
  \ar[d]&
%   \ar[d, "\beta(f_1)"']&
  \cdots\ar[r, "d^\top"]&
  {[k]}\ar[d, "\beta_f"']\\
  {[n_0]}\ar[r, "f_1"']&
  {[n_1]}\ar[r, "f_2"']&
  \cdots\ar[r, "f_k"']&
  {[n_k]}
\end{tikzcd}\]
  for which all the vertical maps are \toppreserving, and all the maps in
  the top row are $d^\top$. Indeed, building the diagram from the left to 
  the right, in each step it remains to define the
  next $\beta$ map on the last vertex, and here the value is determined by
  the requirement that it be \toppreserving.

If $f \in (\nerve\simplexcategory)_k$ is as above, and $0\leq i \leq j \leq k$, we write $f_{ij} \colon [n_i] \to [n_j]$ for the composite $f_j \cdots f_{i+1}$ (or $\id_{[n_i]}$ when $i=j$).
The resulting map $\beta_f$ can be described explicitly as 
\begin{eqnarray*}
  {}[k] & \longrightarrow & [n_k]  \\
  i & \longmapsto & %f(\top_i)
  f_{ik}(n_i).
\end{eqnarray*}
Let $\newlen \colon \el (\nerve \simplexcategory) \to \simplexcategory$ be 
the right fibration associated to $\nerve \simplexcategory$ (actually, a discrete fibration of 1-categories), and let $\newlast_{\simplexcategory} \colon \el (\nerve \simplexcategory) \to \simplexcategory$ be the functor sending $f \colon [k] \to \simplexcategory$ to $f(k) = [n_k]$.
\end{blanko}

\begin{lemma}\label{lem beta natural}
The maps $\beta_f$ define a natural transformation of functors
\[ \beta\colon \newlen \Rightarrow \newlast_{\simplexcategory} \colon \el (\nerve\simplexcategory) \to \simplexcategory.\]
\end{lemma}

\begin{proof}
Let $f\colon [k] \to \simplexcategory$ and $g\colon [\ell] \to 
\simplexcategory$ be two objects in $\el (\nerve\simplexcategory)$, and suppose that $\gamma \colon g \to f$ is a map.
That is, $\gamma$ is a map $[\ell] \to [k]$ in $\simplexcategory$ with $f\gamma = g$.
As above, write $f_{ij} \colon [n_i] \to [n_j]$ for $f(i \to j)$ and $g_{ij} \colon [m_i] \to [m_j]$ for $g(i \to j)$.
We have $\newlast_{\simplexcategory}(\gamma) = f_{\gamma(\ell)k} \colon [m_\ell] = [n_{\gamma(\ell)}]  \to [n_k]$, so we wish to show that the square
\[
\begin{tikzcd}
{[\ell]} \rar{\gamma} \dar[swap]{\beta_g} & {[k]} \dar{\beta_f} \\
{[m_\ell]} \rar[swap]{f_{\gamma(\ell)k}} & {[n_k]}
\end{tikzcd}
\]
commutes.
But for $0\leq t \leq \ell$ we have
\begin{align*}
f_{\gamma(\ell)k}(\beta_g(t)) &= f_{\gamma(\ell)k}(g_{t\ell}(m_t)) \\
&= f_{\gamma(\ell)k}(f_{\gamma(t)\gamma(\ell)}(n_{\gamma(t)})) \\
&= f_{\gamma(t)\ell}(n_{\gamma(t)}) = \beta_f(\gamma(t)),
\end{align*}
so $\beta$ is a natural transformation.
\end{proof}

We are going to lift this natural transformation to an arbitrary simplicial space $X$.

\begin{lemma}[{\cite[\S5.1]{MG:GC}}]\label{lem:last}
    There is a natural transformation $\newlast \colon \el \nerve \Rightarrow \id_{\catinf}$. 
    At an $\infty$-category $\CC$,
    it takes an object $f\colon [k]{\to}\CC$ to its last vertex $f(k)$.
\end{lemma}

\begin{proof}
For a simplicial space $X\colon \simplexcategory\op \to \spaces$, temporarily write $q\colon \int X \to \simplexcategory\op$ for associated left fibration.
Then $q\op \colon (\int X)\op \to \simplexcategory$ is equivalent to the right fibration $p\colon \el (X) \to \simplexcategory$.
Mazel-Gee in \cite[Construction 5.4]{MG:GC} produced the map $\int \nerve\DD \to \DD$ which on objects picks out the \emph{first} vertex, rather than the last.
But we can recover the desired map as follows.
First, take opposites to get $\el(\nerve\DD) \simeq (\int \nerve\DD)\op \to \DD\op$.
Then, pull back as follows:
\[
\begin{tikzcd}
\el (\nerve(\DD)\op) \rar \dar \drpullback & \el (\nerve \DD) \dar \rar & \DD\op \\
\simplexcategory \rar[swap]{\text{rev}} & \simplexcategory ,
\end{tikzcd}
\]
where $\text{rev}$ is the unique nontrivial automorphism of $\simplexcategory$. 
Instantiating the top line at $\DD = \CC\op$ then gives the desired map 
$\el (\nerve\CC) \to \CC.$
As $\int \nerve\DD \to \DD$ is natural by \cite{MG:GC}, so too is $\el (\nerve \CC) \to \CC$.
\end{proof}

\begin{blanko}[Construction of the last-vertex map]\label{bl construction last vertex}
  We instantiate Lemma~\ref{lem:last} at the canonical projection $p\colon \el(X)\to \simplexcategory$, for $X$ a general simplicial space. This gives the solid (lower) square in the diagram
\begin{equation}\label{diag cylinder}
\begin{tikzcd}[sep={15mm,between origins}]
\el (\Nel X )
\ar[dd, "\el\nerve(p)"']
\ar[rr, dotted, bend left, "\widetilde\xi_X"]
\ar[rr, phantom, "\Downarrow", start anchor=center, end anchor=center]
\ar[rr, bend right, "\newlast_{\el(X)}"']
& & \el(X)
\ar[dd, "p"]
\\
\\
\el (\nerve \simplexcategory ) 
\ar[rr, bend left, "\newlen"]
\ar[rr, phantom, "\Downarrow\beta"]
\ar[rr, bend right, "\newlast_{\simplexcategory}"' ]
& & \simplexcategory  .
\end{tikzcd}
\end{equation}
We now lift the natural transformation $\beta$ to $\newlast_{\el(X)}$.
This is possible because $p\colon \el(X) \to \simplexcategory$ is a right fibration. Indeed,
write $\beta$ as $[1] \times \el (\nerve \simplexcategory) \to \simplexcategory$, and consider
the commutative square 
\[ \begin{tikzcd}[column sep = large]
\{1\} \times \el (\Nel X) \ar[rr, "\newlast_{\el X}"] \dar & & \el (X) \dar{p} \\
{[1]} \times \el (\Nel X) \rar[swap]{\id \times \el \nerve (p)} &  {[1]} \times \el (\nerve \simplexcategory) \rar[swap]{\beta} & \simplexcategory .
\end{tikzcd} \]

Since the vertical map on the left is final while the vertical map on the right is a right fibration, we get a unique lift $\beta_X$ in the square, which is
the asserted lifted natural transformation, whose domain we call $\widetilde\xi_X \colon \el (\Nel X) \to \el (X)$.

This functor $\widetilde \xi_X$ is a right fibration because it fits
into the upper square of \eqref{diag cylinder} with $\el(\nerve p)$, $\newlen$, and $p$ all right fibrations. 
Hence $\widetilde \xi_X$ comes from a simplicial map
$$
\xi_X \colon \Nel (X) \to X
$$
which is the ``last-vertex'' map. 
In \cref{lem xi natural} we will show this is natural in $X$.
\end{blanko}

The argument constructing $\beta_X$ above just as well proves the following general lemma, which we will use several times below.
\begin{lemma}\label{lem nt lifting}
Suppose we are given a commutative square in $\catinf$ and a natural 
transformation $\gamma$ as depicted below.
\[ \begin{tikzcd}
\CC \rar{g'} \dar & \EE \dar{p} \\
\mathcal{A}  \rar{g} & \mathcal{B}
\end{tikzcd} \qquad 
\begin{tikzcd}
\mathcal{A}  \rar[bend right, "g"'] \rar[bend left, "f"] \rar[phantom, "\Downarrow \gamma"]& \mathcal{B}
\end{tikzcd}
\] 
If $p$ is a right fibration, then there is a unique lift of $\gamma$ to a natural transformation with codomain $g'$.
More precisely, the space of dotted fillers in 
\[ \begin{tikzcd}
\{ 1 \} \times \CC \rar{g'} \dar & \{ 1 \} \times \EE \dar \\
{[1]} \times \CC \rar[dotted, "\gamma'" description] \dar & {[1]} \times \EE \dar{\id \times p} \\
{[1]} \times \mathcal{A}  \rar[swap]{\gamma} & {[1]} \times \mathcal{B}
\end{tikzcd} \]
is contractible. \qed
\end{lemma}

\begin{lemma}\label{lem xi natural}
The maps $\widetilde{\xi}_X$ assemble into a natural transformation \[ \widetilde{\xi} \colon \el \Nel  \Rightarrow \el \colon \simpspaces \to \RFib(\simplexcategory)\subseteq (\catinf)_{/\simplexcategory}. \]
Consequently, we also have a natural transformation
\[
  \xi \colon \Nel \Rightarrow \id_{\simpspaces}.
\]
\end{lemma}

\begin{proof}
Let $Y \to X$ be a simplicial map, with associated right fibration $q \colon \el (Y) \to \el (X)$ over $\simplexcategory$.
Consider the square
\begin{equation}\label{diag to lift}
\begin{tikzcd}[column sep=huge]
\{1\} \times \el (\Nel Y) \rar{q \circ \newlast_{\el Y}} \dar & \el (X) \dar{p} \\
{[1]} \times \el (\Nel Y) \rar[swap]{\beta \circ (\id \times \el \nerve (pq))} \urar[dotted] & \simplexcategory,
\end{tikzcd}
\end{equation}
which has a contractible space of lifts since $p$ is a right fibration and the left-hand map is final.
We will actually show that the top square in the diagram
\begin{equation}\label{diag naturality}
\begin{tikzcd}
{[1]} \times \el (\Nel Y) \rar{\beta_Y} \dar["\id \times \el \nerve (q)" swap] & \el (Y) \dar{q} \\
{[1]} \times \el (\Nel X) \rar{\beta_X} \dar[swap]{\id \times \el \nerve(p)} & \el (X) \dar{p} \\
{[1]} \times \el (\nerve \simplexcategory) \rar{\beta} & \simplexcategory 
\end{tikzcd}
\end{equation}
commutes, by showing that both ways around this square are lifts in \eqref{diag to lift}.
For the bottom triangle in \eqref{diag to lift}, this amounts to the known commutativity of the outer rectangle and the bottom square of \eqref{diag naturality}.
The top triangle in \eqref{diag to lift} commutes in each case since
\[
\begin{tikzcd}
\el (\Nel Y) \ar[r, "\newlast_{\el Y}"] \ar[d] \ar[rd, "\el \nerve(q)" description]& \el (Y) \ar[rd, "q"] &   \\
{[1] \times \el (\Nel Y)} \ar[rd, "\id\times \el \nerve(q)"'] &  \el (\Nel X) \ar[r, 
"\newlast_{\el X}"'] \ar[d] & \el(X) \\
& {[1]\times \el(\Nel X)} \ar[ru, "\beta_X"', bend right]
\end{tikzcd}
\]
and
\[
\begin{tikzcd}
\el (\Nel Y) \dar \ar[r, "\newlast_{\el Y}"] &  \el (Y) \rar{q} & \el (X) \\
{[1]} \times \el (\Nel Y) \urar[swap]{\beta_Y} & &
\end{tikzcd}
\]
both commute.
By uniqueness of lifts of \eqref{diag to lift}, the top square in \eqref{diag naturality} commutes. Hence 
\[ \begin{tikzcd}
\el (\Nel Y) \ar[r, "\widetilde{\xi}_{Y}"] \dar[swap]{\el \nerve(q)} &  \el (Y) \dar{q} \\
\el (\Nel X) \ar[r, "\widetilde{\xi}_{X}"'] &  \el (X)
\end{tikzcd} \]
commutes as well.
\end{proof}

Notice that if $\CC$ is an $\infty$-category, then 
$\xi_{\nerve \CC} = \nerve(\newlast_\CC)$.

\begin{lemma}\label{el-L}
  The natural transformation $\last \colon \Nel \Rightarrow \id$ 
  is cartesian on right fibrations. That is,
  for $p\colon Y \to X$ a right fibration between simplicial spaces, 
  the naturality square
  \[
  \begin{tikzcd}
  \Nel(Y) \ar[r, "\last_Y"] \ar[d, "p'"'] & Y \ar[d, "p"]  \\
  \Nel(X) \ar[r, "\last_X"'] & X
  \end{tikzcd}
  \]
  is a pullback.
\end{lemma}

\begin{proof}
  We check it in each simplicial degree separately. In simplicial degree $k$
  we have
  $$
  \Nel(Y)_k = \sum_{[n_0]\to\cdots\to [n_k]} Y_{n_k} .
  $$
  If the chain of maps $[n_0]\to\cdots\to [n_k]$ is called $f$, then 
  the lower-segments 
  construction (\ref{lower-segment}) gives us a \toppreserving\ map $\beta_f \colon [k] \to [n_k]$,
  and $\last_Y$ is given in the $f$-summand by
  $$
 Y_{n_k} \stackrel{\beta_f\upperstar} \longrightarrow Y_{k} .
  $$
  Altogether, the square we want to show is a pullback (in degree $k$)
  is identified 
  with
  \[
  \begin{tikzcd}
  \sum_f Y_{n_k} \ar[r, "\beta_f\upperstar"] \ar[d] & Y_{k} \ar[d, 
  "p_{k}"]  \\
  \sum_f X_{n_k} \ar[r, "\beta_f\upperstar"'] & X_{k} 
  \end{tikzcd}
  \]
 where the horizontal maps on each summand depend on $f\in (\nerve \simplexcategory)_k$.
  The left vertical map $p'_k$ respects $f$ (since $p'$ is a morphism of 
  right fibrations over $\nerve\simplexcategory$). Therefore, since pullbacks commute with sums, the 
  pullback property can be established separately for each summand. 
  For a fixed chain $f$, the square is thus
  \[
  \begin{tikzcd}
  Y_{n_k} \ar[r, "\beta_f\upperstar"] \ar[d] & Y_k \ar[d, "p_{k}"]  \\
  X_{n_k} \ar[r, "\beta_f\upperstar"'] & X_{k} ,
  \end{tikzcd}
  \]
  and this square is a pullback since $\beta_f$ is \toppreserving\ and 
  $p$ is a right fibration.
\end{proof}

\begin{lemma}\label{L-final}  For any simplicial space $A$,
  the last-vertex map $\last \colon \Nel(A)\to A$ is final.
\end{lemma}
The analogue of this lemma for simplicial sets appears in works of Lurie~\cite[4.2.3.14]{HTT} and Cisinski~\cite[7.3.9]{Cisinski:HCHA}.
\begin{proof}
  To show that $\last_A \colon \Nel (A) \to A$ is final, we need to show that there is a contractible space of lifts for the square
  \begin{equation}\label{NELXYZ}
  \begin{tikzcd}
  \Nel(A) \ar[d, "\last_A"'] \ar[r, "f"] & Y \ar[d, "p"]  \\
  A \ar[r, "g"'] & X
  \end{tikzcd}
  \end{equation}
  for any right fibration $p\colon Y \to X$ and arbitrary simplicial maps $f$ and $g$.
  
  By Lemma~\ref{el-L} we have the pullback square
  \[
  \begin{tikzcd}
  \Nel(Y) \drpullback \ar[r, "\last_Y"] \ar[d, "p'"'] & Y \ar[d, "p"]  \\
  \Nel(X) \ar[r, "\last_X"'] & X .
  \end{tikzcd}
  \]
This pullback square is the outer square of the diagram
\[ \begin{tikzcd}[sep=small]
\Nel(Y) \ar[rrr,"\last_Y"] \ar[ddd,"p'"'] & & & Y \ar[ddd,"p"] \\
& \Nel(A) \ar[dr, "\last_A"] \ar[urr,"f"'] \ar[ddl,"g'"] \ar[ul,dashed,"h'"] \\
& & A  \ar[dr,"g"] \\
\Nel(X) \ar[rrr,"\last_X"] & & & X ,
\end{tikzcd} \]
where the two distorted squares are \eqref{NELXYZ} and naturality of $\last$ with respect to $g$.
The universal property of the pullback now gives the dashed arrow $h'$.

All three small triangles in the following diagram commute, hence the outer triangle commutes as well. 
\[ \begin{tikzcd}
\Nel(A) \ar[dr,"g'"'] \ar[rr,"h'"] \ar[ddr,bend right] & & \Nel(Y) \ar[dl,"p'"] \ar[ddl, bend left] \\[-0.4cm]
& \Nel (X) \dar  \\
& \nerve \simplexcategory  
\end{tikzcd} \]
Since the nerve functor is fully faithful, this implies that $h'$ is the nerve of a right fibration $\el (A) \to \el (Y)$ over $\simplexcategory$, hence $h' = \Nel(h)$ for a unique simplicial map $h \colon A \to Y$, as in the solid triangle

  \[
  \begin{tikzcd}
  { \color{gray} \Nel(A) } \ar[d, color=gray, "\last_A"'] \ar[r, color=gray, "f"] & Y \ar[d, "p"]  \\
  A \ar[r, "g"'] \ar[ur,"h"] & X  .
  \end{tikzcd}
  \]
Meanwhile, the upper triangle is the outer triangle in the commutative diagram
\[
\begin{tikzcd}
\Nel (A) \ar[rr,bend left=30,"f"] \dar["\xi_A"'] \ar[r,"h'"] & \Nel (Y) \rar["\xi_Y"] & Y \\
A \ar[urr,"h"']
\end{tikzcd} \]
so $h$ is a lift in the square.
\end{proof}

\begin{blanko}[Alternate interpretation of proof]
We want to prove that $\Map(A,Y) \to \Map(\last_A, p)$ is an equivalence.
But the statement of \cref{el-L} is that 
\[
\Map(\Nel A , \Nel Y) \to \Map(\Nel A, Y) \times_{\Map(\Nel A, X)} \Map(\Nel A, \Nel X)
\]
is an equivalence.
Now we have \[ \Map(\last_A, p) = \Map(\Nel A, Y) \times_{\Map(\Nel A, X)} \Map(A, X) , \]
and $\Map(A, X) \subset \Map(\Nel A, \Nel X)$ consists of those maps living over $\nerve \simplexcategory$.
So this proof is about identifying the two subspaces on the top row:
\[ \begin{tikzcd} 
\Map(A,Y) \rar \dar[hook,"i"] & \Map(\Nel A, Y) \times_{\Map(\Nel A, X)} \Map(A, X)  \dar[hook,"j"]
\\
\Map(\Nel A , \Nel Y) \rar{\sim} & 
\Map(\Nel A, Y) \times_{\Map(\Nel A, X)} \Map(\Nel A, \Nel X)  .
\end{tikzcd} \]

But we're exhibiting a map in the opposite direction on the top by
composing $j$ with a homotopy inverse on the bottom, and then showing that
it lands in $\Map(A,Y)$ (or factors through $i$). In that case the
exhibited map is automatically a homotopy inverse since the downward arrows
are monomorphisms of $\infty$-groupoids.
\end{blanko}

\begin{lemma}\label{lemma:epsilonL is equiv}
  For $\last_X \colon \Nel(X) \to X$ the last-vertex map of 
  \ref{lastvertex}, the counit $(\last_X)\lowershriek (\last_X)\upperstar \Rightarrow \Id$
  is an equivalence. In particular, \[(\last_X)\upperstar\colon 
  \simprfib(X) \to \simprfib(\Nel X)\] is fully faithful.
\end{lemma}
\begin{proof}
  For any right fibration $p\colon Y \to X$, the pullback diagram 
  \[
  \begin{tikzcd}
  \Nel(Y) \drpullback \ar[r, "\last_Y"] \ar[d, "p'"'] & Y \ar[d, "p"]  \\
  \Nel(X) \ar[r, "\last_X"'] & X
  \end{tikzcd}
  \]
  of Lemma~\ref{el-L}, together with the fact that $\last_Y$ is final 
  (Lemma~\ref{L-final}), shows that $p$ is already the 
  right-fibration part of the
  final-rightfibration factorization of $\last_X\circ p'$, so $\varepsilon_p \colon 
  (\last_X)\lowershriek (\last_X)\upperstar (p) \to p$ is the identity.
\end{proof}

\begin{cor}
  For a simplicial space $X$, we have a natural equivalence $(\last_X)\upperstar \simeq \Nel_X$ of 
  functors from $\simprfib(X)$ to $\simprfib(\Nel(X))$.
\end{cor}

\noindent 
Here, $\Nel_X$ takes a right fibration $p\colon Y \to X$ to $\Nel(p) \colon \Nel(Y) \to \Nel(X)$, following the convention that subscripts on functors indicate functors induced on slices (or subcategories of slices).

%%%%%%%%%%%%%%%%%%%%%%%%%%%%%%%%%%%%%%%%%%%%%%%%%%
\section{Edgewise subdivision and the natural transformation \texorpdfstring{$\lambda$}{λ}}
%%%%%%%%%%%%%%%%%%%%%%%%%%%%%%%%%%%%%%%%%%%%%%%%%%
\label{sec:Sd}

Consider the functor 
\begin{eqnarray*}
  Q  \colon  \simplexcategory & \longrightarrow & \simplexcategory \\
  {}[n] & \longmapsto & [n]\op\join [n] = [2n{+}1]  .
\end{eqnarray*}
With the following special notation (following Waldhausen~\cite{Waldhausen:AKTS})
for the elements 
of the ordinal
$[n]\op\join [n] = [2n{+}1]$,
  \begin{equation}\label{eqn.twisted_n}
  \begin{tikzcd}[cramped, row sep=small, column sep=small]
    0\ar[r]&1\ar[r]&\cdots\ar[r]&n\\
    0'\ar[u]&1'\ar[l]&\cdots\ar[l]&n' ,
    \ar[l]
  \end{tikzcd}
  \end{equation}
the functor $Q$ is 
described on arrows by sending a
coface map $d^i  \colon [n{-}1]\to [n]$ to the monotone injection that omits the 
elements $i$ and $i'$, and by sending a codegeneracy map $s^i  \colon [n] 
\to [n{-}1]$ to the monotone surjection that repeats both $i$ and $i'$. 

The next lemma follows from the definition.
\begin{lemma}\label{Q(top)=active}
  The functor $Q \colon \simplexcategory \to \simplexcategory$ sends 
  \toppreserving\ maps to active maps, giving this commutative square:
  \[
  \begin{tikzcd}
  \simplexcategory \ar[r, "Q"]  & \simplexcategory  \\
  \simplexcategory^t \ar[u] \ar[r, "Q"'] & 
  \simplexcategory_{\operatorname{act}}  . \ar[u]
  \end{tikzcd}
  \]
\end{lemma}

\begin{blanko}[Edgewise subdivision functor]\label{sd}
We now define the \emph{edgewise subdivision} functor $\sd = Q\upperstar \colon \simpspaces \to \simpspaces$. 
  For $X\colon \simplexcategory\op\to\spaces$, the simplicial space $\sd(X)$ is given by precomposing with $Q\colon \simplexcategory \to 
  \simplexcategory$:
  $$
  \sd (X) \coloneqq Q\upperstar X = X \circ Q .
  $$
  At the level of right fibrations over 
  $\simplexcategory$, this is simply the pullback 
  \begin{equation}\label{eq:omega}
  \begin{tikzcd}
  Q\upperstar (\el X) \drpullback \ar[d] \ar[r, "\omega"] & \el(X) \ar[d]  \\
  \simplexcategory \ar[r, "Q"'] & \simplexcategory .
  \end{tikzcd}
  \end{equation}
  This means that we have the identification
  $$
  Q\upperstar (\el X) = \el( \sd X).
  $$
  The top horizontal map $\omega\colon \el( \sd X) \to \el(X)$ has been 
  named because it will be used in several proofs below. 
  On objects it is given 
  by sending an $n$-simplex $\Delta^n \to \sd(X)$ to the corresponding map 
  under adjunction $Q\lowershriek\Delta^n \to X$, which is a $(2n+1)$-simplex of $X$. 

  As usual, $Q\upperstar$ has both a left adjoint $Q\lowershriek$ (sending a representable $\Delta^n$ to $\Delta^{2n+1}$) and a right adjoint $Q\lowerstar$, which will play a key role in Section~\ref{sec main via rke}.
\end{blanko}

\begin{lemma}
  \label{culf<=>rfib}
  A simplicial map $f\colon Y \to X$ is culf if and only if $\sd(f)\colon 
  \sd(Y) \to \sd(X)$ is a right fibration.
\end{lemma}

\begin{proof}
  Suppose $f\colon Y\to X$ is culf. By \cref{lemma right fib char}, to check that $\sd(Y) \to \sd(X)$ is a right 
  fibration is to show that for every every last-vertex inclusion $\ell\colon [0]\to[n]$
  there is a unique lift for the diagram on the left, 
  \[
  \begin{tikzcd}[sep={45pt,between origins}] 
  \Delta^0 \ar[d, "\ell"'] \ar[r] & \sd(Y) \ar[d, "\sd(f)"]  \\
  \Delta^n \ar[r] \ar[ru, dotted]& \sd(X)
  \end{tikzcd}
  \qquad
  \begin{tikzcd}[sep={45pt,between origins}]
  \Delta^1 \ar[d, -act, "Q\lowershriek(\ell)"'] \ar[r] & Y \ar[d, "f"]  \\
  \Delta^{2n+1} \ar[r] \ar[ru, dotted] & X,
  \end{tikzcd}
  \] 
  or equivalently, by adjunction (\ref{factorization systems}), a unique lift for the diagram on the 
  right.
  But if $f$ is culf then there is such a unique lifting, because 
  $Q\lowershriek(\ell)$ 
  is active (by \cref{Q(top)=active}).

  Conversely, suppose $\sd(Y) \to \sd(X)$ is a right 
  fibration. To check that $Y\to X$ is culf, we should find a lift against
  any active map $\Delta^1 \actto \Delta^k$, and by \cref{ulf=culf} it is 
  enough to treat $k>0$. For odd values of $k$, this 
  is the same adjunction argument in reverse.
  For even $k>0$, it is enough to observe that every active map of the form 
  $\Delta^1 \actto\Delta^{2n}$ ($n>0$) is a retract of $\Delta^1 \actto\Delta^{2n+1}$, 
  like those appearing in the 
  adjunction argument, so if $f$ is right orthogonal to $\Delta^1 \actto\Delta^{2n+1}$
  it is also right orthogonal to $\Delta^1 \actto\Delta^{2n}$. 
\end{proof}
\begin{blanko}[Remark]
  In the special case of $1$-categories, this result is due to Bunge and 
  Niefield~\cite[Proposition~4.4]{Bunge-Niefield}.
\end{blanko}

We will need also the following corollary which generalizes 
\cref{Q(top)=active}.

\begin{cor}\label{final<=>ambifinal}
  If $f\colon B \to A$ is final, then
  $Q\lowershriek(f) \colon Q\lowershriek(B) \to Q\lowershriek(A)$ is 
  ambifinal.
\end{cor}
\begin{proof}
  We need to check that there is a unique lift in the diagram on the left 
  (for every culf map $p$),
  but by the adjunction argument (\ref{factorization systems}) this is equivalent to having a unique 
  lift in the diagram on the right:
  \[
  \begin{tikzcd}[sep={45pt,between origins}]
  Q\lowershriek (B) \ar[d, "Q\lowershriek(f)"'] \ar[r] &  Y \ar[d, "p"]  \\
  Q\lowershriek (A) \ar[r] \ar[ru, dotted]& X
  \end{tikzcd}
  \quad
  \Leftrightarrow
  \quad
  \begin{tikzcd}[sep={45pt,between origins}]
  B \ar[d, "f"'] \ar[r] & Q\upperstar (Y) \ar[d, "Q\upperstar(p)"]  \\
  A \ar[r] \ar[ru, dotted]& Q\upperstar(X) ,
  \end{tikzcd}
\]
  and this lift exists uniquely since $Q\upperstar(p)$ is a right fibration by 
  \cref{culf<=>rfib}.
\end{proof}

\begin{blanko}[The natural transformation $\lambda \colon \Nel \Rightarrow \sd$]\label{bl nt lambda}
  We construct a natural transformation
  $$
  \lambda \colon \Nel \Rightarrow \sd
  $$
  whose component on a simplicial space $X$ is given in simplicial 
  degree $0$ by sending $\Delta^n \to X$ (a $0$-simplex of $\Nel(X)$)
  to the long edge $\Delta^1 \actto \Delta^n \to X$ (considered as a 
  $0$-simplex of $\sd(X)$).
  
  The action of $\lambda$ on higher simplices is given, with 
  reference to the general combinatorial {\em middle-segments 
  construction} below, by sending 
  $$
  f \; \colon \
  \Delta^{n_0} \to \cdots\to \Delta^{n_k} \to X 
  \qquad\qquad \in (\Nel X)_k
  $$
  to
  $$
  Q\lowershriek \Delta^{k} \stackrel{\alpha_f}{\actto} \Delta^{n_k} \to X
    \qquad\qquad \in (\sd X)_k .
  $$
  The construction is similar to that of $\xi$ in Section~\ref{sec:lastvertex}, and it relies on some of the same lemmas established there.

\end{blanko}

\begin{blanko}[Remark]
  In the $1$-category case, the natural transformation $\lambda$ goes back to Thomason's notebooks
  \cite[p.152]{Thomason:notebook85}; it was exploited by
  G\'alvez--Neumann--Tonks~\cite{Galvez-Neumann-Tonks:JPAA2013} to exhibit
  Baues--Wirsching cohomology as a special case of Gabriel--Zisman
  cohomology. The case of general simplicial sets is from Kock--Spivak~\cite{Kock-Spivak:1807.06000}, whose treatment we upgrade to the case of general simplicial spaces.
\end{blanko}

\begin{blanko}[Middle-segments construction]\label{middle-segment}
  [Cf.~\cite[Lemma 3.3]{Kock-Spivak:1807.06000}).] This is a two-sided variant of 
the lower-segments construction \ref{lower-segment}: For
any $k$, let $f \in (\nerve\simplexcategory)_k$ denote a sequence
of maps
$[n_0]\stackrel{f_1}\to[n_1]\stackrel{f_2}\to\cdots\stackrel{f_k}\to[n_k]$
in $\simplexcategory$. Then there is a unique commutative diagram of the form
\begin{equation}
\begin{tikzcd}[column sep=large]
  Q[0]\ar[r, "Q\left(d^\top\right)"]
  \ar[d, ->|]&
%   \ar[d, ->|, "\alpha(n_0)"']&
  Q[1]\ar[r, "Q\left(d^\top\right)"]
  \ar[d, ->|]&
%   \ar[d, ->|, "\alpha(f_1)"']&
  \cdots\ar[r, "Q\left(d^\top\right)"]&
  Q[k]\ar[d, ->|, "\alpha_f"']\\
  {[n_0]}\ar[r, "f_1"']&
  {[n_1]}\ar[r, "f_2"']&
  \cdots\ar[r, "f_k"']&
  {[n_k]} ,
\end{tikzcd}
\end{equation}
i.e.~for which all the vertical maps are active and all the maps in the 
top row are
of the form $Q(d^\top)$.

 Indeed, building the diagram from the left to 
  the right, in each step it remains to define the
  next vertical map on the first and the last vertex, and here the value is determined by
  the requirement that it be active.

The resulting map $\alpha_f$ can be described explicitly as 
\begin{eqnarray*}
  {}Q[k] & \longrightarrow & [n_k]  \\
  i & \longmapsto & f_{ik}(n_i) \\
%   f(\top_i)  \\
    i' & \longmapsto & 
%   f(\bot_i)
  f_{ik}(0) .
\end{eqnarray*}
Here, as in \ref{lower-segment}, we are writing $f_{ij} \colon [n_i] \to [n_j]$
for $f(i\to j)$.
\end{blanko}

\begin{lemma}
The maps $\alpha_f$ define a natural transformation of functors \[ \alpha\colon Q\circ \newlen \Rightarrow \newlast_{\simplexcategory} \colon \el \nerve(\simplexcategory) \to \simplexcategory.\]
\end{lemma}

\begin{proof}
The proof is a minor variation on that of \cref{lem beta natural}.
As in that proof, we consider two objects $f\colon [k] \to 
\simplexcategory$ and $g\colon [\ell] \to \simplexcategory$ of $\el (\nerve\simplexcategory)$ along with a map $\gamma \colon g \to f$.
We wish to show that the square
\[
\begin{tikzcd}
{Q[\ell]} \rar{Q\gamma} \dar[swap]{\alpha_g} & {Q[k]} \dar{\alpha_f} \\
{[m_\ell]} \rar[swap]{f_{\gamma(\ell)k}} & {[n_k]}
\end{tikzcd}
\]
commutes.
Recall that $Q[k]$ has two types of elements: $t$ and $t'$.
For $0\leq t \leq \ell$ we have
\begin{align*}
f_{\gamma(\ell)k}(\alpha_g(t')) &= f_{\gamma(\ell)k}(g_{t\ell}(0)) \\
&= f_{\gamma(\ell)k}(f_{\gamma(t)\gamma(\ell)}(0)) \\
&= f_{\gamma(t)\ell}(0) = \alpha_f(\gamma(t)') = \alpha_f((Q\gamma)(t')).
\end{align*}
As $\alpha_f(t) = \beta_f(t)$ we have
$f_{\gamma(\ell)k}(\alpha_g(t)) = f_{\gamma(\ell)k}(\beta_g(t)) = \beta_f(\gamma(t)) = \alpha_f((Q\gamma)(t))$ by \cref{lem beta natural}.
Hence the square commutes and $\alpha$ is a natural transformation.
\end{proof}

\begin{blanko}[Construction of $\widetilde \lambda$] \label{construction lambdatilde}
  Let $X$ be a simplicial space, with unstraightening $p\colon 
  \el X \to \simplexcategory$.
As in \ref{bl construction last vertex} and \ref{lem nt lifting},
we can lift $\alpha$ to a natural transformation 
\[
\begin{tikzcd}[sep=large]
\el (\Nel X) \rar[bend left] \rar[bend right,"\newlast_{\el X}"'] \rar[phantom,"\Downarrow \alpha_X"] & \el (X)
\end{tikzcd}
\]
by taking the diagonal lift in the square
\[ \begin{tikzcd}[column sep = {10em,between origins}, 
row sep = {5em,between origins}]
\{1\} \times \el (\Nel X) \ar[rr, "\newlast_{\el X}"] \dar & & \el (X) \dar{p} \\
{[1]} \times \el (\Nel X) \ar[urr, dashed, "\alpha_X" description] 
\rar[swap]{\id \times \el \nerve (p)} &  {[1]} \times \el (\nerve \simplexcategory) \rar[swap]{\alpha} & \simplexcategory .
\end{tikzcd} \]
We write $\mu_X \colon \el (\Nel X) \to \el (X)$ for the domain of the natural 
transformation $\alpha_X$. It thus fits into the outer commutative square 
in the diagram
\begin{equation}\label{eq:lambdatilde}
\begin{tikzcd}
\el (\Nel X) \ar[drr, bend left, "\mu_X"] \ar[dd, "\el\nerve(p)"'] \drar[dotted,"\tilde \lambda_X"]
\\
& \el (\sd X) \rar{\omega_X} \dar[swap]{q} \drar[phantom, "\lrcorner" very near start] & \el (X) \dar{p}  \\
\el(\nerve\simplexcategory) \ar[r, "\newlen"'] & \simplexcategory 
\rar[swap]{Q} & \simplexcategory .
\end{tikzcd}
\end{equation}
Next, we use the universal property of the pullback square in the diagram
to obtain a map
\[
\tilde \lambda_X \colon \el (\Nel X) \to \el (\sd X).
\]
Notice that $\tilde \lambda_X$ is automatically a right fibration since the
other maps in the left trapezoid of \eqref{eq:lambdatilde} are right
fibrations.
\end{blanko}

\begin{lemma}
The maps $\tilde \lambda_X$ assemble into a natural transformation
\[ \tilde\lambda \colon \el \Nel  \Rightarrow \el \sd \colon \simpspaces \to \RFib(\simplexcategory)\subseteq (\catinf)_{/\simplexcategory}. \]
Consequently, we also have a natural transformation $\lambda \colon \Nel \Rightarrow \sd$.
\end{lemma}

\begin{proof}
A minor modification of the proof of \cref{lem xi natural} (simply replacing $\beta$ by $\alpha$) shows that $\mu \colon \el\Nel  \Rightarrow \el$ is a natural transformation.
Now suppose $Y\to X$ is a simplicial map, and consider the diagram
\[
\begin{tikzcd}
\el (\Nel Y) \rar{\tilde \lambda_Y} \dar 
& \el (\sd Y) \rar{\omega_Y} \rar \dar & \el (Y) \dar \\
\el (\Nel X) \rar{\tilde \lambda_X} \dar & \el (\sd X) \rar{\omega_X} \rar \dar \drpullback & \el (X) \dar  \\
\el(\nerve\simplexcategory) \ar[r, "\newlen"'] & 
\simplexcategory \rar[swap]{Q} & \simplexcategory .
\end{tikzcd}
\]
We want to show that the upper left square commutes.
Everything else in this diagram commutes (using that $\mu$ is a natural transformation for the top rectangle); since the bottom square is a pullback, the two maps $\el(\Nel Y) \to \el (\sd X)$ are equivalent.
\end{proof}

\begin{blanko}[Remark]\label{bl lambda is as we say}
We've formally constructed the natural transformation $\lambda \colon \Nel \Rightarrow \sd$, but we should check that it actually behaves as we have described in \ref{bl nt lambda}. 
To this end, suppose we have $f\in (\Nel X)_k$, regarded as an object $f 
\colon [k] \to \Nel (X)$ of $\el (\Nel X)$.
Write $p(f) \in \el (\nerve \simplexcategory)$ as $[n_0] \to [n_1] \to \cdots \to [n_k]$.
Then $\mu_X$ takes $f$ to the object
\[
\begin{tikzcd}
{Q[k]} \ar[rr,"\mu_X(f)"] \drar["\alpha_{p(f)}"'] & & X \\
& {[n_k]} \ar[ur,"f(k)"']
\end{tikzcd}
\]
of $\el (X)$. This is the expected thing, except that we should not be regarding it 
as a $(2k+1)$-simplex of $X$, but rather as a $k$-simplex of $\sd (X)$. This is
exactly what the pullback of \ref{sd} accomplishes, so $\tilde \lambda_X(f)$ is
indeed what was specified in \ref{bl nt lambda}.
\end{blanko}

\begin{lemma}\label{activeequiv}
  For any simplicial space $X$, 
  the simplicial map $\lambda_X\colon \Nel \Rightarrow \sd$
  sends 1-simplices in $\Nel(X)$ lying over active maps in $\simplexcategory$
  to degenerate $1$-simplices of $\sd(X)$.
\end{lemma}
\begin{proof}
Let $f = \left( \Delta^{n_0}\xrightarrow{f_1}\Delta^{n_1}\to X\right)$ be a 1-simplex in $\Nel(X)$ with $f_1$ active.
Then $\lambda_X(f)$ is defined as in the triangle in the below diagram.
To prove the lemma, it is enough to show that $\alpha_{p(f)}$ factors through 
$Q(s^0)$ as in the left-hand square of the diagram
\[
\begin{tikzcd}
Q_!\Delta^0 \dar[-act] & Q_! \Delta^1 \lar["Q(s^0)"'] \dar[-act,"\alpha_{p(f)}"'] \drar["\lambda_X(f)"] \\
\Delta^{n_0} \rar[-act,"f_1"'] & \Delta^{n_1} \rar & X .
\end{tikzcd}
\]
Since $f_1$ is active, we have $\alpha_{p(f)}(0') = f_1(0) = 0$ and $\alpha_{p(f)}(0) = f_1(n_0) = n_1$.
Since $\alpha_{p(f)}$ is order preserving, and $1' \leq 0' \leq 0 \leq 1$, we then have $\alpha_{p(f)}(1') = 0$ and $\alpha_{p(f)}(1) = n_1$.
Thus $\alpha_{p(f)}$ factors as indicated in the left square above.
\end{proof}

\begin{blanko}[Remark]\label{rmk:len}
  Suppose $f, g\colon X \to \nerve \simplexcategory$ are any two maps. Then the
  square below left commutes, hence so does the square below right.
  \[\begin{tikzcd}
  X \rar{f} \dar[swap]{g} & \nerve \simplexcategory \dar  & & \el (X) \rar{\el(f)} 
  \dar[swap]{\el(g)} & \el (\nerve \simplexcategory) \dar{\newlen} \\
  \nerve \simplexcategory \rar & \ast &  & \el (\nerve \simplexcategory) \rar[swap]{\newlen} & \simplexcategory.
  \end{tikzcd}\]
\end{blanko}

The following formalizes \cite[3.7]{Kock-Spivak:1807.06000}.
\begin{lemma}\label{lem:QB=AQ}
  The natural transformations $\beta$ and $\alpha$ are related 
  as follows:
\[
\begin{tikzcd}[sep={15mm,between origins}]
\el (\nerve \simplexcategory) 
\ar[rr, bend left, pos=0.45, "\newlen"]
\ar[rr, phantom, "\Downarrow \beta", start anchor=center, end anchor=center]
\ar[rr, bend right, pos=0.45, "\newlast_{\simplexcategory}"']
& & \simplexcategory \ar[rr, "Q"] && \simplexcategory
\end{tikzcd}
\]
is equivalent to
\[
\begin{tikzcd}[sep={15mm,between origins}]
\el (\nerve \simplexcategory) \ar[rr, "\el\nerve(Q)"'] &&
\el (\nerve \simplexcategory)  
\ar[rr, bend left, "Q\circ \newlen"]
\ar[rr, phantom, "\Downarrow\alpha"]
\ar[rr, bend right, "\newlast_{\simplexcategory}"' ]
& & \simplexcategory  .
\end{tikzcd}
\]
\end{lemma}

\begin{proof}
  Note first that the domains of these natural transformation agree:
  $$
  Q \circ \newlen \circ \el(\nerve Q) 
  = Q \circ \newlen \circ \el(\nerve \id) = Q \circ \newlen .
  $$
  Here the first equation uses Remark~\ref{rmk:len}.
  The codomains of the natural transformations agree by
  naturality of $\newlast$.
  We now check the whiskerings by a direct computation.
  Let $f\colon [k] \to \simplexcategory$ be
  $[n_0]\stackrel{f_1}\to[n_1]\stackrel{f_2}\to\cdots\stackrel{f_k}\to[n_k]$
  and write $g = Q(f)$ for
\[\begin{tikzcd}
{Q[n_0]} \rar{Q(f_1)} & {Q[n_1]} \rar{Q(f_2)} & \cdots \rar{Q(f_k)} & {Q[n_k]} ,
\end{tikzcd}\]
so that $g_{ij}(t) = f_{ij}(t)$ and $g_{ij}(t') = f_{ij}(t)'$.
We'll use our preferred names $n_t$ and $n_t'$ for the maximal and minimal elements of $Q[n_t] = [n_t]\op\join [n_t]$ (in place of $2n_t+1$ and $0$) when referring to the definition of $\alpha$ from \ref{middle-segment}.
We then have
\[
\begin{gathered}
\alpha_g(t) = g_{tk} (\max(Q[n_t])) = g_{tk}(n_t) = f_{tk}(n_t) = \beta_f(t) = Q(\beta_f)(t)
\\
\alpha_g(t') = g_{tk} (\min(Q[n_t])) = g_{tk}(n_t') = f_{tk}(n_t)' = \beta_f(t)' = Q(\beta_f)(t'). 
\end{gathered}
\]
Thus $\alpha_{Q(f)} = Q(\beta_f)$.
\end{proof}

Lifting this equation to general simplicial spaces, we get the following.
\begin{lemma}\label{lem:xi-tilde and mu}
    For any simplicial space $X$, we have the commutative diagram
    \[
    \begin{tikzcd}
\el(\Nel\sd X) \ar[d, "\el\nerve(\omega_X)"'] 
\ar[r, "\widetilde\xi_{\sd X}"] & \el(\sd X) \ar[d, "\omega_X"]  \\
\el(\Nel X) \ar[r, "\mu_X"'] & \el (X)    .
\end{tikzcd}
    \]
\end{lemma}

\begin{proof}
    Both composites are the domain of the lift of the (two versions of the) 
	natural transformation of Lemma~\ref{lem:QB=AQ} along 
	the right fibration $\el(X)\to\simplexcategory$, cf.~Lemma~\ref{lem nt lifting}. In more detail,
\begin{equation}\label{eq:liftwhiskbeta}
\begin{tikzcd}[sep={15mm,between origins}]
\el(\Nel\sd X)
\ar[rr, bend left, "\widetilde\xi_{\sd X}"]
\ar[rr, phantom, "\Downarrow \beta_X", start anchor=center, end anchor=center]
\ar[rr, bend right, "\newlast_{\sd X}"']
& & \el(\sd X) \ar[rr, "\omega_X"] && \el (X)
\end{tikzcd}
\end{equation}
is the lift of
\[
\begin{tikzcd}[sep={15mm,between origins}]
\el (\nerve \simplexcategory )
\ar[rr, bend left, "\newlen" pos=45/100]
\ar[rr, phantom, "\Downarrow \beta", start anchor=center, end anchor=center]
\ar[rr, bend right, "\newlast_{\simplexcategory}"' pos=45/100]
& & \simplexcategory \ar[rr, "Q"] && \simplexcategory  ,
\end{tikzcd}
\]
and on the other hand
\begin{equation}\label{eq:liftwhiskalpha}
\begin{tikzcd}[sep={15mm,between origins}]
\el(\Nel\sd X) \ar[rr, "\el\nerve(\omega)"'] &&
\el (\Nel X ) 
\ar[rr, bend left, "\mu_X"]
\ar[rr, phantom, "\Downarrow \alpha_X"]
\ar[rr, bend right, "\newlast_{X}"']
& & \el (X)  
\end{tikzcd}
\end{equation}
is the lift of
\[
\begin{tikzcd}[sep={15mm,between origins}]
\el(\nerve \simplexcategory) \ar[rr, "\el\nerve(Q)"'] &&
\el (\nerve \simplexcategory ) 
\ar[rr, bend left, "Q\circ \newlen"]
\ar[rr, phantom, "\Downarrow\alpha"]
\ar[rr, bend right, "\newlast_{\simplexcategory}"']
& & \simplexcategory  .
\end{tikzcd}
\]
Therefore these two natural transformations \eqref{eq:liftwhiskbeta} and \eqref{eq:liftwhiskalpha} agree, and in particular their domains agree, which is the assertion of the lemma.
\end{proof}

The following is the analog of \cite[Lemma 3.8]{Kock-Spivak:1807.06000} for simplicial spaces.

\begin{lemma}\label{lambda-omega}
  There is a natural commutative diagram of simplicial spaces
  \[
  \begin{tikzcd}[column sep={70pt,between origins}]
   &&& \sd(X)  \\
  \Nel(X) \ar[rrru, bend left=12, "\lambda_X" ] &&
  Q\upperstar (\Nel X) 
  \ar[ll, "\nerve(\omega_X)"] \ar[r, phantom, "\simeq" description]
  & \Nel(\sd X) . \ar[u, "\last_{\sd X }"']
  \end{tikzcd}
  \]  
\end{lemma}
\begin{proof}
We establish the corresponding triangle at the level of elements:
  \[
  \begin{tikzcd}[column sep={90pt,between origins}]
   & \el(\sd X)  \\
  \el(\Nel X) \ar[ru, bend left=12, "\widetilde\lambda_X" ] & 
  \ar[l, "\el\nerve(\omega_X)"] \el(\Nel\sd X) .
  \ar[u, "\widetilde\last_{\sd X}"']
  \end{tikzcd}
  \]  
  Since this is the equation of two maps with codomain $\el (\sd X)$, we can
  exploit the pullback characterization of the latter:
  \[
  \begin{tikzcd}
  \el (\sd X) \drpullback \ar[d, "q"'] \ar[r, "\omega_X"] & \el(X) \ar[d, "p"]  \\
  \simplexcategory \ar[r, "Q"'] & \simplexcategory .
  \end{tikzcd}
  \]
  It is thus
  enough to show that the two maps become equal after postcomposition with
  $q$ and that they become equal after postcomposition with $\omega_X$.
  
  For postcomposition with $q$ we compute
  \begin{eqnarray*}
    q \circ \widetilde \lambda_X \circ \el \nerve(\omega_X) & \stackrel{\eqref{eq:lambdatilde}} = 
	 & \newlen \circ \el \nerve (p) \circ 
	 \el\nerve(\omega_X) \\
	 & \stackrel{\eqref{eq:omega}} = & \newlen \circ \el \nerve (Q) \circ 
	 \el \nerve(q) \\
	 &  = & \newlen \circ \el \nerve (Q\circ q) \\
	 & \stackrel{\text{\ref{rmk:len}}} = & \newlen \circ \el \nerve (q) \\
	 & \stackrel{\ref{bl construction last vertex}}= & q \circ 
	 \widetilde\xi_{\sd X} .
  \end{eqnarray*}
  
  For postcomposition with $\omega_X$, we compute
    \begin{eqnarray*}
    \omega_X \circ \widetilde \lambda_X \circ \el \nerve(\omega_X) & 
	\stackrel{\eqref{eq:lambdatilde}}= & \mu_X \circ \el\nerve(\omega_X) 
	\\
	& \stackrel{\ref{lem:xi-tilde and mu}}=  & \omega_X \circ 
	 \widetilde\xi_{\sd X} .
    \end{eqnarray*}
\end{proof}

The following cartesian property was established for discrete decomposition
spaces in \cite[Proposition~3.9]{Kock-Spivak:1807.06000}.

\begin{lemma}\label{lemma:lambdaculf}
  The natural transformation $\lambda\colon \Nel \Rightarrow \sd$ is 
  cartesian on culf maps. In other words, for every culf map $F\colon 
  Y \to X$ of simplicial spaces, the naturality square
  \[
  \begin{tikzcd}
  \Nel(Y) \ar[r, "\lambda_Y"] \ar[d, "\Nel(F)"'] & \sd(Y) \ar[d, 
  "\sd(F)"]  \\
  \Nel(X) \ar[r, "\lambda_X"'] & \sd(X)
  \end{tikzcd}
  \]
  is a pullback.
\end{lemma}
  
\begin{proof}
  This proof follows the same idea as that of Lemma~\ref{el-L}, but using
  the middle-segments construction \ref{middle-segment} instead of the
  lower-segments construction \ref{lower-segment}. We check it in each
  simplicial degree separately. In simplicial degree $k$ we have
  $$
  \Nel(Y)_k = \sum_{[n_0]\to\cdots\to [n_k]} Y_{n_k} .
  $$
  If the chain of maps $[n_0]\to\cdots\to [n_k]$ is called $f$, then the
  middle-segments construction gives us an active map $\alpha_f \colon
  [2k{+}1] \actto [n_k]$, and $\lambda_Y$ is given in the $f$-summand by
  $$
 Y_{n_k} \stackrel{\alpha_f\upperstar} \longrightarrow Y_{2k+1} .
  $$
  Altogether, the square we want to show is a pullback (in degree $k$)
  is identified 
  with
  \[
  \begin{tikzcd}
  \sum_f Y_{n_k} \ar[r, "\alpha_f\upperstar"] \ar[d] & Y_{2k+1} \ar[d, "F_{2k+1}"]  \\
  \sum_f X_{n_k} \ar[r, "\alpha_f\upperstar"'] & X_{2k+1}
  \end{tikzcd}
  \]
 where the horizontal maps on each summand depend on $f\in (\nerve \simplexcategory)_k$.
  The left vertical map respects $f$ (since $\Nel(F)$ is a morphism of 
  right fibrations over $\nerve\simplexcategory$).
  Therefore, since pullbacks commute with sums, the
  pullback property  can be established separately for each summand. 
  For a fixed chain $f$, the square is thus
  \[
  \begin{tikzcd}
  Y_{n_k} \ar[r, "\alpha_f\upperstar"] \ar[d, "F_{n_k}"'] & Y_{2k+1} \ar[d, "F_{2k+1}"]  \\
  X_{n_k} \ar[r, "\alpha_f\upperstar"'] & X_{2k+1} ,
  \end{tikzcd}
  \]
  and this square is a pullback since $\alpha_f$ is active and $F$ is culf.
\end{proof}

%%%%%%%%%%%%%%%%%%%%%%%%%%%%%%%%%%%%%%%%%%%%%%%%%%
\section{Culfy and righteous maps}
%%%%%%%%%%%%%%%%%%%%%%%%%%%%%%%%%%%%%%%%%%%%%%%%%%
\label{sec:culfy}

For the proof of \cref{untwisting theorem} below, it is helpful to recast the notions of culf maps and right fibrations from the $\infty$-category of simplicial spaces $\PrSh(\simplexcategory)$ to the equivalent $\infty$-category $\RFib(\simplexcategory)$ of right fibrations over $\simplexcategory$.

\begin{blanko}[Categories of elements]
Recall that for a presheaf $X\colon \CC\op\to\spaces$, the category of elements $\el(X) \coloneqq \CC \comma X$
is the $\infty$-category which 
has as objects and arrows, respectively, diagrams of the form
$$
\begin{tikzcd}
m \ar[d, "\tau"]  \\
X 
\end{tikzcd}
\qquad\qquad
\begin{tikzcd}[column sep={2.5em,between origins}]
m' \ar[rd, "\tau'"'] \ar[rr, "\alpha"] && m \ar[ld, "\tau"]  \\
& X   &
\end{tikzcd}
$$
with $m$ and $m'$ objects in $\CC$ (that is, representables).
(We shall be interested
in cases when $\CC$ is a subcategory of $\simplexcategory$, so we use letters such as $n$ to
denote the objects in $\CC$.)
\end{blanko}

\begin{lemma}
  A natural transformation 
$$
\begin{tikzcd}[column sep={5em,between origins}]
  \CC\op \ar[r, bend left, pos=0.45, "Y"] \ar[r, bend right, pos=0.45, "X"'] \ar[r, 
phantom, "\text{\scriptsize $\Downarrow f$}"]   & \spaces
\end{tikzcd}
$$
is cartesian if and only if the corresponding right fibration
$$
\el(Y) \stackrel{p}\longrightarrow \el(X)
$$
is also a left fibration.
  
\end{lemma}

\begin{proof}
To say that $p \colon \el (Y) \to \el(X)$ is a left fibration means that given
an object $\tau \in \el(Y)$ and an arrow $\theta\colon p\tau \to \sigma$,
there is a ``unique'' lift $\tau \to \gamma$.
See for example Riehl--Verity~\cite[Proposition~5.5.6]{RiehlVerity:EICT}.
Uniqueness means that there is a contractible space of such lifts, as we now detail. 
The arrow $\theta\colon p\tau \to \sigma$ 
amounts to the solid square in the diagram
\[
\begin{tikzcd}[sep={5em,between origins}]
m \ar[d, "\tau"'] \ar[r, "\alpha"] 
& n \ar[d, "\sigma"] \ar[ld, dotted, "\gamma"] \\
Y \ar[r, "f"'] & X.
\end{tikzcd}
\]
The space of lifts of $\theta$ is the space of fillers $\gamma \colon n \to Y$ in this square.
Contractibility of the space of fillers for each such $\theta$ (with $\alpha$ fixed) means 
that the following diagram of spaces is a pullback:
\[
\begin{tikzcd}[column sep={10em,between origins}, row sep={4em,between origins}]
\Map(n,Y) \ar[d, "\text{post }f"'] \ar[r, "\text{pre }\alpha"] 
\drpullback & \Map(m,Y) \ar[d, "\text{post }f"]  \\
\Map(n,X) \ar[r, "\text{pre }\alpha"'] & \Map(m,X). 
\end{tikzcd}
\]
But this square being a pullback for every $\alpha$ precisely means that $f\colon Y \Rightarrow X$ is cartesian.
\end{proof}

\begin{cor}\label{cor:culfy}
A simplicial map 
$$
\begin{tikzcd}[column sep={5em,between origins}]
  \simplexcategory\op \ar[r, bend left, pos=0.43, "Y"] \ar[r, bend right, pos=0.45, "X"'] \ar[r, 
phantom, "\text{\scriptsize $\Downarrow f$}"]  & \spaces
\end{tikzcd}
$$
is culf if and only if the corresponding right fibration $p$
$$
\el(Y) \stackrel{p}\longrightarrow \el(X) \longrightarrow \simplexcategory
$$
becomes a left fibration after restriction to $\simplexcategory_{\operatorname{act}} \subset 
\simplexcategory$,
the subcategory of active maps.
\end{cor}

\bigskip

We shall call such functors $p$ {\em culfy maps} of right fibrations over
$\simplexcategory$. We also use the term culfy map for the corresponding
notion for right fibrations of simplicial spaces over
$\nerve\simplexcategory$. That is, a culfy map between right fibrations
over $\nerve \simplexcategory$ is one whose pullback along
$\nerve\simplexcategory_{\operatorname{act}} \to
\nerve\simplexcategory$ is a left fibration as well. 
We define full subcategories $\kat{culfy}_{\simplexcategory}(\el X) \subseteq \RFib_{\simplexcategory}(\el X)$ and 
$\kat{Culfy}_{\nerve\simplexcategory}(\Nel X)\subseteq  
\RFIB_{\nerve\simplexcategory}(\Nel X)$ having the culfy maps as objects, so that the dashed arrows in the diagram below become equivalences by definition.
(Note that the middle row of the diagram concerns
right fibrations between $\infty$-categories, rather than between
simplicial spaces.)
\[
\begin{tikzcd}
\culf(X) \rar[hook] \dar[dashed,"\simeq"] & \simpspaces_{/X} \dar["\simeq"] &[-0.4cm] Y \to X \dar[mapsto] \\
\kat{culfy}_{\simplexcategory}(\el X) \rar[hook] \dar[dashed,"\simeq"] & \RFib_{\simplexcategory}(\el X) \dar["\nerve"', "\simeq"] & \el (Y) \to \el (X) \to \simplexcategory \dar[mapsto] \\
\kat{Culfy}_{\nerve\simplexcategory}(\Nel X) \rar[hook]  & 
\RFIB_{\nerve\simplexcategory}(\Nel X)  & \Nel (Y) \to \Nel (X) \to 
\nerve\simplexcategory  .
\end{tikzcd}
\]

\begin{cor}\label{cor righteous maps}
  A simplicial map
$$
\begin{tikzcd}[column sep={5em,between origins}]
  \simplexcategory\op \ar[r, bend left, pos=0.43, "Y"] \ar[r, bend right, pos=0.45, "X"'] \ar[r, phantom, "\text{\scriptsize $\Downarrow f$}"]  & \spaces
\end{tikzcd}
$$
is a right fibration if and only if the corresponding right fibration $p \coloneq \el(f)$:
$$
\el(Y) \stackrel{p}\longrightarrow \el(X) \longrightarrow \simplexcategory
$$
becomes a left fibration after restriction to $\simplexcategory^t \subset 
\simplexcategory$, the subcategory of \toppreserving\ monotone maps (that is, 
presheaves on $\simplexcategory^t$ are ``simplicial objects with missing top face maps'' --- see \ref{dec}).
\end{cor}

We call such functors $p$ {\em righteous maps} of right fibrations
over $\simplexcategory$.  
We also use the term righteous map for the 
corresponding notion for right fibrations of simplicial spaces over 
$\nerve\simplexcategory$. 
The following diagram summarizes the 
notions, where $\kat{righteous}_{\simplexcategory}(\el X)  \subseteq \RFib_{\simplexcategory}(\el X)$ and $\kat{Righteous}_{\nerve\simplexcategory}(\Nel X) \subseteq 
\RFIB_{\nerve\simplexcategory}(\Nel X)$ are the full subcategories of righteous maps. The middle row concerns right fibrations of $\infty$-categories 
rather than of simplicial spaces. 

\[
\begin{tikzcd}
\RFIB(X) \rar[hook] \dar[dashed,"\simeq"] & \simpspaces_{/X} \dar["\simeq"] &[-0.4cm] Y \to X \dar[mapsto] \\
\kat{righteous}_{\simplexcategory}(\el X) \rar[hook] \dar[dashed,"\simeq"] 
& \RFib_{\simplexcategory}(\el X) \dar["\nerve"', "\simeq"] & \el (Y) \to 
\el (X) \to \simplexcategory \dar[mapsto] \\
\kat{Righteous}_{\nerve\simplexcategory}(\Nel X) \rar[hook]  & 
\RFIB_{\nerve\simplexcategory}(\Nel X)  & \Nel (Y) \to \Nel (X) \to 
\nerve\simplexcategory  .
\end{tikzcd}
\]

%%%%%%%%%%%%%%%%%%%%%%%%%%%%%%%%%%%%%%%%%%%%%%%%%%
\section{Main theorem via pullback along \texorpdfstring{$\lambda$}{λ}}\label{sec untwist lambda}
%%%%%%%%%%%%%%%%%%%%%%%%%%%%%%%%%%%%%%%%%%%%%%%%%%

We know by Lemma~\ref{culf<=>rfib} that the edgewise subdivision of a 
culf map is a right fibration. The following main theorem gives an inverse 
construction.

\begin{theorem}[\cref{thm_c}]\label{untwisting theorem}
  For $X$ a simplicial space, the functor
  \begin{eqnarray*}
    \sd_{X} \colon \culf(X) & \longrightarrow & \simprfib(\sd X)  \\
    Y{\to}X & \longmapsto & \sd(Y){\to}\!\sd(X)
  \end{eqnarray*}
  is an equivalence. The inverse equivalence is given essentially by 
  pullback along $\lambda_X \colon \Nel(X) \to \sd(X)$, as detailed in the proof.
\end{theorem}

Henceforth, subscripts on functors, such as $\sd_{X}$, indicate the
functors induced on slices (or subcategories of slices).

\begin{proof}[Proof of \cref{untwisting theorem}]
We have the following commutative triangles:
  \begin{equation}\label{eq diag diagram}
  \begin{tikzcd}[row sep={5.5em,between origins}, column sep={4.3em,between origins}]
  \culf(X) \ar[rrrrr, "\sd_{X}"] \ar[d, "\Nel_{X}"'] &&&&& 
  \simprfib(\sd X)
  \ar[d, "\Nel_{\sd X}=(\last_{\sd X})\upperstar"]  \ar[llllld, "\lambda_X\upperstar"'] 
  \\
  \simprfib(\Nel(X)) \ar[rrr, "Q\upperstar"'] &&& 
  \simprfib(Q\upperstar \Nel X) \ar[rr, phantom, "\simeq" description]
 && \simprfib(\Nel \sd X)   .
  \end{tikzcd}
  \end{equation}
  The upper left triangle commutes because of 
  Lemma~\ref{lemma:lambdaculf}, which says that applying $\sd$ and 
  then pulling back along $\lambda$ is the same as applying $\Nel$.
  
  The lower right triangle involves some canonical 
  identifications, first of all $Q\upperstar( \Nel X) \simeq \Nel( 
\sd X)$. Note further that since $\last$ is cartesian 
on right fibrations (Lemma~\ref{el-L}), the pullback square of Lemma~\ref{el-L} gives 
the identification $\Nel_{\sd X}=(\last_{\sd X})\upperstar$ indicated.
The triangle now commutes by Lemma~\ref{lambda-omega}.

We know from Lemma~\ref{culf<=>rfib} that the functor
$Q\upperstar \colon \simprfib(\Nel X) \to \simprfib(\Nel \sd X)$ restricts to a functor 
$Q\upperstar\colon \simplculfy_{\nerve\simplexcategory}(\Nel X) \to 
\simplrighteous_{\nerve\simplexcategory}(\Nel \sd X)$. 
Here, as in \cref{sec:culfy}, $\simplculfy_{\nerve\simplexcategory}(\Nel X) 
\subset \simprfib_{\nerve\simplexcategory}(\Nel X)$ is the $\infty$-category of maps of right fibrations $Y \to \Nel(X) \to \nerve\simplexcategory$ that become also left fibrations after base change along $\nerve\simplexcategory_{\operatorname{act}}\subset \nerve\simplexcategory$.
Similarly, 
\[
  \simplrighteous_{\nerve\simplexcategory}(\Nel \sd X) \subset 
  \simprfib_{\nerve\simplexcategory}(\Nel\sd X)
\] 
is the $\infty$-category of maps of right fibrations $Y \to \Nel(\sd X) \to \nerve\simplexcategory$ that become also left fibrations after base change along $\nerve\simplexcategory^t \subset \nerve\simplexcategory$.
The point is now that all the 
downgoing maps in \eqref{eq diag diagram} --- including 
  $\lambda_X\upperstar$ --- actually land in these subcategories of culfy or righteous maps, 
  as indicated in the diagram
  \begin{equation}\label{diag culfy righteous}
  \begin{tikzcd}[row sep=2.5em]
  \culf(X) \ar[r, "\sd_X"] \ar[d, "\Nel_X"'] & 
  \simprfib(\sd X)
  \ar[d, "\Nel_{\sd X}=(\last_{\sd X})\upperstar"]  \ar[ld, "\lambda_X\upperstar"'] \\
  \simplculfy_{\nerve\simplexcategory}(\Nel X) \ar[r, "Q\upperstar"'] 
  & 
  \simplrighteous_{\nerve\simplexcategory}(\Nel \sd X) . 
  \end{tikzcd}
  \end{equation}
We already know that the vertical maps $\Nel$ land in these subcategories and are equivalences (by definition of these 
  classes of maps).

  We have to check that $\lambda\upperstar$ 
  lands in culfy maps. Let $Y \to \sd(X)$ be a right fibration. We 
  need to check that after pullback along $\lambda$ and restriction to 
  $\simplexcategory_{\operatorname{act}}$ the result is a left fibration. 
  The relevant diagram is
  \[
  \begin{tikzcd}
  \Delta^0 \ar[d, "d^1"'] \ar[r, "\name{y}"] & \cdot \drpullback \ar[d] \ar[r] & \cdot 
  \drpullback \ar[d] \ar[r] & Y \ar[d]  
  \\
  \Delta^1 \ar[r, "\name{e}"'] \ar[ru, dotted, "?"'] &\Nel (X)_{\mid \operatorname{act}} \drpullback 
  \ar[d] \ar[r] & \Nel (X)
  \ar[d] \ar[r, "\lambda_X"'] & \sd(X)  .
  \\
  & \nerve \simplexcategory_{\operatorname{act}} \ar[r] & \nerve \simplexcategory &
  \end{tikzcd}
  \]
  By \cref{lem rfib creates segal}, all objects in the second column are Segal spaces; by the dual of \cref{lem rfib between segal spaces}, to prove the map of interest is a left fibration it is enough to prove it is orthogonal to $d^1 \colon \Delta^0 \to \Delta^1$. 
  So we need to show, for any maps $\name{y}$ and $\name{e}$ making the leftmost 
  square commute, that the space of lifts indicated is contractible.
As the two squares in the upper right are given by pullback, this space of lifts is equivalent 
  to the space of lifts of the long composite rectangle. But since 
  the arrow $e$ in $\Nel (X)$ is active, $\lambda_X$ takes it to a 
  degenerate $1$-simplex in 
  $\sd(X)$ (by Lemma~\ref{activeequiv}), so it lifts uniquely, since 
  $Y \to \sd(X)$ is a right fibration.
  
  Finally in \eqref{diag culfy righteous}, since the vertical maps are
  equivalences, we see that $\lambda_X\upperstar$
  has up-to-homotopy inverses on both sides, so all maps in \eqref{diag culfy righteous} are equivalences.
  In particular, $\sd_X$ is an equivalence.
\end{proof}

%%%%%%%%%%%%%%%%%%%%%%%%%%%%%%
\section{Main theorem via right Kan extension}
\label{sec main via rke}
%%%%%%%%%%%%%%%%%%%%%%%%%%%%%%

In this section we give a very different description of the inverse 
featured in the main theorem,
in terms of the right adjoint to edgewise subdivision. In a sense this
description is much more direct than the pullback-along-$\lambda$
description given in \cref{untwisting theorem}, but in
practice the right adjoint is not easy to compute.

\bigskip

The edgewise subdivision of an ordinary category $\mathcal{C}$ is just the
twisted arrow category, so its objects are the arrows of $\mathcal{C}$. In
the special case $\mathcal{C}= \Delta^n$, which is a poset, an arrow is
completely specified by its endpoints, so we denote by $(i,j)$
the object
in $\sd(\Delta^n)$ corresponding to the arrow $i{\to}j$ in $\Delta^n$.
With this convention, $Q\upperstar(\Delta^3) = \sd(\Delta^3)$ may be visualized as in the following picture
\[\begin{tikzcd}[row sep=tiny, column sep=small]
(0,0) \drar & & (1,1)\drar \dlar & & (2,2) \dlar \drar & & (3,3) \dlar \\
& (0,1) \drar & & (1,2) \dlar \drar & & (2,3) \dlar \\
& & (0,2) \drar & & (1,3) \dlar\\
&& & (0,3),
\end{tikzcd}\]
and the pictures for other representables are similar.

\bigskip

We will arrive shortly at the $Q\upperstar \isleftadjointto Q\lowerstar$ 
  adjunction, but first we need a few preliminary results on the 
  $Q\lowershriek \isleftadjointto Q\upperstar$ adjunction.
  
\begin{lemma}
  The unit for the $Q\lowershriek \isleftadjointto Q\upperstar$ adjunction
  is given on representables by
  \begin{eqnarray*}
    \eta_{\Delta^n} \colon \Delta^n & \longrightarrow & Q\upperstar 
  Q\lowershriek \Delta^n = \sd(\Delta^{2n+1})  \\
    i & \longmapsto & (n{-}i,n{+}i{+}1) .
  \end{eqnarray*}
\end{lemma}

\begin{proof}
  For a general simplicial space $X$, the adjunction equivalence 
  $\Map(Q\lowershriek \Delta^n, X) \allowbreak \isopil \Map(\Delta^n, Q\upperstar X)$ acts 
  by sending a $(2n+1)$-simplex in $X$ to the {\em same} simplex, but 
  reinterpreted as an $n$-simplex in $\sd(X)$. Now instantiate at $X= 
  Q\lowershriek \Delta^n$, and consider the identity map 
  $\id \colon Q\lowershriek \Delta^n \to Q\lowershriek \Delta^n$. The unit
  will be the corresponding map under the adjunction equivalence. This is now the
  $n$-simplex
    \begin{equation}
  \begin{tikzcd}[cramped, row sep=small, column sep=small]
    n{+}1\ar[r]&n{+}2\ar[r]&\cdots\ar[r]&2n{+}1\\
    n\ar[u]&n{-}1\ar[l]&\cdots\ar[l]&0 ,
    \ar[l]
  \end{tikzcd}
  \end{equation}
  where clearly the $i$th vertex is the arrow 
  $(n{-}i,n{+}i{+}1)$.
\end{proof}

\begin{cor}\label{cor counit final}
  The unit $\eta_{\Delta^n} \colon \Delta^n \to Q\upperstar 
  Q\lowershriek \Delta^n$ is final.
\end{cor}
\begin{proof}
  It sends the terminal object $n\in \Delta^n$ to the terminal object
  $(0,2n{+}1) \in \sd(\Delta^{2n+1})$.
  By \cref{lem: terminal objects}, $\eta_{\Delta^n}$ is final.
\end{proof}

\begin{prop}\label{prop counit ambifinal} The counit $\varepsilon_{\Delta^n} \colon Q\lowershriek Q\upperstar \Delta^n \to \Delta^n$ is ambifinal.
\end{prop}
\begin{proof}
By \cref{final<=>ambifinal}, we know that $Q\lowershriek$ sends final maps to ambifinal maps; 
combining this with \cref{cor counit final}, we see that $Q\lowershriek (\eta_{\Delta^n})$ is ambifinal.
A triangle identity gives that $\varepsilon_{Q\lowershriek \Delta^n} \circ Q\lowershriek (\eta_{\Delta^n}) \simeq \id_{Q\lowershriek\Delta^n}$, so by right cancellation, $\varepsilon_{Q\lowershriek \Delta^n} = \varepsilon_{\Delta^{2n+1}}$ is ambifinal.
We thus have the result for odd-dimensional simplices.
But $\Delta^{2n}$ is a retract of $\Delta^{2n+1}$, and the left 
class in any factorization system is closed under retracts, 
so $\varepsilon_{\Delta^{2n}}$ is ambifinal as well.
\end{proof}

We now come to the right adjoint $Q\lowerstar$ to edgewise subdivision 
$Q\upperstar$.

\begin{prop}\label{prop rke rfib to culf} 
The right Kan extension functor $Q\lowerstar \colon \simpspaces \to \simpspaces$ takes right fibrations to culf maps. 
\end{prop}
\begin{proof}
  Let $p$ be a right fibration; by the definition of culf \ref{culf}, we wish to show that $a \perp Q\lowerstar (p)$ for every active map $a \colon \Delta^n \actto
  \Delta^m$. 
  By behavior of liftings with adjunctions \ref{factorization systems}, this is equivalent to $Q\upperstar(a) \perp p$.
  But this holds by \cref{comprehensive} since $Q\upperstar(a)$ is last-point-preserving, hence final (\cref{lem: terminal objects}).
\end{proof}

The following general result should be well known to experts, but we could 
not find a suitable reference for it.
Note that
the relationship between $\varepsilon$ and $\eta'$ is not dual (via taking opposite $\infty$-categories) to the relationship between $\varepsilon'$ and $\eta$.

\begin{lemma}\label{lem adjoint string relation}
Let $\CC$ and $\DD$ be $\infty$-categories and
\[
\begin{tikzcd}[column sep=large]
    \DD  \rar["F" description] 
    \rar[phantom, bend right=18, "\scriptscriptstyle\perp"]  % bottom one
    \rar[phantom, bend left=18, "\scriptscriptstyle\perp"] % top one
    & \CC
    \lar[bend right=30, "L" swap] 
    \lar[bend left=30, "R"] 
\end{tikzcd}
\]
a string of adjoint functors. 
Then the units and counits
\begin{align*}
\eta & \colon \id_\CC \Rightarrow FL & \eta' & \colon \id_\DD \Rightarrow RF \\
\varepsilon & \colon LF \Rightarrow \id_\DD &  \varepsilon' & \colon FR 
\Rightarrow \id_\CC 
\end{align*}
are related by the following:
\[ 
\begin{tikzcd}
  \Map(b,FRc) \dar["{\Map (b,\varepsilon'_c)}"'] \rar{\sim} & \Map(Lb, Rc) \dar{\sim}  
& 
  \Map(d,e) \rar["{\Map(\varepsilon_d,e)}"] \dar["{\Map(d,\eta'_e)}"'] & \Map(LFd,e)
  \\
  \Map(b,c) & \lar["{\Map (\eta_b,c)}"] \Map(FLb, c)
& 
  \Map(d,RFe) \rar["\sim"'] & \Map(Fd,Fe) \uar["\sim"']
\end{tikzcd} 
\qquad
\]
where $b,c \in \CC$ and $d,e\in \DD$.
\end{lemma}

\begin{proof}
   For clarity, we temporarily write $[-,-]$ instead of $\Map(-,-)$.
   The first square in the statement commutes because it arises from the 
   pasting diagram
\[ 
\begin{tikzcd}
{[b,FRc]} \rar{L} \ar[ddr, bend right=10,"\eta_{FRc} \circ \text{-}" near start] \ar[dd,"\id"] & 
{[Lb,LFRc]}  \rar["\varepsilon_{Rc} \circ \text{-}"] \dar["F"] & {[Lb,Rc]}  \ar[dd,"F"] \\
& {[FLb,FLFRc]} \dar["\text{-}\circ \eta_b"] \ar[dr,"F\varepsilon_{Rc} \circ \text{-}"] \\
{[b,FRc]} \dar["\varepsilon_c' \circ \text{-}"] & {[b,FLFRc]} \lar["F\varepsilon_{Rc} \circ \text{-}"'] & {[FLb,FRc]} \dar["\varepsilon_c' \circ \text{-}"] \ar[ll, bend left=15,"\text{-}\circ \eta_b"] \\
{[b,c]} & & {[FLb,c]} . \ar[ll,"\text{-}\circ \eta_b"]
\end{tikzcd}
\]
One of the triangles in the top left is a triangle identity, and the other commutes since $\eta$ is a natural transformation $\id_\CC \Rightarrow FL$.
The square in the upper left uses that $F$ is a functor, while the other two commute by associativity of composition.

The second square in the statement commutes because it arises from the 
pasting diagram
\[ 
\begin{tikzcd}
{[d,e]} \ar[ddd,"\eta'_e \circ \text{-}"']  \ar[rr, "\text{-} \circ \varepsilon_d"] \ar[dr,"LF"] & & {[LFd,e]} \\
& {[LFd,LFe]} \rar["\id"] \dar["LF\eta'_e \circ \text{-}" swap] & {[LFd,LFe]} \uar["\varepsilon_e\circ \text{-}"'] \\
& {[LFd,LFRFe]} \ar[ur,"L\varepsilon'_{Fe} \circ \text{-}" swap] \\
{[d,RFe]} \rar["F"] & {[Fd,FRFe]} \uar["L"] \rar["\varepsilon'_{Fe} \circ 
\text{-}"] & {[Fd,Fe]} . \ar[uu,"L"']
\end{tikzcd}
\]
The top region commutes because $\varepsilon$ is a natural transformation $LR \Rightarrow \id_\CC$.
The bottom left (resp.\ bottom right) region commutes since $LF$ (resp.\ $L$) is a functor.
The triangle is $L$ applied to the triangle identity of $L\isleftadjointto F$.
\end{proof}

\begin{cor}\label{cor expectation 4}   The unit $\eta'$ of the $Q\upperstar \isleftadjointto Q\lowerstar$ 
  adjunction
  can be computed in simplicial degree $n$ as
  $$
  \Map(\varepsilon_{\Delta^n},X) \colon \Map(\Delta^n,X) \to \Map(Q\lowershriek Q\upperstar \Delta^n, X) = \Map(\Delta^n, Q\lowerstar Q\upperstar X) .
  $$
\end{cor}

\begin{lemma}\label{lem unit cartesian on culf}
  The unit $\eta'$ of the $Q\upperstar \dashv Q\lowerstar$ adjunction is cartesian on culf maps.
\end{lemma}

\begin{proof}
Let $p \colon Y \to X$ be culf. 
By \cref{culf<=>rfib} and \cref{prop rke rfib to culf}, $Q\lowerstar Q\upperstar (p)$ is culf as well.
We will show that the diagram
\[ \begin{tikzcd}
Y \dar[swap]{p} \rar{\eta'_Y} & Q\lowerstar Q\upperstar Y  \dar{Q\lowerstar Q\upperstar p}  \\
X \rar[swap]{\eta'_X} & Q\lowerstar Q\upperstar X
\end{tikzcd} \]
is a pullback in each simplicial degree $n$.
Using \cref{cor expectation 4}, this becomes
\[ \begin{tikzcd}[column sep={13em,between origins}]
\Map(\Delta^n, Y) \dar \rar{\Map(\varepsilon_{\Delta^n}, Y)} & 
\Map(Q\lowershriek Q\upperstar \Delta^n, Y) \dar \\
\Map(\Delta^n, X) \rar[swap]{\Map(\varepsilon_{\Delta^n}, X)} & 
\Map(Q\lowershriek Q\upperstar \Delta^n, X) ,
\end{tikzcd} \]
where $\varepsilon$ is the counit for the $Q\lowershriek  \dashv Q\upperstar$ adjunction.
By \cref{prop counit ambifinal}, $\varepsilon_{\Delta^n}$ is ambifinal, so this square is a pullback.
\end{proof}

\begin{lemma}\label{lem counit cart on rfib}
  The counit $\varepsilon'$ of the $Q\upperstar \dashv Q\lowerstar$
  adjunction is cartesian on right fibrations.
\end{lemma}
\begin{proof}
  Let $f\colon Y \to X$ be a right fibration. We show that the diagram
  \[ \begin{tikzcd}
  Q\upperstar Q\lowerstar Y \dar[swap]{Q\upperstar Q\lowerstar (f)} \ar[r, 
  "\varepsilon'_Y"] & Y \ar[d, "f"] \\
  Q\upperstar Q\lowerstar X \rar[swap]{\varepsilon'_X}  & X
  \end{tikzcd} \]
  is a pullback in each simplicial degree $n$. Using \cref{lem adjoint
  string relation}, this becomes
  \[ 
  \begin{tikzcd}
\Map(\Delta^n,Q\upperstar Q\lowerstar Y) 
  \rar{\sim} \dar[swap]{\Map(\Delta^n,Q\upperstar Q\lowerstar (f))}  
&[-0.4cm]
\Map(Q\upperstar Q\lowershriek  \Delta^n,Y) 
   \rar{\Map(\eta_{\Delta^n}, Y)} \dar["{\Map(Q\upperstar Q\lowershriek \Delta^n,f)}"]
&[+1.2cm]
\Map(\Delta^n,Y) 
  \dar["{\Map(\Delta^n,f)}"] 
\\
\Map(\Delta^n, Q\upperstar Q\lowerstar X) 
  \rar{\sim} 
& 
\Map(Q\upperstar Q\lowershriek \Delta^n, X) \rar[swap]{\Map(\eta_{\Delta^n}, X)} 
& 
\Map(\Delta^n, X) ,
\end{tikzcd}
\]
  where $\eta$ is the unit for the $Q\lowershriek \dashv Q\upperstar$
  adjunction. Since $\eta_{\Delta^n}$ is final by \cref{cor counit final}
  and $f$ is a right fibration, this is a pullback.
\end{proof}

\begin{blanko}[Slicing adjunctions]\label{slicing adj}
Recall that given an adjunction
\[
\begin{tikzcd}
\DD \rar[bend left, "F"] \ar[r,phantom,"\scriptstyle\perp"] & \CC \lar[bend left, "G"] 
\end{tikzcd}
\]
where $\DD$ has pullbacks, we can obtain an adjunction
\[
\begin{tikzcd}
\DD_{/d}
\rar[bend left, "F_d", start anchor = {[yshift=1.5ex]east}, end anchor = {[yshift=1.5ex]west}] 
\rar[phantom,"\scriptstyle\perp"] & 
\CC_{/Fd} 
\lar[bend left, start anchor = {[yshift=-1.5ex]west}, end anchor = {[yshift=-1.5ex]east}] 
\end{tikzcd}
\]
whose right adjoint is given by applying $G_{Fd} \colon \CC_{/Fd} \to
\DD_{/GFd}$ and then pulling back along the unit $\eta_d \colon d \to GFd$
(see \cite[Proposition 5.2.5.1]{HTT}). 

Instantiating to the $Q\upperstar \dashv Q\lowerstar$
adjunction, we get the sliced adjunction
\[
\begin{tikzcd}
\simpspaces_{/X} 
\rar[bend left, "\sd_X", start anchor = {[yshift=1.5ex]east}, end anchor = {[yshift=1.5ex]west}] 
\rar[phantom,"\scriptstyle\perp"] & 
\simpspaces_{/\sd(X)}  .
\lar[bend left, "(\eta'_X)\upperstar \circ Q\lowerstar", start anchor = {[yshift=-1.5ex]west}, end anchor = {[yshift=-1.5ex]east}] 
\end{tikzcd}
\]
By \cref{culf<=>rfib} and \cref{prop rke rfib to culf}, this adjunction 
restricts to an adjunction
\begin{equation}\label{eq culf rfib}
\begin{tikzcd}
\culf(X) 
\rar[bend left, "\sd_X", start anchor = {[yshift=1.5ex]east}, end anchor = {[yshift=1.5ex]west}] 
\rar[phantom,"\scriptstyle\perp"] & 
\simprfib(\sd X)  .
\lar[bend left, "(\eta'_X)\upperstar \circ Q\lowerstar", start anchor = {[yshift=-1.5ex]west}, end anchor = {[yshift=-1.5ex]east}] 
\end{tikzcd}
\end{equation}
In detail, the right adjoint acts on a
given right fibration $W \to Q\upperstar X$ by first
applying $Q\lowerstar$
to get a culf map $Q\lowerstar W \to Q\lowerstar Q\upperstar X$ (by 
Proposition~\ref{prop rke rfib to culf}), and then
pulling back along the unit $\eta_X'$ of the $Q\upperstar \dashv Q\lowerstar$ 
adjunction to get a culf map $Y \to X$ as in 
\[ \begin{tikzcd}
Y \rar \dar \drpullback & Q\lowerstar W \dar \\
X \rar[swap]{\eta_X'} & Q\lowerstar Q\upperstar X .
\end{tikzcd} \]
\end{blanko}

The next two lemmas together show that the adjunction \eqref{eq culf rfib} is an adjoint 
equivalence, so that in particular the functor $(\eta'_X)\upperstar \circ 
Q\lowerstar$ is an alternative description of the inverse to
the 
equivalence displayed in \cref{thm_c}.

\begin{lemma}\label{lem composition first way}
The counit
\[
\begin{tikzcd}
\simprfib(\sd X) \rar{Q\lowerstar} \ar[rrr, bend right=15,"\id"'] \ar[rrr, phantom, bend right=8, "\Downarrow"] & \culf(Q\lowerstar Q\upperstar X) \rar{(\eta'_X)\upperstar}  & \culf(X) \rar{Q\upperstar} & \simprfib(\sd X)
\end{tikzcd}
\]
of the adjunction \eqref{eq culf rfib} is an equivalence.
\end{lemma}

\begin{proof}
Let $p\colon W \to Q\upperstar X$ be a right fibration. 
Applying $Q\lowerstar$ and pulling back along $\eta'_X$ gives 
the culf map $q$ on the left in
\[ \begin{tikzcd}
Y \rar  \dar[swap]{q} \drpullback & Q\lowerstar W \ar[d, "Q\lowerstar (p)"]  \\
X \rar[swap]{\eta'_X} & Q\lowerstar Q\upperstar X.
\end{tikzcd} \]
This is sent by $Q\upperstar$ to the left square in the following diagram, which is again 
a pullback since $Q\upperstar$ is also a right adjoint.
\[ \begin{tikzcd}[column sep=large]
Q\upperstar Y  \drpullback \ar[r] \ar[d, "Q\upperstar (q)"'] &
Q\upperstar Q\lowerstar  W \drpullback \dar[swap]{Q\upperstar Q\lowerstar (p)} \rar{\varepsilon_W'} &
W \dar{p} \\
Q\upperstar X 
\rar{Q\upperstar (\eta'_X)} \ar[rr, bend right=20, "\id"'] & 
Q\upperstar Q\lowerstar Q\upperstar X \rar{\varepsilon_{Q\upperstar X}'} &
Q\upperstar X
\end{tikzcd} \]
The right square is a pullback by Lemma~\ref{lem counit cart on rfib}.
Since the large rectangle is thus a pullback, with bottom edge an identity,
it follows that the top edge $Q\upperstar  Y \to W$ is an equivalence.
Hence $Q\upperstar q \simeq p$, as desired. 
\end{proof}

\begin{lemma}
The unit
\[
\begin{tikzcd}
\culf(X) \rar[swap]{Q\upperstar} \ar[rrr, bend left=15,"\id"] \ar[rrr, phantom, bend left=8, "\Downarrow"] & \simprfib(\sd X) \rar[swap]{Q\lowerstar} & \culf(Q\lowerstar Q\upperstar X) \rar[swap]{(\eta'_X)\upperstar}  & \culf(X)
\end{tikzcd}
\]
of the adjunction \eqref{eq culf rfib} is an equivalence.
\end{lemma}
\begin{proof}
If $p\colon Y \to X$ is culf, then 
\[ \begin{tikzcd}
Y \dar[swap]{p} \rar{\eta'_Y} \drpullback & Q\lowerstar Q\upperstar Y  \dar{Q\lowerstar Q\upperstar (p)} \\
X \rar[swap]{\eta'_X} & Q\lowerstar Q\upperstar X
\end{tikzcd} \]
is a pullback by \cref{lem unit cartesian on culf}.
\end{proof}

Combining the previous two lemmas, we have established our second proof of 
\cref{thm_c}:
\begin{theorem}[\cref{thm_c}]\label{thm_c_actual_2}
The adjunction $\sd_X \colon \simpspaces_{/X} \rightleftarrows \simpspaces_{/\sd X}$ restricts to an equivalence $\culf(X) \simeq \RFib(\sd X)$.
\end{theorem}

%%%%%%%%%%%%%%%%%%%%%%%%%%%%%%
\section{Decomposition spaces and Rezk completeness}
%%%%%%%%%%%%%%%%%%%%%%%%%%%%%%

\begin{blanko}[Decomposition spaces/$2$-Segal spaces]\label{decomp}
  A {\em decomposition space} \cite{GKT1} (or {\em $2$-Segal
  space}~\cite{Dyckerhoff-Kapranov:1212.3563}) is a simplicial 
  $\infty$-groupoid $X \colon \simplexcategory\op\to\spaces$ that takes 
  active-inert pushouts in $\simplexcategory$ to pullbacks in $\spaces$.
Namely, each span $[m] \leftarrowtail [n] \actto [p]$ in $\simplexcategory$ admits a pushout
\[ \begin{tikzcd}
{[n]} \rar[-act] \dar[tail] \drar[phantom, "\ulcorner" very near end] & {[p]} \dar[tail] \\
{[m]} \rar[-act] & {[q]} ,
\end{tikzcd} \]
and each such pushout is sent to a pullback of $\infty$-groupoids.
\end{blanko}

\begin{lemma}[Bergner et al.~\cite{BOORS:Edgewise}]
\label{BOORS:Edgewise}
  A simplicial space $X$ is a decomposition space if and only if $\sd(X)$ is a Segal space.
\end{lemma}

The following is \cref{thm_d} from the introduction.
\begin{theorem}[\cref{thm_d}]\label{thm_d_actual}
  The $\infty$-category of decomposition spaces and culf maps is 
  locally an $\infty$-topos. 
  More precisely, for $X$ a 
  decomposition space, we have an equivalence
  $$
  \Decomp_{/X} \simeq \simprfib(\sd X) \simeq \simprfib(\widehat{\sd X}) \simeq  \PrSh(\widehat {\sd X}).
  $$
\end{theorem}

Here $\widehat{(-)}$ denotes the Rezk completion of a Segal space.

\begin{proof}
  The first step in the equivalence is \cref{thm_c}. The second step is
  \cref{rfib-hat} below, which says that Rezk completion of Segal spaces
  does not affect right fibrations --- note that $\sd(X)$ is a Segal space
  since $X$ is a decomposition space. The last step is
  straightening/unstraightening for complete Segal spaces. 
\end{proof}

\begin{prop}\label{rfib-hat}
  Suppose $X$ is a Segal space.
  Then pulling back along the completion map $X \to \widehat{X}$ 
  induces an equivalence $\RFIB(\widehat{X}) \to \RFIB(X)$.
\end{prop}

This is proved in the appendix.

\bigskip

We finish the paper with the observation (\cref{prop twist rezk}) that in many cases it is not necessary to Rezk complete, namely when the decomposition space itself is Rezk complete, as is usually the case for
decomposition spaces of combinatorial origin (e.g.\ M\"obius decomposition spaces \cite[Corollary 8.7]{GKT2}). For this we first need a few results about Rezk completeness for decomposition spaces.

\begin{blanko}[Equivalences]
  Let $X$ be a decomposition space. An arrow $f\in X_1$ is called an
  {\em equivalence} if there exists $\sigma \in X_2$ such that
  $d_2(\sigma) = f$ and $d_1(\sigma) = s_0 d_1 (f)$ and there exists
  $\tau\in X_2$ such that $d_0(\tau) = f$ and $d_1(\tau) = s_0 d_0 (f)$:
  \[
\begin{tikzcd}[column sep={2.7em,between origins}]
 & y \ar[rd,"g"] &   \\
x \ar[ru, "f"] 
\ar[rr, "\id_x"'] 
&\ar[u, phantom, "\sigma" description] & x
\end{tikzcd}
\qquad
\begin{tikzcd}[column sep={2.7em,between origins}]
 y \ar[rr, "\id_y"]
 \ar[rd,"h"'] 
 &\ar[d, phantom, pos=0.35, "\tau" description]& y \\
 & x \ar[ru, "f"'] &
\end{tikzcd}
\]
  We denote by $X_1^{\operatorname{eq}} \subset X_1$ the full sub
  $\infty$-groupoid spanned by the equivalences. Note that degenerate arrows are always equivalences.
\end{blanko}

\begin{blanko}[Remark]
  Feller~\cite{Feller:2204.01910} proposes a stronger notion of equivalence for $1$-simplices in a decomposition space, which are instead witnessed by maps from the nerve of the free-living isomorphism.
  For Segal spaces, this makes no difference (by \cite[Theorem~6.2]{Rezk:MHTHT}), but generally this change will result in a potentially larger class of Rezk complete decomposition spaces than the following definition from \cite[5.13]{GKT2}.
\end{blanko}

\begin{blanko}[Rezk completeness]\label{def:Rezk}
  A decomposition space $X$ is called {\em Rezk complete} when the
  canonical map $s_0 \colon X_0 \to X_1^{\operatorname{eq}}$ is a homotopy equivalence.
\end{blanko}

\begin{blanko}[Remark]\label{remark on rezk completeness}
The Rezk completeness condition as formulated here only refers 
to $1$-simplices. However, since decomposition spaces have the 
property that degeneracy can be detected on principal edges \cite[\S2]{GKT2},
the $1$-dimensional condition implies other conditions  corresponding to degeneracies. We will not attempt at
distilling this observation into a general statement, but only prove
the following illustrative case, which we will actually need in the proof of \cref{prop twist rezk}. Intuitively it says that 
the space of $3$-simplices whose first and third principal
  edges are equivalences is homotopy equivalent to $X_1$. 
\end{blanko}

\begin{lemma}\label{010-equivalence}
  For any Rezk complete decomposition space $X$, 
  the square 
  \[\begin{tikzcd}
    X_1 \rar{s_\bot s_\top} \dar[swap]{(s_0d_1 , \id , s_0d_0)} & X_3 
  \dar{(d_\top d_\top , d_\top d_\bot , d_\bot d_\bot)} \\
    X_1^{\operatorname{eq}} \times X_1 \times X_1^{\operatorname{eq}} \rar & X_1 \times X_1 \times X_1
  \end{tikzcd}
  \]
is a pullback.
\end{lemma}
\begin{proof}
By Rezk completeness in the sense of \ref{def:Rezk}, the square is homotopy equivalent to 
  \[\begin{tikzcd}[column sep={4cm,between origins}]
    X_1 \rar{s_\bot s_\top}\drpullback \dar[swap]{(d_1 , \id , d_0)} & X_3 
  \dar{(d_\top d_\top , d_\top d_\bot , d_\bot d_\bot)} \\
    X_0 \times X_1 \times X_0 \rar[swap]{s_0 \times \id \times s_0} &
	X_1 \times X_1 \times X_1  ,
  \end{tikzcd}
  \]
  which is a pullback by the characterization of decomposition spaces given in Proposition~6.9 (3) of \cite{GKT1}.
\end{proof}

\begin{prop}\label{prop:Rekz}
If $X$ is a Rezk complete decomposition space, and $Y \to X$ is 
  culf, then also $Y$ is a Rezk complete decomposition space.
\end{prop}

\begin{proof}
By \cite[Lemma 4.6]{GKT1}, we know that $Y$ is a decomposition space.
Further, since pullbacks of monomorphisms are monomorphisms, $s_0 \colon Y_0 \to Y_1$ is a monomorphism of $\infty$-groupoids.

Suppose $X$ is Rezk complete.
Equivalences are preserved by arbitrary maps of simplicial spaces, so $Y_1^{\operatorname{eq}} \to Y_1 \to X_1$ lands in $X_1^{\operatorname{eq}}$.
Since $[1] \actto [0]$ is active and $Y\to X$ is culf, the square in the following diagram is a pullback
\[
\begin{tikzcd}
Y_1^{\operatorname{eq}} \ar[dr,dashed,"r"] \dar[bend right=10] \ar[drr, bend left,hook] \\
X_1^{\operatorname{eq}} \ar[dr, bend right, "\simeq"'] & Y_0 \rar{s_0} \dar \drpullback & Y_1 \dar \\
& X_0 \rar[swap]{s_0} & X_1 ,
\end{tikzcd}
\]
hence we have a map $r\colon Y_1^{\operatorname{eq}} \to Y_0$ which we hope is an equivalence.
Here, $s_0r$ is equivalent to the inclusion $Y_1^{\operatorname{eq}} \to Y^1$.
We thus have the commutative diagram of $\infty$-groupoids
\[
\begin{tikzcd}[column sep=small]
Y_0  \rar \ar[dr,"s_0"'] \ar[rr, bend left=40,"\id"] & Y_1^{\operatorname{eq}} \dar["s_0r" description] \rar{r} & Y_0  \ar[dl, "s_0"] \\
& Y_1 .
\end{tikzcd}
\]
Since the non-horizontal maps are monomorphisms of $\infty$-groupoids, it follows from 
the following general
\cref{lemma about monos} that the inclusion $Y_0 \to Y_1^{\operatorname{eq}}$ is an equivalence.
\end{proof}

\begin{lemma}\label{lemma about monos}
If 
\[
\begin{tikzcd}[column sep=small]
A  \rar{k} \ar[dr,"i"'] \ar[rr, bend left=40,"\id"] & B \dar["j"] \rar{r} & A  \ar[dl, "i"] \\
& C
\end{tikzcd}
\]
is a commutative diagram of spaces with $i$ and $j$ monomorphisms (i.e.\ inclusions of unions of path components), then $k$ and $r$ are equivalences.
\end{lemma}
\begin{proof}
  First take path components of each of the spaces in question.
  Since $\pi_0(i)$ and $\pi_0(j)$ are injections of sets, we have the same is true of $\pi_0(k)$ and $\pi_0(r)$.
  Hence both $\pi_0(k)$ and $\pi_0(r)$ are bijections of sets.
  
  Now suppose $W\in \pi_0(A)$ is some path component of $A$, the element $W' \in \pi_0(B)$ is its image under $\pi_0(k)$, and $Z \in \pi_0(C)$ is the image of $W$ under $\pi_0(i)$.
  By the assumption that $i$ and $j$ are monomorphisms of $\infty$-groupoids, we have the two diagonal legs in the diagram
  \[
\begin{tikzcd}[column sep=tiny]
W \ar[dr,"i|_W"',"\rotatebox{-45}{$\simeq$}", bend right] \ar[rr,"k|_W"] & & W' \ar[dl,"j|_{W'}","\rotatebox{45}{$\simeq$}"', bend left] \\
& Z
\end{tikzcd}\]
are equivalences, hence $k|_W$ is an equivalence as well.
It follows that $k \colon A \to B$ is an equivalence, hence $r\colon B \to A$ is an equivalence.
\end{proof}

\begin{prop}\label{prop twist rezk}
  If $X$ is a Rezk complete decomposition space, then $\sd(X)$ is
  a Rezk complete Segal space.
\end{prop}

\begin{proof}
  We already know from \cref{BOORS:Edgewise} that $\sd(X)$ is a Segal space, so it remains to
  check that $\sd(X)$ is Rezk complete. So assume that $\tau\in (\sd X)_1$ is
  an equivalence, meaning that there exist $\alpha\in (\sd X)_2$
  such that $d^{\smallsd}_2(\alpha) = \tau$ and $d^{\smallsd}_1(\alpha) = 
  s^{\smallsd}_0 d^{\smallsd}_1 (\tau)$ and there exists $\beta\in(\sd X)_2$ such that
  $d^{\smallsd}_0(\beta) = \tau$ and $d^{\smallsd}_1(\beta) = 
  s^{\smallsd}_0 d^{\smallsd}_0 (\tau)$. Spelling out everything in terms of the original face
  and degeneracy maps of $X$, we have $\alpha\in X_5$ such that 
  $d_\bot d_\top (\alpha) = \tau$ and $d_{\bot+1} d_{\top-1}(\alpha) = 
  s_\bot s_\top d_{\bot} d_{\top} (\tau)$ as well as
  $\beta\in X_5$ such that 
  $d_{\bot+2} d_{\top-2} (\beta) = \tau$ and $d_{\bot+1} d_{\top-1}(\beta) = 
  s_\bot s_\top d_{\bot+1} d_{\top-1} (\tau)$. 
We make the following picture of $\tau$:
  \[
\begin{tikzcd}[row sep={15pt,between origins}]
  &\cdot\\
  \cdot \ar[ur, "v"] \\
  {} \ar[r, phantom, "\tau"] & {}\\
  \cdot \ar[uu, "a"]\\
  &\cdot \ar[ul, "u"] \ar[uuuu, "b"']
\end{tikzcd}
\]
  just to have the notation $u \coloneqq d_\top d_\top (\tau)$ and $v \coloneqq d_\bot
  d_\bot (\tau)$. 
  In order to show that $\tau$ is equivalent to an element in the image of $s_{\bot}
  s_{\top} \colon X_1 \to X_3$, 
  we should show that $u$ and $v$ are themselves equivalences; this is because \cref{010-equivalence} gives that the diagram
  \[\begin{tikzcd}
    X_1 \rar{s_\bot s_\top}\drpullback \dar[swap]{s_0d_1 \times \id \times s_0d_0} & X_3 \dar{d_\top d_\top \times d_\top d_\bot \times d_\bot d_\bot} \\
    X_1^{\operatorname{eq}} \times X_1 \times X_1^{\operatorname{eq}} \rar & X_1 \times X_1 \times X_1
  \end{tikzcd}
  \]
  is a pullback.

  The relevant 2-simplices to show that $u$ and $v$ are equivalences are extracted from $\alpha$
  and $\beta$: applying $d_\top d_\top d_\top$ to the $5$-simplices
  $\alpha$ and $\beta$ we get the required $2$-simplices for $u$, and
  applying $d_\bot d_\bot d_\bot$ to the $5$-simplices $\alpha$ and
  $\beta$ we get the required $2$-simplices for $v$. There are thus four
  cases. Just as an illustration of how the argument goes, let us
  consider $ d_\top d_\top d_\top(\alpha) $. Here is a picture of
  $\alpha$:
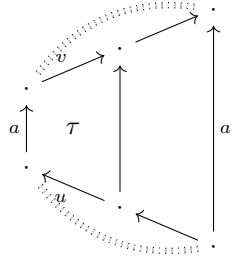
    \[
\begin{tikzcd}[row sep={15pt,between origins}]
  && \cdot \\
  &\cdot \ar[ur] &\\
  \cdot \ar[ur, "v"] \ar[uurr, dotted, equal, bend left]\\
  {} \ar[r, phantom, "\tau"] & {}\\
  \cdot \ar[uu, "a"]\\
  &\cdot \ar[ul, "u"] \ar[uuuu] \\
  && \cdot \ar[ul] \ar[uuuuuu, "a"'] \ar[uull, dotted, equal, bend left]
\end{tikzcd}
\]
The curved dotted lines illustrate the edges obtained by applying
$d_{\bot+1} d_{\top-1}$. By assumption, the resulting $3$-simplex is 
doubly degenerate. Precisely, 
$$d_{\bot+1} d_{\top-1}(\alpha) = 
  s_\bot s_\top d_{\bot} d_{\top} (\tau) = s_\bot s_\top (a).$$
  The fact that this whole $3$-simplex is doubly degenerate in this way
  implies that also the curved edges are degenerate individually.
  The lower triangle in the picture is $d_\top d_\top d_\top(\alpha) $
  and it is thus degenerate in precisely the way required to exhibit
  $u$ as an equivalence (from one side). With similar arguments applied
  to $\beta$ we see that $u$ is also an equivalence from the other side,
  and finally by Rezk completeness we can 
  therefore conclude that $u$ is degenerate. The analogous conclusion 
  for $v$ is reached using $d_\bot d_\bot d_\bot$ instead of $d_\top d_\top d_\top$.
\end{proof}

\begin{blanko}[Remark]
  In the special case where $X$ is a Rezk complete Segal space (and not just a Rezk complete decomposition space), this result was proven by Mukherjee--Rasekh~\cite{Mukherjee:TACSS} in the complete Segal space model structure for bisimplicial sets.
\end{blanko}

\newpage

\appendix

\section{Relative complete maps and Rezk completion}

\begin{center}
\textit{By Philip Hackney, Joachim Kock, and Jan Steinebrunner}
\end{center}

\bigskip

We prove the following:
{
\renewcommand{\thelemma}{\ref{prop rfib equiv X LX}}
\begin{proposition}
For $X$ a Segal space, we have $\kat{RFib}(X)\simeq \kat{RFib}(LX)$,
\end{proposition}
\addtocounter{lemma}{-1}
}

\noindent
where $LX$ is the Rezk completion of $X$ (previously denoted 
$\widehat{X}$).
This result should be attributed to Boavida, who proved
it in the setting of model categories~\cite{Boavida:SOGC}. Our proof is synthetic
and a bit more conceptual, deriving the result from the 
following:

{
\renewcommand{\thelemma}{\ref{prop semi left exact}}
\begin{proposition}
Rezk completion is a semi-left-exact localization.
\end{proposition}
\addtocounter{lemma}{-1}
}

This is of independent interest. For example, it readily implies
that there is a factorization system consisting of the
Dwyer--Kan equivalences and the relative Rezk complete maps (\cref{prop fact system}).
We will also use this to show that relative complete Segal spaces over a Segal space $X$ correspond to complete Segal spaces over $LX$:

{
\renewcommand{\thelemma}{\ref{prop rc maps equiv}}
\begin{proposition}
For any Segal space $X$:
    $\kat{Seg}^{\operatorname{rc}}_{/X} \simeq \kat{CSS}_{/LX} $.\end{proposition}
\addtocounter{lemma}{-1}
}

Before coming to these results, we need to set up some terminology, notation, and a few preliminary results.

\bigskip

Let $E(1)$ denote the {\em strict} nerve of the contractible groupoid with two 
objects. 
By \cite[Theorem 6.2]{Rezk:MHTHT} the space of equivalences 
$X_1^{\operatorname{eq}} \subset X_1$ of a Segal space $X$ is equivalent to 
$$
X_1^{\operatorname{eq}} \simeq \map(E(1), X) .
$$
Recall (from \cite[\S6]{Rezk:MHTHT}) that a Segal space $X$ is called {\em (Rezk) complete}
when either (and hence both) of the maps
$$
X_0 \overset{d_1}{\underset{d_0}\leftleftarrows} X_1^{\operatorname{eq}}
$$
is an equivalence (this is equivalent to the definition from \ref{def:Rezk}). 

The full inclusion $\kat{CSS} \hookrightarrow \kat{Seg}$ of complete 
Segal spaces into all Segal spaces has a left adjoint (reflection)
$$
L \colon \kat{Seg} \to \kat{CSS} .
$$
An explicit formula was given by Rezk~\cite[\S14]{Rezk:MHTHT} (see also 
\cite[Proposition 2.6]{AyalaFrancis:FHC} for a model independent account).
We shall not need the explicit formula. What we do need is \cref{rezk 7.7} (due to Rezk) characterizing the class of maps inverted by $L$
as the Dwyer--Kan equivalences, as we now recall.

Recall that any Segal space $X$ has mapping spaces
\[ \begin{tikzcd}
\map_X(x,x') \rar \dar \ar[dr, phantom, "\lrcorner" very near start]  & X_1 \dar{(d_1,d_0)} \\
\ast \rar[swap]{x,x'} & X_0 \times X_0 .
\end{tikzcd} \]
There is an associated \emph{homotopy category} $\ho (X)$, with set of objects $\pi_0(X_0)$ and $\hom([x],[x']) = \pi_0(\map_X(x,x'))$.
See \cite[\S5]{Rezk:MHTHT}.

\begin{definition}[{\cite[7.4]{Rezk:MHTHT}}]
A map $f\colon Y \to X$ between Segal spaces is a \emph{Dwyer--Kan equivalence} if
\begin{enumerate}
    \item $f$ is \emph{essentially surjective}, that is, the induced map $\ho (f) \colon \ho (Y) \to \ho (X)$ is essentially surjective, and \label{DK def ho equiv}
    \item $f$ is \emph{fully faithful}, that is, for each $y,y' \in Y$, the induced map on mapping spaces
    \[
        \map_Y (y,y') \to \map_X(fy, fy')
    \]
    is an equivalence. \label{DK def ff}
\end{enumerate}
\end{definition}

Note that being fully faithful is equivalent to the assertion that the square
\[ \begin{tikzcd}
Y_n \rar \dar & X_n \dar \\
Y_0^{\times n+1} \rar & X_0^{\times n+1}
\end{tikzcd} \]
is a pullback for $n=1$ or equivalently for all $n$.

\begin{theorem}[{\cite[Theorem 7.7]{Rezk:MHTHT}}]\label{rezk 7.7}
    The Dwyer--Kan equivalences between Segal spaces are precisely the maps that
    are inverted by the completion functor $L \colon \kat{Seg} \to \kat{CSS}$.
\end{theorem}

We now introduce the notion of a relative complete map between Segal spaces, 
which is a variation on what was called a \emph{fiberwise complete Segal space} in \cite[2.2]{BoavidaWeiss:SSECC} and \cite[\S1.4]{Boavida:SOGC}.
(A good notion for maps between general simplicial spaces would utilize arbitrary maps $E(n) \to E(m)$.)

\begin{definition}\label{def rel complete segal}
Suppose $Y$ and $X$ are Segal spaces.
A map $Y \to X$ is \emph{relative complete} if
it is right orthogonal to both morphisms $E(0) \rightrightarrows E(1)$.
\end{definition}

Since there is an automorphism of $E(1)$ that permutes the two morphisms $E(0) \to E(1)$ it suffices to check the orthogonality against only one of them.
Spelling this out we see that $Y \to X$ is relative complete if and only if the square
\[ \begin{tikzcd}
Y_1^{\operatorname{eq}} \rar \dar[swap]{d_i} & X_1^{\operatorname{eq}}  \dar{d_i} \\
Y_0 \rar & X_0 
\end{tikzcd} \]
is a pullback for $i=0$ or $i=1$.

We begin by recording some properties of relative complete maps that are also fully faithful or essentially surjective.

\begin{lemma}\label{lem mono}
If $Y \to X$ is a map between Segal spaces which is relative complete and fully faithful, then it is levelwise a monomorphism (of $\infty$-groupoids).
\end{lemma}
\begin{proof}
We first show that $Y_0 \to X_0$ is a monomorphism.
Relative completeness implies that the square
\[ \begin{tikzcd}
Y_0 \rar \dar \ar[dr, phantom, "\lrcorner" very near start] & Y^{E(1)} \dar \\
X_0 \rar & X^{E(1)}
\end{tikzcd} \]
is a pullback of spaces.
We thus have the composite pullback
\[ \begin{tikzcd}
Y_0 \rar \dar \ar[dr, phantom, "\lrcorner" very near start] & 
Y_1^{\mathrm{eq}} \dar \ar[r, hook] \ar[dr, phantom, "\lrcorner" very near start] &
Y_1 \rar \dar \ar[dr, phantom, "\lrcorner" very near start] &
Y_0^{\times 2} \dar \\
X_0 \rar &
X_1^{\mathrm{eq}} \ar[r, hook] &
X_1 \rar & 
X_0^{\times 2} ,
\end{tikzcd} \]
where the right pullback square says that $Y \to X$ is fully
faithful and the middle square uses conservativity of $\ho (Y) \to
\ho (X)$. The fiber of the diagonal $Y_0 \to
Y_0^{\times 2}$ at $(y,y')$ is the space of paths from $y$ to $y'$,
so this being a pullback shows that $\operatorname{Path}_{Y_0}(y,y')
\to \operatorname{Path}_{X_0}(f_0 y,f_0 y')$ is an equivalence for
all $y,y'$. This implies $f_0$ is a monomorphism ($\pi_0$-injective and an
equivalence of spaces on each path component).

Now the right square in the above diagram is a pullback, so $f_1$ is a monomorphism. Since $Y$ and $X$ are Segal it follows that $f_n \colon Y_n \to X_n$ is a monomorphism for all $n$.
\end{proof}

\begin{lemma}\label{map to LX is surjective}
  For each Segal space $X$ and $n\in \N$, the completion map
  $$
  (\alpha_X)_n \colon X_n \to (LX)_n
  $$
  is $\pi_0$-surjective.
\end{lemma}
\begin{proof}
    We begin with the case $n=0$.
    Since $X \to LX$ is a Dwyer--Kan equivalence it is essentially surjective. 
    This means that every point in $LX_0$ is isomorphic in $LX$ to a point in the image of $X_0 \to LX_0$.
    But $LX$ is complete and hence isomorphic objects are isotopic, therefore $X_0 \to LX_0$ is $\pi_0$-surjective.
    For general $n$ we can use fully faithfulness to write $X_n \to LX_n$ as the base-change of $X_0^{\times n+1} \to LX_0^{\times n+1}$, which is $\pi_0$-surjective by the first part of the proof.
\end{proof}

\begin{lemma}\label{lem epi}
If $Y \to X$ is a map between Segal spaces which is
relative complete and essentially surjective,
then $\pi_0 (Y_0) \to \pi_0 (X_0)$ is surjective.
\end{lemma}
\begin{proof}
Given $x \in X_0$, since $\ho (Y) \to \ho (X)$ is 
essentially surjective there exists $e\colon E(1) \to X$ so that 
the following diagrams commute for some $y\in Y_0$.
\[ \begin{tikzcd}
\Delta^0 \dar[swap]{1}  \rar{y} & Y \dar \\
E(1) \rar{e} & X  .  \\
\Delta^0 \uar{0} \ar[ur, bend right, "x"']
\end{tikzcd} \]
Since $Y\to X$ is relative complete, a unique lift exists in the top square.
\end{proof}

\begin{proposition}\label{prop equiv Segal}
A map $Y \to X$ between Segal spaces is a levelwise equivalence if and only if it is a Dwyer--Kan equivalence and relative complete.
\end{proposition}
\begin{proof}
  By \cref{lem mono}, $Y_0 \to X_0$ is a monomorphism of $\infty$-groupoids, and 
  by \cref{lem epi} it is also surjective on $\pi_0$, so altogether
$Y_0\to X_0$ is an equivalence. Since 
\[ \begin{tikzcd}
Y_1 \rar \dar & X_1 \dar \\
Y_0^{\times 2} \rar & X_0^{\times 2}
\end{tikzcd} \]
is a pullback, we see that also $Y_1\to X_1$ is an equivalence, and since $Y$ and $X$ are 
Segal, all the higher $Y_n \to X_n$ become equivalences too.
\end{proof}

\begin{lemma}\label{lem DK pullback}
Dwyer--Kan equivalences that are $\pi_0$-surjective
in simplicial degree zero are stable under pullback.
\end{lemma}
By \cref{map to LX is surjective} this lemma applies, in particular, to Dwyer--Kan equivalences whose codomain is a complete Segal space.
\begin{proof}
Suppose $B \to A$ is such a Dwyer--Kan equivalence, and consider a pullback diagram
\[ \begin{tikzcd}
Y \rar \dar \ar[dr, phantom, "\lrcorner" very near start] & X \dar \\
B \rar & A.
\end{tikzcd} \]
This yields a cube of spaces, four of whose faces are given below:
\[ \begin{tikzcd}
Y_1 \rar \dar \ar[dr, phantom, "\lrcorner" very near start] & X_1 \dar & 
Y_1 \rar \dar  & X_1 \dar
 \\
B_1 \rar \dar \ar[dr, phantom, "\lrcorner" very near start] & A_1 \dar &
Y_0^{\times 2} \rar \dar \ar[dr, phantom, "\lrcorner" very near start] & X_0^{\times 2} \dar
 \\
B_0^{\times 2} \rar & A_0^{\times 2} &
B_0^{\times 2} \rar & A_0^{\times 2}  .
\end{tikzcd} \]
The bottom left square is a pullback since $B \to A$ is fully faithful; it follows that the top right square is a pullback as well.
Hence $Y \to X$ is fully faithful.

Now if $\pi_0(B_0) \to \pi_0(A_0)$ is surjective, the same is true of $\pi_0(Y_0) \to \pi_0(X_0)$.
In particular, $\ho (Y) \to \ho (X)$ is surjective on objects, hence essentially surjective.
\end{proof}

Let $L$ be the Rezk completion functor on Segal spaces, and $\alpha \colon \id \Rightarrow L$ the completion natural transformation.

\begin{lemma}\label{rel complete cartesian}
The natural transformation $\alpha$ is cartesian on relative complete maps.
\end{lemma}
\begin{proof}
Consider the diagram
\[ \begin{tikzcd}
Y \ar[ddr, bend right,"f"'] \ar[dr] \ar[drr,bend left,"\alpha_Y"] &[-0.5cm] \\[-0.5cm]
& Q \rar \ar[dr, phantom, "\lrcorner" very near start] \dar & LY \dar["Lf"] \\
& X \ar[r, "\alpha_X"'] & LX  .
\end{tikzcd} \]
Since $\alpha_X$ is a Dwyer--Kan equivalence, so too is $Q \to LY$ by \cref{lem DK pullback}.
Then since $\alpha_Y$ is a Dwyer--Kan equivalence we have by 2-of-3 (see \cite[Lemma 7.5]{Rezk:MHTHT}) that $Y\to Q$ is as well.

But now $Lf$ is relative complete since it's a map between complete Segal spaces, 
and so too is $Q \to X$, since
relative complete maps are stable under pullback. The map $f$ was assumed relative complete, so $Y \to Q$ is also relative complete by left cancellation.
\Cref{prop equiv Segal} implies that $Y\to Q$ is an equivalence.
\end{proof}

The proof of the preceding lemma really shows that any square whose vertical edges are relative complete and whose horizontal edges are Dwyer--Kan equivalences must in fact be a pullback.

\begin{corollary}\label{cor new char rel compl}
The following classes of maps between Segal spaces coincide:
\begin{enumerate}
\item relative complete maps, \label{en eq rel compl 1}
\item those maps on which the natural transformation $\alpha$ is cartesian, and\label{en eq rel compl 2}
\item pullbacks of maps between complete Segal spaces. \label{en eq rel compl 3}
\end{enumerate}
\end{corollary}
\begin{proof}
Every map between complete Segal spaces is relative complete, and pullbacks of relative complete maps are relative complete, so class \eqref{en eq rel compl 3} is contained in class \eqref{en eq rel compl 1}.
\Cref{rel complete cartesian} states that \eqref{en eq rel compl 1} is contained in \eqref{en eq rel compl 2}, and the final containment \eqref{en eq rel compl 2} $\subseteq$ \eqref{en eq rel compl 3} is immediate.
\end{proof}

\begin{definition}
  Suppose $\CC \hookrightarrow \DD$ is a reflective subcategory of a
  finitely-complete $\infty$-category $\DD$, with left adjoint $L\colon
  \DD \to \CC$. We say that $L$ is \emph{semi-left-exact} if, for
  all pullback squares of $\DD$ below-left with $S, T \in \CC$,
\[ \begin{tikzcd}
Y \dar \rar \ar[dr, phantom, "\lrcorner" very near start] & T \dar 
& %RIGHT TOP
LY \dar \rar \ar[dr, phantom, "\lrcorner" very near start] & T \dar 
\\ %LEFT BOTTOM
X \rar & S 
& %RIGHT BOTTOM
LX \rar & S
\end{tikzcd} \]
the square above-right is also a pullback (in $\CC$).
\end{definition}

This terminology agrees with that of Cassidy--H\'ebert--Kelly in the 1-cat\-e\-gorical case \cite[p.298]{CassidyHebertKelly:RSLFS}.
Gepner and Kock \cite[1.2]{GepnerKock:ULCCIC} also consider this condition in the case of locally cartesian closed $\infty$-categories, and call such left adjoints \emph{locally cartesian
localizations}.

\begin{proposition}\label{prop sle fact system}
  Suppose $L\colon \DD \to \CC$ is a semi-left-exact reflector (where $\DD$
  is finitely-complete), and $\alpha \colon \id \Rightarrow L$ is the unit
  of the reflection. Then there is a factorization system $(\mathcal{L},
  \mathcal{R})$ on $\DD$, where $\mathcal{L}$ are those maps inverted by
  $L$, and $\mathcal{R}$ is the class of maps on which the natural 
  transformation $\alpha$ is cartesian.
\end{proposition}
\begin{proof}
This is a straightforward generalization to 
$\infty$-categories of a classical theorem of Cassidy--H\'ebert--Kelly \cite[Theorem 4.3]{CassidyHebertKelly:RSLFS}. 
In the setting of $\infty$-categories it is essentially
Proposition 3.1.10 of \cite{ABFJ:LELI} except that they have 
``left-exact'' in place of ``semi-left-exact.''
Inspecting the second paragraph of their proof one sees that only the semi-left-exact condition is actually used.
\end{proof}

Recall the following (for example from \cite[Lemma 3.3]{Steinebrunner:LCFRFCSC}).

\begin{lemma}\label{lem wrong way pasting}
Consider a diagram of spaces
\[ \begin{tikzcd}
A \rar \dar \ar[dr, phantom, "\lrcorner" very near start] &
B \rar \dar &
C \dar \\
X \rar[swap, "f"] &
Y \rar &
Z ,
\end{tikzcd} \]
where the left square is a pullback, the composite rectangle is a pullback, and the map $f$ is $\pi_0$-surjective.
Then the right square is a pullback as well. \qed
\end{lemma}

\begin{proposition}\label{prop semi left exact}
The completion functor $L \colon \kat{Seg} \to \kat{CSS}$ is semi-left-exact.
\end{proposition}
\begin{proof}
Consider a pullback of Segal spaces
\[ \begin{tikzcd}
Y \dar \rar \ar[dr, phantom, "\lrcorner" very near start] & T \dar \\
X \rar & S
\end{tikzcd} \]
where $S$ and $T$ are complete Segal spaces.
Since $T\to S$ is automatically relative complete, it follows that $Y \to X$ 
is relative complete.
We then have the diagram
\[ \begin{tikzcd}
Y \rar{\alpha_Y} \dar \ar[dr, phantom, "\lrcorner" very near start] & LY \rar \dar & T \dar \\
X \rar[swap]{\alpha_X} & LX \rar & S ,
\end{tikzcd} \]
where the left square is a pullback by \cref{rel complete cartesian} and the outer square is a pullback by assumption.
Our goal is to show that the right-hand square is a pullback as well.
By \cref{map to LX is surjective} the map $X_n \to LX_n$ is surjective on path components.
Thus by \cref{lem wrong way pasting}, for each $n$ the right square in 
\[ \begin{tikzcd}
Y_n \rar \dar \ar[dr, phantom, "\lrcorner" very near start] & LY_n \rar \dar & T_n \dar \\
X_n \rar & LX_n \rar & S_n
\end{tikzcd} \]
is a pullback.
Thus $L(T\times_S X) \to T\times_S LX$ is an equivalence, as desired.
\end{proof}

\begin{proposition}\label{prop fact system}
Dwyer--Kan equivalences and relative complete maps constitute a factorization system on $\kat{Seg}$.
\end{proposition}
\begin{proof}
This is the factorization system guaranteed by \cref{prop sle fact system}, using that the localization $L \colon\kat{Seg} \to \kat{CSS}$ is semi-left-exact (\cref{prop semi left exact}).
The Dwyer--Kan equivalences are precisely those maps between Segal spaces which are inverted by $L$.
In \cref{cor new char rel compl}, we identified the relative complete maps as the right class in this factorization system.
\end{proof}

\begin{proposition}\label{prop rc maps equiv}
Suppose $X$ is a Segal space.
Then $\alpha_X\upperstar \colon \kat{Seg}_{/LX} \to \kat{Seg}_{/X}$ restricts to an equivalence $\kat{CSS}_{/LX} \to \kat{Seg}^{\operatorname{rc}}_{/X}$, into the full subcategory of the relative complete maps with codomain $X$. The inverse is given by the completion $L$.
\end{proposition}
\begin{proof}
Recall (from \ref{slicing adj} above and \cite[5.2.5.1]{HTT}) that the sliced
localization functor $L_X \colon \kat{Seg}_{/X} \to \kat{CSS}_{/LX}$
has right adjoint
\begin{equation}\label{eq css to seg}
\begin{tikzcd}[sep=small]
    \kat{CSS}_{/LX} \rar[hook] & \kat{Seg}_{/LX} \rar{\alpha_X\upperstar} & \kat{Seg}_{/X}
\end{tikzcd}
\end{equation}
with counit (at an object $A\to LX$) given by
\[ \begin{tikzcd}
L(X \times_{LX} A) \rar{L(\operatorname{pr}_2)} & LA \rar{\varepsilon_{A}} & A  .
\end{tikzcd} \]
The statement is that this counit is invertible.
But since $L$ is semi-left-exact, the first map is just the projection $L(X\times_{LX} A) \simeq  LX \times_{LX}  LA \to LA$
which is an equivalence, and the second map $\varepsilon_A$ is an equivalence since the inclusion $\kat{CSS} \subset \kat{Seg}$ is full.
(This argument is borrowed from the proof of \cite[Lemma 1.7]{GepnerKock:ULCCIC}, which however unnecessarily assumes the
ambient $\infty$-category to be locally cartesian closed.)
\Cref{cor new char rel compl} identifies the image of  \eqref{eq css to seg} as the relative complete maps.
\end{proof}

\begin{remark}
    In their work on configuration categories, Boavida and Weiss~\cite[Appendix B]{BoavidaWeiss:SSECC} introduce a model category of relative complete Segal spaces over a Segal space $X$. 
    \Cref{prop rc maps equiv} shows that the $\infty$-category underlying this model category is equivalent to the slice $\infty$-category  $\catinf_{/LX}$.
\end{remark}

\begin{corollary}\label{rc over DK equivalence}
If $f\colon Y\to X$ is a Dwyer--Kan equivalence between Segal spaces, then $f\upperstar \colon \kat{Seg}^{\operatorname{rc}}_{/X} \to \kat{Seg}^{\operatorname{rc}}_{/Y}$ is an equivalence.
\end{corollary}
\begin{proof}
Since $Lf \colon LY \to LX$ is an equivalence of simplicial spaces, 
we have the indicated equivalences in the commutative square
\[ \begin{tikzcd}
\kat{CSS}_{/LX} \rar["(Lf)\upperstar","\simeq"'] \dar["\alpha_X\upperstar"',"
\simeq"] & \kat{CSS}_{/LY} \dar["\alpha_Y\upperstar","
\simeq"'] \\
\kat{Seg}^{\operatorname{rc}}_{/X} \rar[swap]{f\upperstar} & 
\kat{Seg}^{\operatorname{rc}}_{/Y} .
\end{tikzcd} \]
It follows that $f\upperstar$ is an equivalence as well.
\end{proof}

The proof of \cref{prop rc maps equiv} and its corollary
did not use anything special about our situation, other than $L\colon \kat{Seg} \to \kat{CSS}$ being semi-left-exact.
We conclude that the following proposition holds (and the case $Y 
\to LY$ recovers the statement of \cref{prop rc maps equiv}).

\begin{proposition}
Suppose $L\colon \DD \to \CC$ is a semi-left-exact reflector and 
$(\mathcal{L}, \mathcal{R})$ is the factorization system on $\DD$ from \cref{prop sle fact system}.
If $f\colon Y\to X$ is in $\mathcal{L}$, then $f\upperstar \colon \mathcal{R}_{/X} \to \mathcal{R}_{/Y}$ is an equivalence. \qed
\end{proposition}

\bigskip
We now examine the implications of these results to our maps of interest, right fibrations.

\begin{lemma}\label{rf rel comp}
Right fibrations between Segal spaces are relative complete.
\end{lemma}
\begin{proof}
It is enough to observe that $1 \colon \Delta^0 \to E(1)$ is a final map 
(by \cref{lem: terminal objects}) since it preserves the last vertex. 
\end{proof}

\begin{proposition}\label{L preserves rfib}
If $Y \to X$ is a right fibration between Segal spaces, then $LY \to LX$ is also a right fibration.
\end{proposition}
\begin{proof}
By \cref{rel complete cartesian} and \cref{rf rel comp}, the following 
naturality square is a pullback.
\[ \begin{tikzcd}
Y \rar \dar \ar[dr, phantom, "\lrcorner" very near start] & LY \dar \\
 X \rar & LX
\end{tikzcd} \]
It follows that the right square below is a pullback, while the left 
square is a pullback since $Y\to X$ is a right fibration.
\[ \begin{tikzcd}
Y_n \rar \dar \ar[dr, phantom, "\lrcorner" very near start]  & Y_0 \rar \dar \ar[dr, phantom, "\lrcorner" very near start] & LY_0 \dar \\
X_n \rar & X_0 \rar & LX_0
\end{tikzcd} \]
We then have the left and outer squares of the following are pullbacks.
\begin{equation}\label{wrong way pasting law} \begin{tikzcd}
Y_n \rar \dar \ar[dr, phantom, "\lrcorner" very near start]  & LY_n \rar \dar & LY_0 \dar \\
X_n \rar & LX_n \rar & LX_0
\end{tikzcd}
\end{equation}
We wish to deduce that the right hand square is a pullback.
But by \cref{map to LX is surjective} $X_n \to LX_n$ is surjective on path components, so by \cref{lem wrong way pasting} the right square in \eqref{wrong way pasting law} is a pullback.
\end{proof}

The following equivalence is a restriction of that of \cref{prop rc maps equiv}.
It more or less recovers \cite[Corollary 5.6]{Boavida:SOGC}.

\begin{proposition}\label{prop rfib equiv X LX}
Suppose $X$ is a Segal space.
Then $\alpha_X\upperstar \colon \kat{RFib}_{/LX} \to \kat{RFib}_{/X}$ is an equivalence, with inverse given by the completion $L$.
\end{proposition}
\begin{proof}
The fully-faithful functor $\kat{RFib}_{/X} \to \kat{Seg}_{/X}$ factors through $\kat{Seg}^{\operatorname{rc}}_{/X}$ by \cref{rf rel comp}.
Any right fibration whose codomain is a complete Segal space also has a complete Segal space as its domain. 
By \cref{L preserves rfib} we have the following lift of $L$.
\[ \begin{tikzcd}
\kat{RFib}_{/X} \rar[dashed, "L"] \dar["\text{f.f.}"] & \kat{RFib}_{/LX} \dar["\text{f.f.}"] \\
\kat{Seg}^{\operatorname{rc}}_{/X} \rar["L"', "\simeq"] & \kat{CSS}_{/LX}
\end{tikzcd} \]
Of course $\alpha_X\upperstar$ restricts to $\kat{RFib}_{/LX} \to \kat{RFib}_{/X}$, and since the vertical maps are fully faithful, this implies the result.
\end{proof}

The following more or less recovers \cite[Proposition 5.5]{Boavida:SOGC}.

\begin{corollary}
If $f\colon Y \to X$ is a Dwyer--Kan equivalence between Segal spaces, then $f\upperstar \colon \kat{RFib}_{/X} \to \kat{RFib}_{/Y}$ is an equivalence.
\qed
\end{corollary}

\begin{remark}\label{rem: flagged}
As a special case of \cite[Theorem 0.27]{AyalaFrancis:FHC}, $\kat{Seg}$ is equivalent to the full subcategory of the arrow $\infty$-category of $\catinf$ on the essentially surjective functors $\DD \to \CC$ where $\DD$ is an $\infty$-groupoid.
Under this equivalence, completion $\kat{Seg} \to \kat{CSS} \simeq \catinf$ may be identified with the codomain map $(\DD \to \CC) \mapsto \CC$.
It may be possible to give alternative proofs of some statements in this appendix using their results.
\end{remark}

\renewcommand{\bibliofont}{\normalsize}

\bibliographystyle{scplain}
\bibliography{2-segal}

\end{document}